\documentclass{amsart}

\usepackage{amsmath,amssymb, amsthm,epsfig}
\usepackage[usenames,dvipsnames]{xcolor}
\usepackage{hyperref} 
\usepackage{verbatim} 
\usepackage{color}     
\usepackage{graphicx}
\usepackage{mathrsfs}
\usepackage{tikz,tikz-cd}

\usepackage{lineno}

\newcommand\R{\mathbb R}
\newcommand\N{\mathbb N}
\newcommand\C{\mathbb C}
\newcommand\D{\mathbb D}

\newcommand\CC{{\mathcal C}}
\newcommand\T{{\mathcal T}}
\newcommand\V{V}
\newcommand\E{E}
\newcommand\F{F}

\newcommand\Julia{\mathcal{J}}

\newcommand\G{G}
\newcommand\g{g}

\newcommand\sphere{\mathbb S^2}
\newcommand\rs{\widehat{\mathbb{C}}}

\newcommand\QS{\text{QS}}
\newcommand\homeo{\text{Homeo}(\Lambda_H)}
\newcommand\homeoplus{\text{Homeo}^+(\Lambda_H)}
\DeclareMathOperator{\Int}{int}

\newtheorem{theorem}{Theorem}[section]
\newtheorem{proposition}[theorem]{Proposition}
\newtheorem{lemma}[theorem]{Lemma}
\newtheorem{corollary}[theorem]{Corollary}

\theoremstyle{remark}        

\newtheorem{definition}[theorem]{Definition}
\newtheorem{remark}[theorem]{Remark}

\begin{document}

\title[Dynamical gaskets]{On dynamical gaskets generated by\\ rational maps, Kleinian groups,\\ and Schwarz reflections}

\begin{author}[R.~Lodge]{Russell Lodge}
\address{Department of Mathematics and Computer Science, Indiana State University, Terre Haute, IN 47809, USA}
\email{russell.lodge@indstate.edu}
\end{author}

\begin{author}[M.~Lyubich]{Mikhail Lyubich}
\address{Mathematics Department and Institute for Mathematical Sciences, 100 Nicolls Rd, Stony Brook, NY 11794, USA}
\email{mlyubich@math.stonybrook.edu}
\end{author}

\begin{author}[S.~Merenkov]{Sergei Merenkov}
\address{Department of Mathematics, City College of New York, New York, NY 10031; and Mathematics Program, CUNY Graduate Center, New York, NY 10016}
\email{smerenkov@ccny.cuny.edu}
\end{author}

\begin{author}[S.~Mukherjee]{Sabyasachi Mukherjee}
\address{School of Mathematics, Tata Institute of Fundamental Research, 1 Homi Bhabha Road, Mumbai 400005, India}
\email{sabya@math.tifr.res.in}

\end{author}

\thanks{M.~Lyubich was supported by NSF grants DMS-1600519 and 1901357, and by a Fellowship
from the Hagler Institute for Advanced Study.\\
\indent S.~Merenkov was supported by NSF grant DMS-1800180.\\
\indent S.~Mukherjee was supported by the Department of Atomic Energy, Government of India, under project no.12-R\&D-TFR-5.01-0500, an endowment of the Infosys Foundation, and SERB research project grant  SRG/2020/000018.}

\date{\today}
\begin{abstract}
  According to the Circle Packing Theorem,  any triangulation of the Riemann sphere can be realized as a nerve of
  a circle packing. Reflections in the dual circles generate a Kleinian group $H$ whose limit set is
  a generalized Apollonian gasket $\Lambda_H$.  We design a surgery that relates $H$
  to a rational map $g$ whose Julia set $\Julia_g$ is (non-quasiconformally)  homeomorphic to $\Lambda_H$.
  We show for a large class of triangulations, however, the groups of quasisymmetries of $\Lambda_H$ and $\Julia_g$ are  isomorphic
  and coincide with the corresponding groups of self-homeomorphisms. Moreover, in the case of $H$,
  this group is equal to the group of M\"obius symmetries of $\Lambda_H$, which is the semi-direct product of $H$
  itself and the group of M\"obius symmetries of the underlying circle packing.  
  In the case of the tetrahedral triangulation (when $\Lambda_ H$ is the classical Apollonian gasket),
  we give a quasiregular model for the above actions which is quasiconformally equivalent to $g$
  and  produces $H$  by a David surgery.
  We also construct a mating between the group
  and the map coexisting in the same dynamical plane and show that it can be generated by Schwarz reflections
  in the deltoid and the inscribed circle.
 \end{abstract}

\maketitle

\setcounter{tocdepth}{1}

\tableofcontents

\section{Introduction}\label{sec:intro}

In this paper we will further explore the celebrated Fatou-Sullivan Dictionary connecting two branches of Conformal 
Dynamics, iteration of rational maps and actions of Kleinian groups. This dictionary is an indispensable source of 
new notions, conjectures, and  arguments, but it does not provide an explicit  common frame for the two areas.
However, in the 1990's Bullet and Penrose \cite{BP} discovered a phenomenon of explicit mating of two actions,
of a quadratic polynomial with a modular group, induced by a single algebraic correspondence on two
parts of its domain. And  recently, an abundant supply of similar matings generated by the Schwarz reflection dynamics 
was produced by Lee and Makarov in collaboration with  two of the  authors of this paper \cite{LLMM1,LLMM2,LLMM3}.   
It turns out that this machinery is relevant to the theme of this paper

Our main example is the classical Apollonian gasket $\Lambda_H$, which is the limit set of a Kleinian reflection group $H$ generated by reflections in four
pairwise kissing circles, see~Figure~\ref{F:Apollonian}. 
In this paper we demonstrate  that this limit set can be topologically realized as the Julia set  $\Julia(g)$ of a hyperbolic 
rational function \cite[\S 4]{BEKP}. In fact,  we construct $g$ in two different ways: by applying the Thurston Realization Theorem
and by constructing an explicit quasiregular model for $g$. (The subtlety of the problem has to do with the fact that $g$ is hyperbolic while $H$ is parabolic, so $\Lambda_H$ and $\Julia(g)$ are not quasiconformally equivalent.) Moreover, we show that  $H$ and $g$ can be  mated by means of the Schwarz reflection in the deltoid and an inscribed circle to produce a hybrid dynamical system alluded above. 
This mating is based upon a surgery replacing   the action of $\bar z^2$  in the disk by
the modular group action, using the classical Minkowski ``question mark function".  This surgery is not quasiconformal,
but it has David regularity. We show this by  direct geometric estimates through the  Farey algorithm.
(Note that a David relation between hyperbolic and parabolic dynamics appeared first in Ha\"issinski's work, see \cite{BF14}.)

Our motivating problem was a problem of 
quasisymmetric classification of fractals,
which attracted a good deal of attention in recent years; see, e.g., \cite{BK02, BM12, BLM, HP12, Me, kW08}. A basic quasiconformal invariant of a fractal $\Julia$ is the group $\QS(\Julia)$ of  its quasisymmetries (``quasisymmetric Galois group"). 
A natural class of fractals to test this problem is the class of Julia sets and limit sets of Kleinian groups. In papers \cite{BLM,LM}, the group $\QS(\Julia)$ was studied for a class of Sierpi\'nski carpet Julia sets and for the Basilica, yielding strikingly different rigidity/flexibility behavior. In this paper we describe $\QS(\Julia)$  for gasket Julia gaskets, exhibiting yet another phenomenon.

Namely,  we prove that $\QS(\Julia(g))$ is a countable group isomorphic to the extension  of $H$ by the tetrahedron symmetry group
(which is the full group of M\"obius symmetries of $\Lambda_H$). Moreover, $\QS(\Julia(g))$  coincides with the group $\textrm{Homeo}(\Julia(g))$ of all orientation preserving self-homeomorphisms of $\Julia(g)$ and $\Lambda_H$. 
It is quite different from the cases studied earlier:

\smallskip 
-- In the Sierpi\'nski carpet case~\cite{BLM},  $\textrm{Homeo}(\Julia)$ is uncountably infinite, while $\QS(\Julia)$ is finite 
  and coincides with the group of M\"obius symmetries of $\Julia$. This is a quasisymmetrically rigid case. 

\smallskip
-- In the Basilica case~\cite{LM}, the  topological and quasisymmetry groups are different but both are uncountably infinite. Moreover,
 they have the same countable ``core" which is an index two extension of the Thompson circle group.
 (The groups are obtained by taking the closures of the core in appropriate topologies.)

\smallskip
Going back to the Apollonian case, we see that though the group $\QS$ does  not quasiconformally distinguish
the Julia set from the corresponding Apollonian limit set, these sets are not quasiconformally equivalent  (as we have already pointed out).
In fact, {\em we do not know a single non-trivial (i.e., different from a quasi-circle) example of a Julia set which is quasiconformally equivalent to 
a limit set of a Kleinian group}.

\begin{figure}[h]
\centerline{\includegraphics[width=50mm]{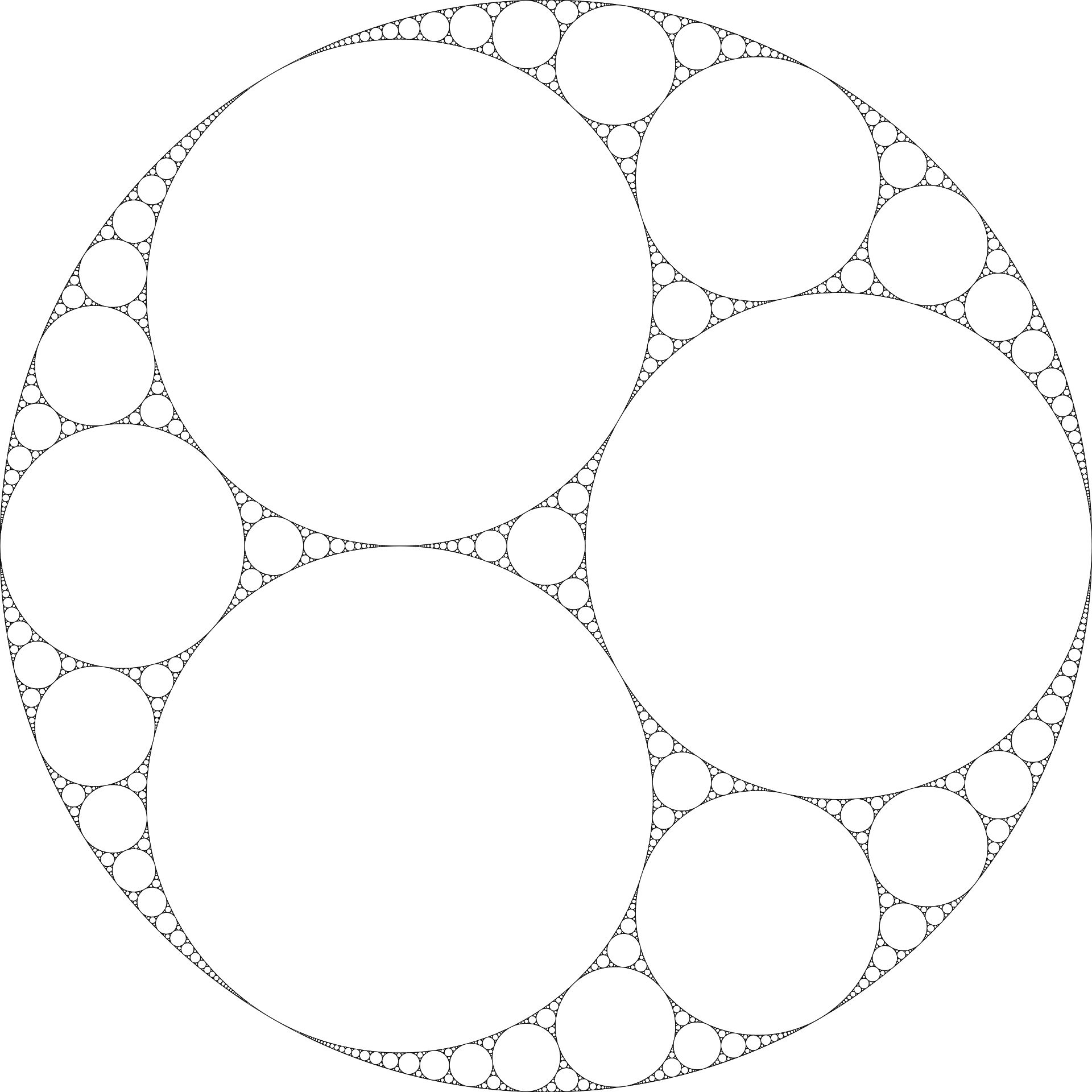}}
\caption{Classical Apollonian gasket.}
\label{F:Apollonian}
\end{figure}

\subsection{The outline}

We carry out the discussion for a family of Kleinian groups generalizing the classical Apollonian gasket. Namely, given an arbitrary triangulation of the sphere,  by the Circle Packing theorem it can be realized as the adjacency graph of some circle packing, unique up to M\"obius transformations. Consider the dual circle packing comprising the circles passing through tangency points
 of various triples of kissing circles. (The original circles are associated to the vertices of the triangulation, while the dual ones are associated to the faces.) The Kleinian reflection group generated by all reflections in the dual circles is our {\em (generalized)  Apollonian group} and its limit set is the {\em (generalized) Apollonian gasket}.\footnote{Below, we will often skip the adjective ``generalized"} Note that it is a {\em cusp group}: all components of its domain of discontinuity are round disks, and the corresponding quotient Riemann surfaces are punctured spheres. The classical Apollonian gasket corresponds to the tetrahedral triangulation, and the associated Kleinian group is a \emph{maximal cusp group}. Section~\ref{sec:TrianglesToGaskets} details this  construction.

In Section~\ref{sec:RoundGasketSym} we prove that, in the case when a triangulation is irreducible in the sense that any triangle (i.e., a $1$-cycle composed of three edges) bounds a face, every topological symmetry of an Apollonian gasket can be written as a composition of finitely many anti-conformal reflections as above and a M\"obius symmetry of the circle packing. Moreover, this  group splits into a semi-direct product of the above (Theorem~\ref{thm:groupSplits}). We conclude that  the M\"obius, topological, and quasisymmetry groups of the Apollonian gasket are all the same.

Section~\ref{group_equiv_map_sec} is devoted to the construction of a piecewise anti-M\"obius map $N$ on the Riemann sphere cooked up from the generators of the Apollonian group. This map, which we call the \emph{Nielsen map}, is orbit equivalent to the Apollonian group, and enjoys Markov properties when restricted to the limit set.

In Section~\ref{Sec:RoundGasketsNoninvertible} we carry out a surgery that turns the Nielsen map to an orientation reversing branched cover $\mathcal{G}$ that coincides with the Nielsen map on the complement of the circle packing and is topologically equivalent to $\D \rightarrow \D$,  $z\mapsto \overline{z}^k$, on each disk of the packing (with $k$ depending on the disk). By  construction, its Julia set coincides with the Apollonian gasket, on which it agrees with the Nielsen map.

In Section~\ref{sec:GasketJulia} we use W. Thurston's  Realization Theory to show that the above map $\mathcal{G}$ is equivalent to an anti-rational map $g$ (Proposition~\ref{prop:NoObstructionNerveGamma}). We then apply a Pullback Argument in Theorem~\ref{Thm_IdentificationOfLimitAndJulia} to show that, in fact, the Julia set $\Julia(g)$ is homeomorphic to the limit set $\Lambda_H$, and $g\vert_{\Julia(g)}$ is topologically conjugate to $\mathcal{G}\vert_{\Lambda_H}$. This gives one more manifestation of the intimate connection between (anti-)rational dynamics and Kleinian (reflection) actions in the spirit of the Fatou-Sullivan dictionary (Corollary~\ref{prop:group_anti_rat_conjugate}).

In Section~\ref{sec:JuliaQuasisym} we establish our main result,  Theorem \ref{thm:homeosAsQuasisym}, by showing that each topological symmetry of the Julia set of the anti-rational map $g$ is induced by a piecewise dynamical homeomorphism, and that such homeomorphisms are in fact quasisymmetries. Therefore, a complete account of the quasisymmetry group of $\Julia(g) $ is  given. Let us emphasize once again that 
due to the presence of tangent circles in the round gasket such sets are not quasisymmetric to the corresponding Julia sets. Thus, the ``obvious" way of identifying the quasisymmetry groups fails.  

In Section~\ref{affine_model_sec} we describe an alternative construction of the cubic anti-rational map $g_\T$, corresponding to the tetrahedral triangulation $\T$, by producing a quasiregular (in fact, piecewise affine outside the critical Fatou components) model and applying the Measurable Riemann Mapping Theorem. 

In Section~\ref{sec:David} we develop a technique to produce matings between a rational map and the Nielsen map of the triangle reflection group using David surgery.

Finally in Section~\ref{mating_sec}, we apply the main result of Section~\ref{sec:David} on the cubic anti-rational map $g_\T$ (constructed in Section~\ref{affine_model_sec}) to recover the Nielsen map of the classical Apollonian gasket reflection group. Along the way, we construct a ``hybrid dynamical system'' that binds together the Nielsen map of the classical Apollonian reflection group and anti-rational map $g_\T$ on the same dynamical plane, and explicitly characterize this hybrid dynamical system as the \emph{Schwarz reflection map} with respect to a deltoid and an inscribed circle.

\subsection{Further developments}

The topological connection between generalized Apollonian gasket limit sets and gasket Julia sets of critically fixed anti-rational maps
  discovered in this paper has been further developed in several follow-up papers.
  In particular, it was generalized from triangulations to arbitrary polyhedral tilings in \cite{LLM},
  \footnote{Note that Kleinian groups associated with such tiling have also appeared in the work of Kontorovich and Nakamura \cite{KN}.}
  and then our construction of Nielsen maps was applied to produce a dynamical correspondence between limit sets of ``kissing reflection groups''
  and Julia sets of critically fixed anti-rational maps. It led, in particular,  to a full classification of anti-rational maps in terms of Tischler graphs. Such a classification was independently given by Lukas Geyer \cite{Geyer}
(without relating it to groups). Like in our paper,
the method of using certain invariant arcs to justify the absence of a Thurston obstruction played a key role (in the spirit of Pilgrim and Tan Lei \cite{PT}). 
 However, the quasisymmetry groups of the corresponding limit and Julia sets remain unknown in general. 

The David surgery machinery  introduced  in this paper has also been developed further.
  Namely, in \cite{LMMN}, the David Extension Lemma~\ref{L:DavidExt} has been generalized to a broader class of situations
  (by means of  dynamical  techniques instead of  number-theoretic tools used here),
  and the surgery techniques of Theorem~\ref{T:DavidS} and Proposition~\ref{Nielsen_recovery_prop}
  have been adapted to show that the homeomorphisms between Julia sets of critically fixed anti-rational maps and limit sets of kissing reflection groups
 extend to David homeomorphisms of the plane.
  The David surgery also played a key role in \cite{LMMN}  in the proof of a Combination Theorem
  for suitable Kleinian reflection groups and anti-holomorphic polynomials.

\subsection*{Acknowledgement} During the work on this project the third author held a vi\-si\-ting position at the Institute for Mathematical Sciences at Stony Brook University. He thanks this Institute for its constant hospitality. The first and last authors would like to acknowledge the support of the Institute for Mathematical Sciences at Stony Brook University during part of the work on this project. Thanks are due to Dimitrios Ntalampekos for providing the authors with useful references on quasiconformal removability. We would also like to thank the anonymous referee for careful reading of the paper and making useful comments that led to improvement of exposition.

\section{Round gaskets from triangulations}\label{sec:TrianglesToGaskets}
All graphs are assumed to be \emph{simple}, i.e., no edge connects a vertex to itself and there is at most one edge connecting any two vertices.  
A \emph{triangulation} $\T$ of $\sphere$ is a finite embedded graph that is maximal in the sense that the addition of one edge results in a graph that is no longer both embedded and simple.
Denote the sets of vertices, edges, and faces of some triangulation $\T$ of $\sphere$ by $\V_\T$, $\E_\T$, and $\F_\T$ respectively. By convention we assume that $|\V_\T|\geq 4$ to avoid degeneracies. We also assume that two faces share at most one edge. Two embedded graphs in $\sphere$  are said to be \emph{isotopic} if there is an orientation preserving homeomorphism of $\sphere$ sending the vertices and edges of one graph to the other.

A \emph{circle packing} $\mathcal C$ is a finite collection of closed geometric disks in $\rs$ with pairwise disjoint interiors whose union is connected. The \emph{nerve} of $\mathcal C$ is a finite embedded graph whose vertices correspond to disks, and two vertices are connected by an edge if and only if the corresponding disks are tangent. Up to isotopy, we may assume that the vertices of the nerve of a circle packing $\mathcal C$ are the spherical centers of the disks in $\mathcal C$ and the edges are the unique spherical geodesic segments connecting the corresponding centers through the point of tangency.
There is evidently a bijection between edges and points of tangency in the circle packing.   If $\T$ is a triangulation of $\rs$, by the Circle Packing theorem, also known as the Koebe--Andreev--Thurston theorem~\cite[Corollary~13.6.2]{T}, we may assume that after an isotopy $\T$ is a nerve of a circle packing, and this circle packing is unique up to M\"obius transformations. 
In what follows, we fix such a circle packing and denote it by $\mathcal C_\T$. Whenever necessary, we specify a particular normalization that $\mathcal C_\T$ satisfies.  

Each face $f\in \F_\T$ contains a unique complementary component of the circle packing $\mathcal C_\T$ called an \emph{interstice} associated to $f$, which we denote by $\Delta_f$. Every interstice is an open Jordan region bounded by three circular arcs that may only intersect at their endpoints. For a given interstice $\Delta_f$, the three boundary arcs lie in three circles in $\mathcal C_\T$, and denote by $D_1, D_2, D_3$  the corresponding open disks enclosed by these (oriented) circles. Let $C$ be the unique circle that passes through the three mutual tangency points of $D_1, D_2, D_3$. We say that such a circle $C$ \emph{corresponds to} $f$. (It can also be described as a spherical circle inscribed in the face $f$, and hence circumscribing the interstice $\Delta_f$.) In this way, the collection of all circles corresponding to faces is a circle packing orthogonal to $\mathcal C_\T$, and the nerve of this orthogonal packing is the planar dual of $\T$ (see Figure \ref{F:apollonianLimit}). Denote by $R_f$ the anti-M\"obius reflection with respect to this unique circle $C$. Observe that $R_f$ fixes the three points of mutual tangency of $D_1, D_2, D_3$  in the closure $\overline{\Delta_f}$ of $\Delta_f$, and 
$$
R_f(\overline{\Delta_f})=\left(\bigcup_{g\neq f}\overline{\Delta_g}\right)\cup\left(\bigcup_{D\neq D_1, D_2, D_3}\overline{D}\right),
$$ 
where $g$ ranges over the faces of $\T$, and $D$ over the open disks in $\mathcal C_\T$.

 Let $v$ be any vertex in $\T$ and let $C_v$ be the corresponding circle in the circle packing $\T$ centered at $v$. We denote the open disk in $\widehat{\C}$ enclosed by the oriented circle $C_v$ by $D_v$.

Let $H_\T$ be the group generated by all reflections $R_f$, i.e., 
$$
H_\T=\langle R_f, f\in \F_\T \rangle.
$$ 
For convenience we omit the subscript and simply write $H$ when $\T$ is understood. 
If $f_1, f_2, \dots, f_k$ is the full list of (distinct) faces of $\mathcal T$, the group $H$ is finitely presented with the presentation
\begin{equation}\label{E:HPres}
H=\langle R_{f_1}, R_{f_2},\dots, R_{f_k}\colon  R_{f_1}^2=R_{f_2}^2=\dots =R_{f_k}^2={\rm id}\rangle.
\end{equation}
The \emph{limit set} $\Lambda_H$ of $H$ is defined to be the minimal nonempty $H$-invariant compact subset of $\rs$, and the \emph{regular set} of $H$ is given by $\Omega_H:=\rs\setminus \Lambda_H$. It is easy to show that 

\begin{equation}\label{E:LambdaUnion}
\Lambda_H=\overline{\bigcup_{h\in H}\bigcup_{v\in\V_\T}h\cdot C_v}
\end{equation}
\begin{equation}\label{E:OmegaUnion}
 \Omega_H=\bigcup_{h\in H}\bigcup_{v\in\V_\T}h\cdot (D_v).
\end{equation}

\begin{definition}
A set of this form is called a \emph{round gasket} or an \emph{Apollonian gasket}.
A \emph{gasket} is defined to be any subset of the Riemann sphere that is homeomorphic to a round gasket.
\end{definition}

\begin{figure}[h]
\centerline{\includegraphics[width=100mm]{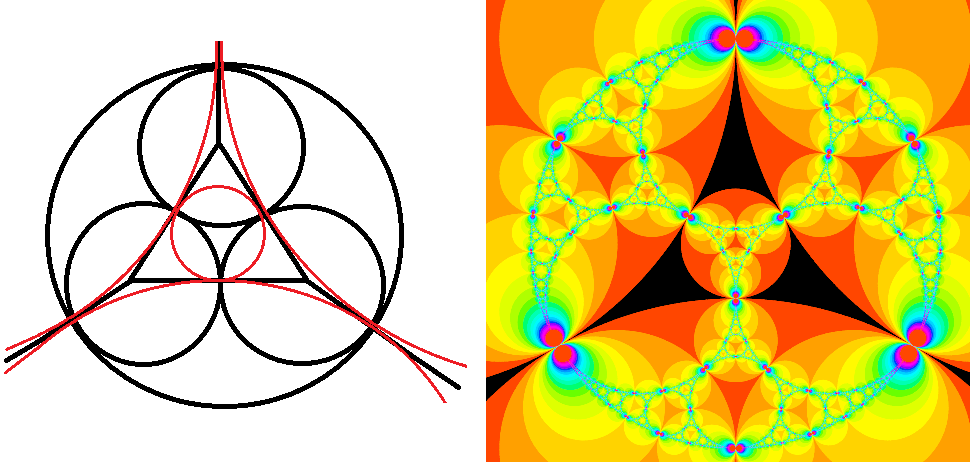}}
\caption{The tetrahedral graph $\T$ and its corresponding circle packing consisting of four circles appears on the left in black. The generators of the group $H$, which also happen to have a tetrahedral configuration, are depicted in red. The limit set $\Lambda_H$ appears on the right. 
 }
\label{F:apollonianLimit}
\end{figure}

\begin{remark}
Evidently any two peripheral topological disks (i.e., topological disks such that removal of their boundaries does not separate the gasket) in a gasket may touch at most at one point because the same property holds for round gaskets. For the same reason, at most two such disks may touch at the same point.
\end{remark}

The \emph{classical Apollonian gasket} in Figure~\ref{F:Apollonian} and Figure~\ref{F:apollonianLimit} is the limit set obtained using the construction above when $\T$ is the tetrahedral triangulation. Namely, $\T$ has four vertices and each pair of vertices is connected by an edge. The circle packing $\mathcal C_\T$ consists of four pairwise mutually tangent disks. The limit set $\Lambda_H$ in this case is the residual set obtained from $\rs$ by removing the interiors of the disks in the circle packing $\mathcal C_\T$, the largest open disk in each interstice, and in each resulting interstice, ad infinitum.

\section{Round gasket symmetries} \label{sec:RoundGasketSym}

This section describes properties of topological symmetries of round gaskets. We give an explicit description of the group of such symmetries for a large class of triangulations. In contrast to the Basilica Julia set, each topological symmetry of a round gasket is topologically extendable to a homeomorphism of the sphere.

\begin{lemma}\label{L:Extension}
Suppose that $\Lambda$ and $\Lambda'$ are compact subsets of $\rs$ whose complementary components have closures homeomorphic to closed topological disks. Moreover, assume that the sequences of the diameters of the complementary components of $\Lambda$ and $\Lambda'$ go to 0.    Then any homeomorphism $\xi\colon \Lambda\to\Lambda'$ can be extended to a global homeomorphism of $\rs$.

In particular, if $\xi:\Lambda_H\to\Lambda_H$ is a homeomorphism of a limit set $\Lambda_H$, then $\xi$ can be extended to a homeomorphism of $\rs$. 
\end{lemma}
\begin{proof} 
The boundary circle of each complementary component of $\Lambda$ or $\Lambda'$ is \emph{peripheral}, i.e., it is  a topological circle in $\Lambda$, respectively $\Lambda'$, whose removal does not separate $\Lambda$, respectively $\Lambda'$. It is easy to see that the boundary circles of complementary components of $\Lambda$ and $\Lambda'$ form the full family of peripheral circles in $\Lambda$, respectively $\Lambda'$. Thus, since $\xi$ is a homeomorphism of $\Lambda$ onto $\Lambda'$, this map takes each peripheral circle to another peripheral circle.
Therefore, we can extend $\xi$ homeomorphically into each complementary component of $\Lambda$ to a homeomorphism between the closures of complementary components of $\Lambda$ and $\Lambda'$. Since diameters of peripheral circles go to 0, we conclude that the extension of $\xi$ above is a global homeomorphism.
\end{proof}

It is clear that if $\xi$ can be extended to an orientation preserving homeomorphism, it cannot be extended to an orientation reversing one, and vice versa. We say that $\xi:\Lambda_H\to\Lambda_H$ is \emph{orientation preserving} if it can be extended to an orientation preserving homeomorphism $\rs\to\rs$. Denote by $\homeoplus$ the group of orientation preserving homeomorphisms of $\Lambda_H$, and by $\homeo$ the group of all homeomorphisms of $\Lambda_H$.

Let $D$ be a component of $\Omega_H$. Denote by $|h|$ the word length of $h\in H$ with respect to its generating set from~\eqref{E:HPres}. The \emph{generation of $D$} is defined to be the minimal word length of $h\in H$ so that $\overline{h(D)}$ is a disk in the circle packing $\mathcal C_\T$.

\begin{lemma}[Decreasing generation]\label{Lem:DecreasingGeneration}
Let $D$ be a component of $\Omega_H$ that is a subset of some face $f\in\F_\T$. Then the reflection $R_f$ reduces the generation of $D$ by one.
\end{lemma}
\begin{proof}
If $D$ is a subset of a face, then $\overline{D}$ not a disk in the circle packing $\mathcal C_\T$. Then $D$ is in the $H$-orbit of the interior of some disk $D_0$ in $\mathcal C_\T$, specifically $D=h(D_0)$ for $|h|\geq 1$. Since $D$ is in the face $f$, it follows that $h= R_f\circ h'$, where $|h'|=|h|-1$. Then since $R_f$ is an involution, $|R_f\circ h|=|h'|=|h|-1$ and so the generation of $R_f(D)$ is one less than the generation of $D$.
\end{proof}

Two components of $\Omega_H$ \emph{touch} if their closures intersect. Three components of $\Omega_H$ are said to \emph{mutually touch} if each component touches the other two. We use the same terminology for disks in the original circle packing $\mathcal C_\T$.

A great deal can be said about a homeomorphism of $\Lambda_H$ by understanding its action on the boundary of three mutually touching components of $\Omega_H$.

\begin{lemma}
\label{Lem:DiskTriplesInTriangulations}
Let $D_1,D_2$ be touching components of $\Omega_H$. Then $D_1\cup D_2$ intersects at most one complementary component of the original circle packing $\mathcal C_\T$. 

As a consequence, we also conclude that if $D_1,D_2$, and $D_3$ are mutually touching components of $\Omega_H$, then $D_1\cup D_2\cup D_3$ intersects at most one complementary component of the circle packing $\mathcal C_\T$.
\end{lemma}
\begin{proof} Suppose that $D_1$ and $D_2$ intersect different complementary components of the packing. Then $D_1$ and $D_2$ are contained in distinct faces of the triangulation $\T$.  Neither $D_1$ nor $D_2$ have closures intersecting $\T$, i.e., the vertices or edges of $\T$, because the only points of $\T$ outside of the interior of the original packing $\mathcal C_{\T}$ lie at points of tangency for packing circles. But then $D_1$ and $D_2$ do not touch, contrary to the hypothesis.
\end{proof}

\begin{lemma}[Small triangles to big triangles]\label{Lem:DisksReflectedToPacking}
Let $D_1,D_2,D_3$ be mutually touching components of $\Omega_H$. Then there exists an orientation preserving $h\in H$ so that the closures of $h(D_1)$, $h(D_2)$, $h(D_3)$ are distinct mutually touching disks in the original circle packing $\mathcal C_\T$.
\end{lemma}

\begin{proof}
By Lemma \ref{Lem:DiskTriplesInTriangulations}, either all three disks intersect $\T$ in which case we are done, or there is some disk with positive generation and the union of the three disks intersects  a face $f$. Apply $R_f$ to the three disks. This decreases the generation of disks that are subsets of $f$ by Lemma \ref{Lem:DecreasingGeneration}, and  preserves generations of those disks that intersect the boundary $\partial f$, i.e., disks of generation zero. Iterate until all disks have generation zero. If the resulting map $h$ is orientation reversing, post-compose the map $h$ with $R_f$, where $f$ is a face of $\T$ whose boundary is contained in the closure of $h(D_1)\cup h(D_2)\cup h(D_3)$.  
\end{proof}

\begin{lemma}\label{L:ThreeDisks}
Let $\xi$ be an orientation preserving homeomorphism of $\rs$ such that $\xi|_{\Lambda_H}\colon \Lambda_H\to \Lambda$, where $\Lambda$ is a closed subset of $\rs$ with $\Omega=\rs\setminus\Lambda$ being a union of pairwise disjoint open geometric disks. Assume that for three mutually touching open disks $D_1, D_2, D_3$ of $\mathcal C_\T$ we have that $\xi(D_i)=D_i,\ i=1,2,3$. Then $\Lambda=\Lambda_H$ and $\xi|_{\Lambda_H}$ is the identity transformation. 
\end{lemma}

\begin{proof}
Indeed, our assumption implies that $\xi$ takes the finite circle packing $\mathcal C_\T$ to a circle packing in $\rs$. Moreover, since $\xi$ is orientation preserving and $\xi(D_i)=D_i,\ i=1,2,3$, the uniqueness part of the Circle Packing theorem~\cite[Corollary~13.6.2]{T} implies that $\xi$ fixes setwise each disk in $\mathcal C_\T$. 

 Now, if $f$ is a face of $\T$, we use $R_f$ to reflect  $\mathcal C_\T$ across the boundary of the circle that corresponds to $f$. We denote the union circle packing of $\mathcal C_\T$ and its reflection by $\mathcal C_{\T,f}$. The nerve of this new circle packing is also a triangulation, and all the new disks resulting in the reflections are the closures of disks in $\Omega_H$. Again, from lemma's assumption and the uniqueness of the circle packing, we obtain that $\xi$ fixes each disk of $\mathcal C_{\T,f}$. Continuing this reflection  procedure inductively, say on the diameter of faces of circle packings resulting in successive reflections, we obtain that $\xi$ must fix setwise each disk of $\Omega_H$. This implies, in particular, that $\Lambda=\Lambda_H$.
 
    Finally, each point $p$ in $\Lambda_H$ is an accumulation point for shrinking disks of $\Omega_H$. Since each such disk is fixed by $\xi$, the point $p$ must be a fixed point of $\xi$. We therefore conclude that $\xi|_{\Lambda_H}$ is the identity transformation.     
\end{proof}

We now put a restriction on the triangulation to obtain a simpler formulation of our main symmetry classification results. A \emph{separating triangle} of some triangulation $\mathcal{T}$ is a 3-cycle in $\mathcal{T}$ that is not the boundary of a face. We say that a triangulation is \emph{reduced} if it has no separating triangles. It is natural to expect such a condition in light of the example that appears in Figure \ref{pic:Unreduced}. Specifically, one sees why the reduction hypothesis is needed for Theorem \ref{thm:groupSplits}. 

Examples of reduced triangulations abound. For example, the barycentric subdivision of all faces of a reduced triangulation (i.e., the subdivision of triangular faces into six triangles so that, for each face, there are four new vertices, one on each edge and one inside the face) will result in a reduced triangulation.

\begin{figure}[h]
\centerline{\includegraphics[width=110mm]{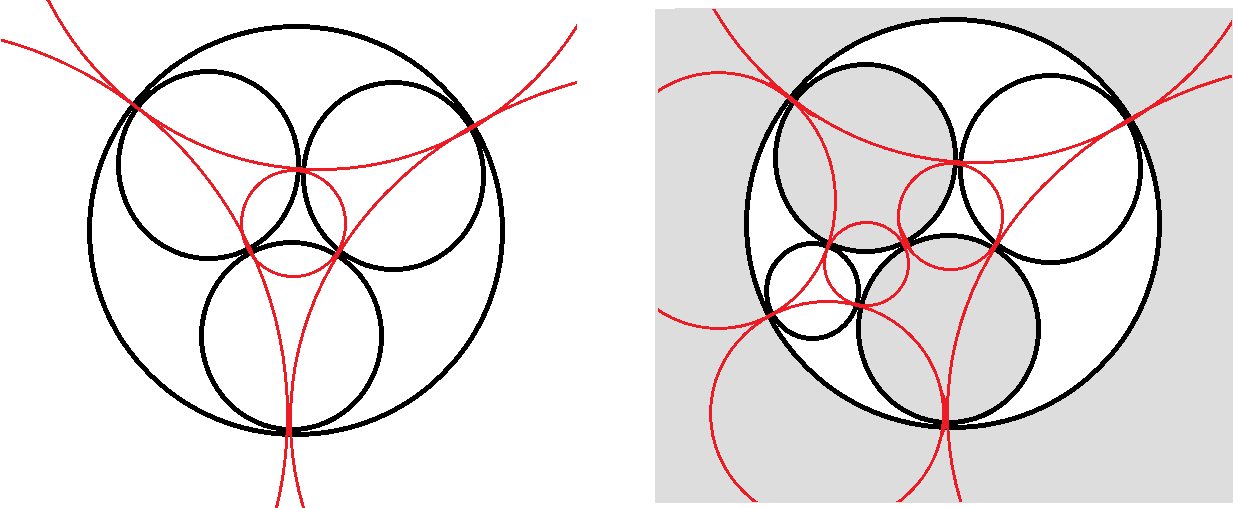}}
\caption{Circle packing corresponding to the tetrahedron in black (left), and the circle packing corresponding to a graph produced by gluing two tetrahedra along a common face also drawn in black (right). The second graph is not reduced due to the 3-cycle that passes through the three grey-shaded disks. Reflecting about the red dual circles produces the classical Apollonian gasket limit set in both cases, despite the fact that the tetrahedron is more symmetric than the second graph (cf. Theorem \ref{thm:groupSplits}).}\label{pic:Unreduced}
\end{figure}

Denote by $\text{Aut}^{\T}(\rs)$ the group of all M\"obius transformations that induce a symmetry of $\T$, i.e., preserve $\mathcal C_\T$. This group is finite because we assume that $\mathcal C_\T$ has more than two disks. The group $\text{Aut}^{\T}(\rs)$ also preserves the dual circle packing, i.e., each element of $\text{Aut}^{\T}(\rs)$ takes a circle that corresponds to $f\in F_\T$ to a circle that corresponds to another $f'\in F_\T$. Therefore, the group $\text{Aut}^{\T}(\rs)$ consists of outer automorphisms of the group $H$, and so $\text{Aut}^{\T}(\rs)$ is a subgroup of $\text{Out}(H)=\text{Aut}(H)/\text{Inn}(H)$.

If $D$ is an open disk in $\mathcal C_\T$, a \emph{flower centered at} $D$ is the disk $\overline{D}$ along with a collection of disks $\{\overline{D_i}\}_{i=0}^{n-1}$, such that each disk $D_i,\ i=0,1,\dots, n-1$, touches $D$ as well as the disks $D_{i-1}, D_{i+1}$, where the indices are taken modulo $n$, and the disks $D_i,\ i=0, 1,\dots, n-1$, are arranged in a cyclic order around $D$. The disks
$\overline{D_i},\ i=0,1,\dots, n-1$, are referred to as \emph{petals} of $D$.

The triangulation $\T$ corresponding to a circle packing $\CC=\CC_\T$ can be realized geometrically as follows.
Mark a point $x_i$ in the interior of each disk $D_i$ of the packing and call it the {\em center} of $D_i$.
Connect  any  two centers  $x_i$ and $x_j$ of two touching disks $D_i$ and $D_j$ by an edge $\gamma_{ij}$  concatenated of two
(spherical) geodesic segments in $D_i$ and $D_j$.  So, $\gamma_{ij}$ meets $\Lambda$ at a single point
where $D_i$ touches $D_j$. We say that two such geometric realizations $\T$ and $\T'$ \emph{coincide} if $\mathcal C_\T=\mathcal C_{\T'}$. In this case $\T$ and $\T'$ are isotopic relative to the points of tangency of the disks in $\mathcal C_\T=\mathcal C_{\T'}$.

If $\Lambda$ is an Apollonian gasket corresponding to a circle packing $\CC=\CC_\T$,
then we say that $\CC$ is a {\em generating packing} for $\Lambda$.   

\begin{proposition}\label{geometric triangulations}
  Let $\CC$ and $\CC'$ be two circle packings generating the same gasket $\Lambda$,
  and let $\T$ and $\T'$ be the corresponding geometric triangulations that we assume to be reduced.
  If $\T$ and $\T'$ share a face,  then they coincide,
  or equivalently $\CC=\CC'$.
\end{proposition} 

\begin{proof}
  Sharing of a face means that $\CC$ and $\CC'$ share three touching disks,
  $D_i=D_i'$,  $i=0,1,2$, and the corresponding interstice, $\Delta= \Delta'$.
  Let us show that in this case they share the whole flower centered at $D_0$.

  Orient the boundary circle $C_0= C_0'$ of $D_0=D_0'$ so that the boundary arc of $\Delta$ is oriented from
$D_1$ to $D_2$. Denote the remaining petals of $D_0$ and $D_0'$ by $\overline{D_i}$ and $\overline{D_i'}$ respectively, where $i\geq 3$. 
Let $i+1$ be the first moment when $D_{i+1}' \not= D_{i+1}$.
Assume for definiteness that $D_{i+1}'$ is closer to $D_2$ than $D_{i+1}$. Then  $D_{i+1}'$ is trapped
inside the interstice $\Delta_i$ attached to  $\{ D_0, D_i, D_{i+1} \} $; see Figure~\ref{F:petals}.
Moreover, the assumption that $\T'$ is reduced implies that $D_{i+1}$ does not belong to the flower
$\{\overline{D_j'}\}$. Hence all further $D_j',\ j> i+1$, are either contained in $\Delta_i$ or are disjoint from its closure (compare
Lemma~\ref{Lem:DiskTriplesInTriangulations}), and so the flower $\{\overline{D_j'}\}$ gets broken in the sense that the petals of this flower are separated by the disk $D_{i+1}$.

\begin{figure}[h]
\begin{tikzpicture}
    \node[anchor=south west,inner sep=0] at (0,0) {\includegraphics[width=7cm]{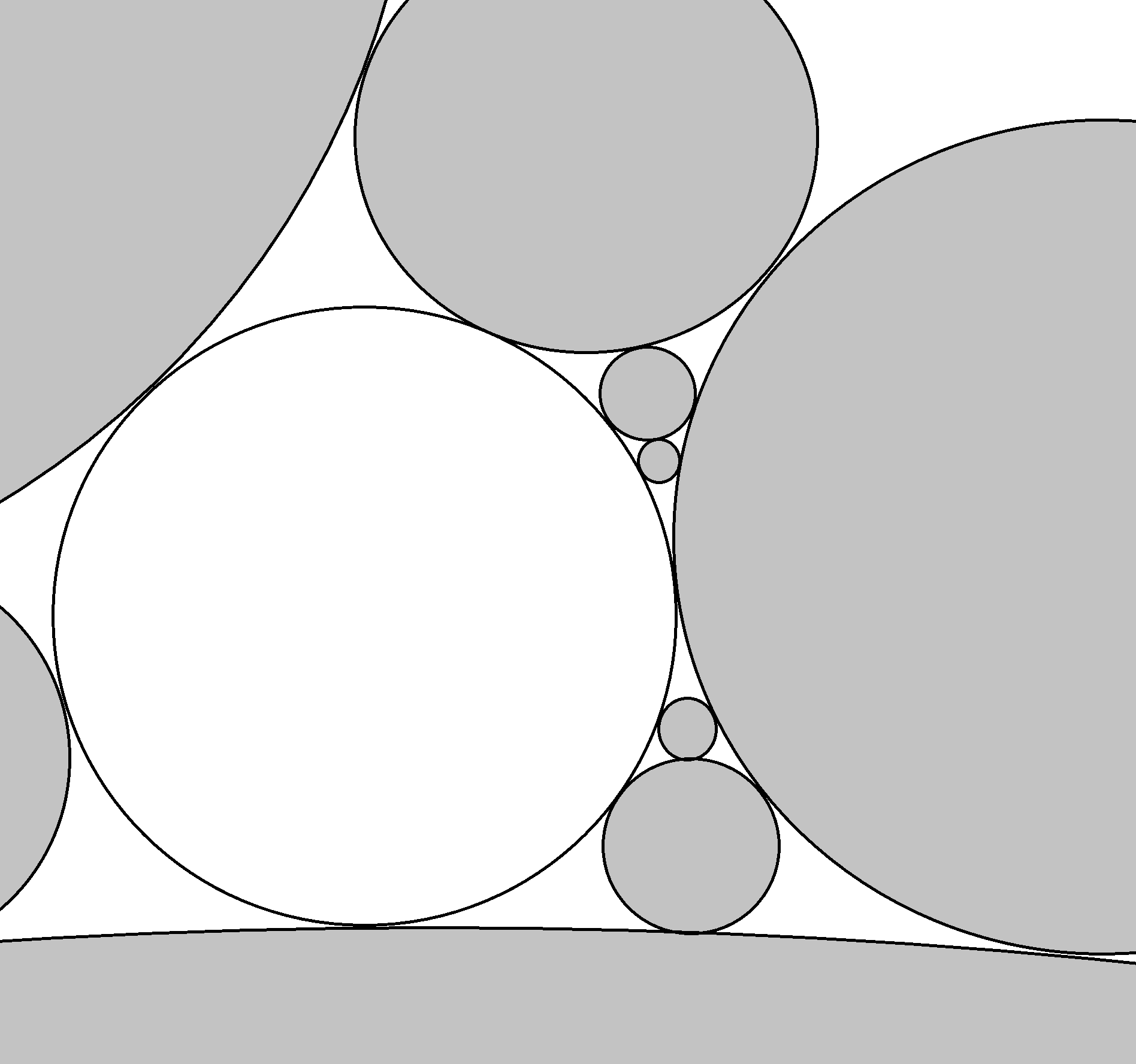}};
\node at (2.43,2.74) {$D_0=D_0'$};
\node at (6.03,3.24) {$D_2=D_2'$};
\node at (3.73,.34) {$D_1=D_1'$};
\node at (3.63,5.7) {$D_{i+1}=D_3$};
\node at (3.99,4.14) {$D_3'$};
\end{tikzpicture}
\caption{Depiction of the petals used in the proof of Proposition~\ref{geometric triangulations}}
\label{F:petals}
\end{figure}

We can now apply this result to the petals of the above flower and conclude that the two packings
share the flowers around each of these petals. 
Proceeding inductively, we complete the proof.
\end{proof}

\begin{theorem}[Decomposition of symmetries] \label{Thm:ClassificationOfSymmetries}
Let $\T$ be a reduced triangulation. Then if $\xi\in\homeo$, we have  
\[\xi=A\circ h|_{\Lambda_H}\] 
for some $h\in H$ and  $A\in\text{Aut}^{\T}(\rs)$.
\end{theorem} 

\begin{proof}
From Lemma~\ref{L:Extension}, we know that $\xi$ extends to a homeomorphism of $\rs$, which we continue to denote by $\xi$.
Let $D_0$ be an open disk in $\mathcal C$ and $\overline{D_1}, \overline{D_2}$ be two petals of $D_0$ in $\mathcal C$.  Since $\T$ is reduced, there is a face $f$ of $\T$ whose boundary is contained in $\overline{D_0}\cup \overline{D_1}\cup \overline{D_2}$. Then $\xi^{-1}(D_i),\ i=0,1,2$, are mutually touching disks of $\Omega_H$. By Lemma~\ref{Lem:DisksReflectedToPacking}, there is $h\in H$ so that $D_i'=h\xi^{-1}(D_i),\ i=0,1,2$, are mutually touching disks in the original circle packing $\mathcal C_\T$. We may assume that the map $h\xi^{-1}$ is orientation preserving by possibly post-composing it with the reflection in the circle passing through the three points of mutual tangency of $D_i',\ i=0,1,2$.

Because $\T$ is assumed to be reduced and $h\xi^{-1}$ is orientation preserving, we conclude that the intersection of $\T$ with $\overline{D_0'}\cup\overline{D_1'}\cup\overline{D_2'}$ bounds a face $f'$ of $\T$.
Let $\mathcal C_\T'$ be the circle packing given by 
$$
\mathcal C_\T'=h\xi^{-1}(\mathcal C_\T).
$$
Note that since $h\xi^{-1}$ is an orientation preserving homeomorphism of $\rs$, this map induces a graph isomorphism of $\T$ onto the nerve $\T'$ of $\mathcal C_\T'$ such that the face $f$ of $\T$ is mapped to the face $f'$ of $\T'$, which is also a face of $\T$.

Let $A$ be a M\"obius transformation such that $A(D_i')=D_i,\ i=0,1,2$. Such an $A$ exists because one can always map the three points of mutual tangency of $D_0', D_1', D_2'$ to the corresponding three points of mutual tangency of $D_0, D_1,D_2$ by a M\"obius transformation, and it would necessarily map $D_i'$ onto $D_i,\ i=0,1,2$. 

Now we apply Lemma~\ref{L:ThreeDisks} to  $Ah\xi^{-1}$ to conclude that this map restricted to $\Lambda_H$ is the identity map, i.e., $\xi=Ah$ on $\Lambda_H$. It remains to show that $A\in\text{Aut}^{\T}(\rs)$, i.e., that $A$ preserves $\mathcal C_\T$. This is equivalent to $\mathcal C_{\T}'=\mathcal C_\T$.

Since $\xi$ and $h$ leave $\Lambda_H$ invariant, the identity $A=\xi h^{-1}$ restricted to $\Lambda_H$ shows that so does $A$.
Hence $A(\CC_\T) $ is a generating packing for $\Lambda_H$. 
Moreover, it shares the face $f=A(f')$ of $\CC_\T$ corresponding to $\{ D_0, D_1, D_2 \}$.  
By Proposition \ref{geometric triangulations},  $A(\CC_\T)= \CC_\T$.
\end{proof}

With some additional work, this decomposition implies that $\homeo$ splits.

\begin{theorem}\label{thm:groupSplits}
Let $\T$ be a reduced triangulation. Then $H$ is a normal subgroup of $\homeo$ and 
\[\homeo=\text{Aut}^{\T}(\rs) \ltimes H.\]

\end{theorem}

\begin{proof}
First, $\text{Aut}^{\T}(\rs)\cap H$ is trivial because any nontrivial $h\in H$ has the property that there is some open disk $D$ in the original circle packing $\mathcal C_\T$ whose image $h(D)$ has non-zero generation. In contrast, every element of $\text{Aut}^{\T}(\rs)$ preserves the generation of any disk in $\mathcal C_\T$.

By Theorem \ref{Thm:ClassificationOfSymmetries}, it now suffices to prove that $\xi H = H \xi$ where $\xi\in \text{Aut}^{\T}(\rs)$. 
Let $f$ be a face of $\T$, and let $R_{f}$ be the generator of $H$ that corresponds to reflection over the boundary of $f$.  Let $f':=\xi(f)$. Then, since $\xi\in \text{Aut}^{\T}(\rs)$, we have that $f'$ is also a face of $\T$. Let $R_{f'}$ be the generator of $H$ that corresponds to the reflection over the boundary of $f'$. Then, applying Lemma~\ref{L:ThreeDisks} to the three disks intersecting the boundary of $f$,  we obtain $\xi^{-1}R_{f'}^{-1}\xi R_{f}={\rm id}$.  Thus, $\xi R_{f}=R_{f'}\xi$, implying that $\xi H = H \xi$.
\end{proof}

We conclude this section by explicitly constructing a fundamental domain for the action of $H$ on $\Omega_H$. See Figure \ref{fig:apolloFD} for the simplest example. Recall that $D_v$ denotes the (round disk) component of $\Omega_H$ containing $v\in V_\T$. Let $D_f$ denote the open disk whose boundary passes through the three vertices of the interstice $\Delta_f$ corresponding to face $f\in\F_\T$. (We recall that $\partial D_f$ is orthogonal to $\partial D_v$ for exactly three vertices $v\in\V_\T$.)

\begin{figure}
\centerline{\includegraphics[width=60mm]{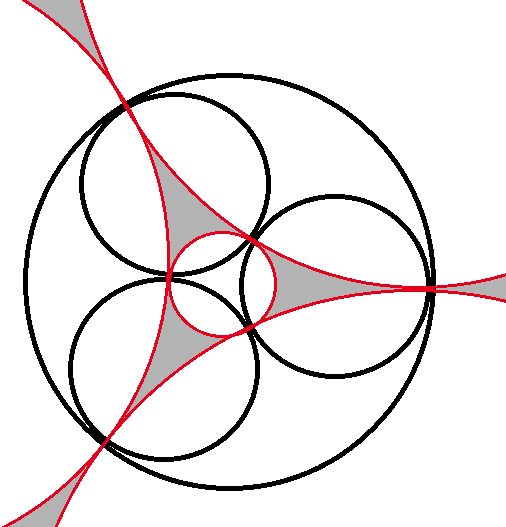}}
\caption{Black circles represent the circle packing that arise when $\T$ is a tetrahedron. The red ``dual'' circles define the generators of $H$, and the grey region represents a fundamental domain.}
\label{fig:apolloFD}
\end{figure}

\begin{proposition}[Fundamental domain for $H$]\label{Prop:FDforH} A fundamental domain for $H$ acting on $\Omega_H$ is given by
\[\left(\bigcup_{v\in\V_\T} D_v\right)\setminus \left(\bigcup_{f\in\F_\T} D_f\right).\]
\end{proposition}
\begin{proof}
By Equation \eqref{E:LambdaUnion}, the $H$-orbit of a point in $\Omega_H$ must intersect $\displaystyle{\bigcup_{v\in\V_\T} D_v}$. The $H$-stabilizer of a given $D_v$ is the group generated by the reflections whose defining circles are orthogonal to the boundary of $D_v$. It thus suffices to compute the fundamental domain for this stabilizer acting on $D_v$ and the conclusion follows by taking the union of the stabilizer fundamental domains over all $v\in \V_\T$.  

The set $\displaystyle{D_v\setminus \bigcup_{f\in\F_\T} D_f}$ is an ideal polygon $P_v$ in ${D}_v$ equipped with the standard Poincar\'e metric, where the number of sides is given by the number of circles tangent to the boundary of $D_v$ in the original packing $\mathcal C_\T$. Denote the sides of $P_v$ by $s_1,\dots, s_k$, and let $R_{s_1},\dots, R_{s_k}$ be the reflections through the circles that define the sides.  Then by the Poincar\'e polygon theorem, $P_v$ is a fundamental domain for the group $\langle R_{s_1},\dots, R_{s_k}\rangle$ acting on $D_v$. The following set equality is immediate:
\[\bigcup_{v\in\V_\T} P_v = \left(\bigcup_{v\in\V_\T} D_v\right)\setminus \left(\bigcup_{f\in\F_\T}D_f\right).\]
\end{proof}


\section{Nielsen maps induced by reflection groups}\label{group_equiv_map_sec}

The goal of this section is to introduce and study some basic properties of maps that are \emph{orbit equivalent} to the reflection groups considered in the paper. These maps, which we call \emph{Nielsen maps}, are defined piecewise using anti-M{\"o}bius reflections that generate the corresponding group, and enjoy certain Markov properties when restricted to the limit set of the group. Related constructions of such Markov maps on the limit set (originally introduced to code geodesics) can be found in \cite{Bowen,BS,Nielsen, Se} (for Fuchsian groups), \cite{Rocha} (for certain Kleinian groups), \cite{CF,GH,PS} (for hyperbolic groups). Our nomenclature ``Nielsen map'' follows \cite{Bowen}, where similar maps arising from Fuchsian groups were called ``Nielsen developments''.

\subsection{The Nielsen map for the regular ideal polygon group}\label{nielsen_1_subsec}

Let us denote the open unit disk and the unit circle in the complex plane (centered at the origin) by $\D$ and $\mathbb{T}$ respectively. For $d\geq 2$, let $C_1, C_2, \cdots, C_{d+1}$ be circles of equal radii each of which intersects $\mathbb{T}$ orthogonally such that $\displaystyle\bigcup_{i=1}^{d+1} \widetilde{C}_i$ (where $\widetilde{C}_i:=C_i\cap\D$) is an ideal $(d+1)$-gon with vertices at the $(d+1)$-st roots of unity. They bound a closed (in the topology of $\D$) region $\Pi$ (see Figure~\ref{ideal_triangle_pic} for $d=2$).

\begin{figure}[ht!]
\begin{center}
\includegraphics[scale=0.32]{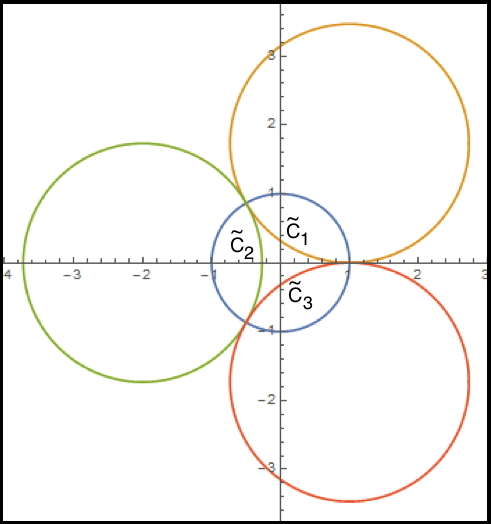}
\end{center}
\caption{The hyperbolic geodesics $\widetilde{C}_1$, $\widetilde{C}_2$ and $\widetilde{C}_3$, which are sub-arcs of the circles $C_1$, $C_2$ and $C_3$ respectively, form an ideal triangle in $\D$.}
\label{ideal_triangle_pic}
\end{figure}

Reflections with respect to the circles $C_i$ are anti-conformal involutions (hence automorphisms) of $\D$, and we call them $\rho_1, \rho_2, \cdots, \rho_{d+1}$. The maps $\rho_1, \rho_2, \cdots, \rho_{d+1}$ generate a reflection group $\mathbf{H}_d$, which is called the (regular) \emph{ideal $(d+1)$-gon group}. As an abstract group, it is given by the generators and relations $$\langle\rho_1, \rho_2, \cdots, \rho_{d+1}: \rho_1^2=\rho_2^2=\cdots=\rho_{d+1}^2=\mathrm{id}\rangle.$$ In the particular case $d=2$, the group $\mathbf{H}_2$ is called the \emph{ideal triangle group}.

We will denote the connected component of $\D\setminus \Pi$ containing $\Int{\rho_i(\Pi)}$ by $\D_i$. Note that $\overline{\D}_1\cup\cdots\cup\overline{\D}_{d+1}=\overline{\D}\setminus\Int{\Pi}$, where $\overline{\D}$ denotes closure of $\D$ in the Euclidean metric.

The \emph{Nielsen map} $\pmb{\rho}_d:\overline{\D}\setminus\Int{\Pi}\to\overline{\D}$ associated with the ideal polygon group $\mathbf{H}_d$ is defined as:
$$
z \mapsto \begin{array}{lll}
                    \rho_i(z) & \mbox{if}\ z\in \overline{\D}_i,\ i=1, \cdots, d+1
                                          \end{array}. 
$$
Clearly, $\pmb{\rho}_d$ restricts to an expansive orientation reversing $d$-fold covering of $\mathbb{T}$ with associated Markov partition $\mathbb{T}=(\partial\D_1\cap\mathbb{T})\cup(\partial\D_2\cap\mathbb{T})\cup\cdots\cup(\partial\D_{d+1}\cap\mathbb{T})$. The corresponding transition matrix is given by 
$$
M=\begin{bmatrix}
0&1&1&\dots&1&1\\
1&0&1&\dots&1&1\\
\hdotsfor{6}\\
1&1&1&\dots&1&0
\end{bmatrix}.
$$

Two points $x,y\in\overline{\D}$ are said to lie in the same grand orbit of the Nielsen map $\pmb{\rho}_d$ if there exist non-negative integers $n_1,n_2$ such that $\pmb{\rho}_d^{\circ n_1}(x)=\pmb{\rho}_d^{\circ n_2}(y)$. On the other hand, an $\mathbf{H}_d-$orbit is defined as the set $\{h(z):h\in \mathbf{H}_d\}$, for some $z\in\overline{\D}$. 
It is easy to see that the grand orbit of any point in $\overline{\D}$ under $\pmb{\rho}_d$ coincides with its orbit under $\mathbf{H}_d$ (compare Proposition~\ref{prop:orbit_equiv_1}). In other words, $\pmb{\rho}_d$ is orbit equivalent to $\mathbf{H}_d$ on $\overline{\D}$.

Let us now describe how the expanding $d$-fold covering of the circle $\overline{z}^d:\mathbb{T}\to\mathbb{T}$ is related to $\pmb{\rho}_d$. The map $\overline{z}^d\vert_\mathbb{T}$ admits the same Markov partition as $\pmb{\rho}_d$ with the same transition matrix $M$. Moreover, the symbolic coding maps for $\pmb{\rho}_d$ and $\overline{z}^d$ (coming from their common Markov partitions) have precisely the same fibers, and hence they induce a homeomorphism $\pmb{\mathcal{E}}_d:\mathbb{T}\to\mathbb{T}$ conjugating $\pmb{\rho}_d$ to $\overline{z}^d$ (see \cite[\S 2]{LLMM1} for a more detailed discussion).

\subsection{The Nielsen map for $H_\T$}\label{nielsen_2_subsec}
In this subsection, we associate to an arbitrary triangulation $\T$ of $\mathbb{S}^2$, a \emph{Nielsen map} $N_\mathcal{T}$ that is orbit equivalent to the reflection group $H=H_\T$. As mentioned earlier, the map $N_\mathcal{T}$ is defined piecewise using the anti-M{\"o}bius reflections $R_f$ ($f\in F_\mathcal{T}$), and enjoys certain Markov properties when restricted to the limit set $\Lambda_H$. 

Let us fix an arbitrary triangulation $\mathcal{T}$ of $\mathbb{S}^2$, and consider the circle packing $\mathcal{C}_\mathcal{T}$ along with its dual circle packing (see Figure~\ref{fig:group_map} for the case of the tetrahedral triangulation). We denote the open disks bounded by the dual circles by $D_f$ (such that $D_f$ contains a unique triangular interstice of $\mathcal{C}_\mathcal{T}$), and the reflection in $C_f:=\partial D_f$ by $R_f$, for $f\in F_\mathcal{T}$. The set of points where the disks $D_f$ touch is denoted by $S$. Each connected component of 
$$
T^0:=\widehat{\C}\setminus\displaystyle\left(S\bigcup_{f\in F_\mathcal{T}} D_f\right)
$$ 
is called a \emph{fundamental tile}. 

\begin{figure}[ht]
\begin{tikzpicture}
\node[anchor=south west,inner sep=0] at (-3,0) {\includegraphics[width=0.49\textwidth]{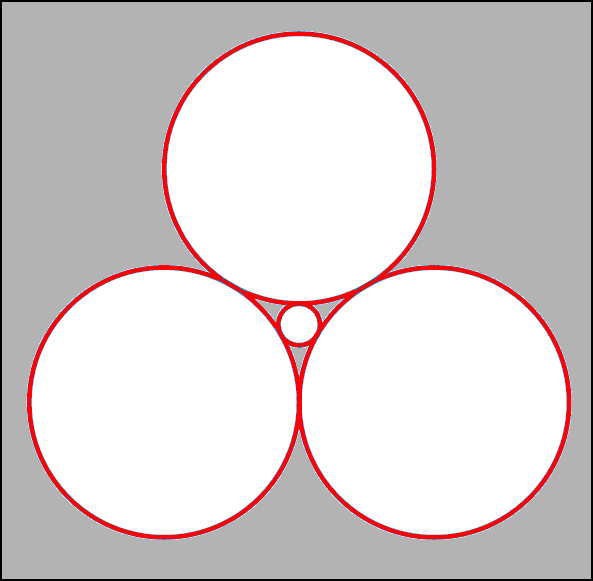}};
\node[anchor=south west,inner sep=0] at (3.6,-0) {\includegraphics[width=0.5\textwidth]{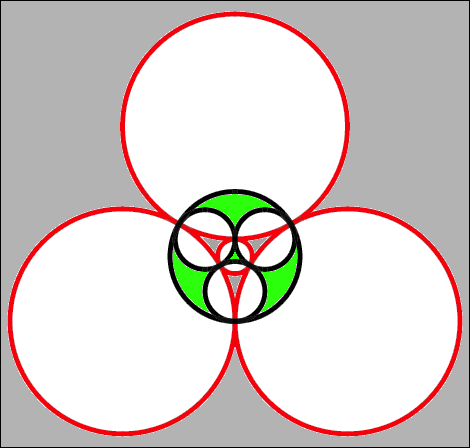}};
\node at (-1.2,1.8) {\begin{huge}$D_1$\end{huge}};
\node at (1.4,1.8) {\begin{huge}$D_2$\end{huge}};
\node at (0.2,4.2) {\begin{huge}$D_3$\end{huge}};
\node at (0.12,2.7) {\begin{tiny}$D_4$\end{tiny}};
\end{tikzpicture}
\caption{Left: The map $N_\mathcal{T}$ is defined as reflection with respect to $\partial D_i$ (in red) on $\overline{D_i}$. The grey region is $T^0$. Right: The four invariant components of the tiling set of $N_\mathcal{T}$ are precisely the disks bounded by the black circles, which form the circle packing $\mathcal{C}_\mathcal{T}$. The interstices of $\mathcal{C}_\mathcal{T}$ are marked in green.}
\label{fig:group_map}
\end{figure}

Recall from Section~\ref{sec:TrianglesToGaskets} that the anti-M{\"o}bius maps $R_f$ ($f\in F_\mathcal{T}$) generate a reflection group $$H:=H_\T=\langle R_f: f\in F_\mathcal{T}\rangle.$$ We know from Proposition~\ref{Prop:FDforH} that $T^0$ is a fundamental domain for the action of $H$ on its domain of discontinuity $\Omega_H$. Here is an alternate way of seeing this. 

For each $f\in F_\mathcal{T}$, let us consider the upper hemisphere $S_f\in\mathbb{H}^3$ such that $\partial S_f\cap\partial\mathbb{H}^3= C_f$; i.e., $C_f$ bounds the upper hemisphere $S_f$. By Poincar{\'e}'s original observation, the anti-M{\"o}bius map $R_f$ extends naturally to the reflection in $S_f$, and defines an anti-conformal automorphism of $\mathbb{H}^3$. Let $\mathcal{P}_\T$ be the convex hyperbolic polyhedron (in $\mathbb{H}^3$) whose relative boundary in $\mathbb{H}^3$ is the union of the hemispheres $S_f$ (see \cite[Figure~9]{Kon} for an illustration of the polyhedron $\mathcal{P}_\T$ in the case of the classical Apollonian gasket). Then, $\mathcal{P}_\T$ is a fundamental domain (called the Dirichlet fundamental polyhedron) for the action of the group $H$ on $\mathbb{H}^3$, and $T^0=\overline{\mathcal{P}_\T}\cap\Omega_H$ (where the closure is taken in $\Omega_H\cup\mathbb{H}^3$) is a fundamental domain for the action of $H$ on $\Omega_H$ (see \cite[\S 3.5]{Marden}, also compare \cite{V}). 

We now define the Nielsen map $N_\mathcal{T}$ on $\displaystyle\bigcup_{f\in F_\mathcal{T}} \overline{D_f}$ by setting $$N_\mathcal{T}\equiv R_f\quad \mathrm{on}\ \overline{D_f}.$$

Let us now briefly describe the Markov properties of $N_\T:\Lambda_H\to\Lambda_H$. To this end, let us first note that $$\Lambda_H=\bigcup_{f\in F_\T} T_f,\ \textrm{where}\ T_f:= \overline{D_f}\cap\Lambda_H.$$ Then, we have  
\begin{itemize}
\item $\Int{T_f}\cap \Int{T_{f'}}=\emptyset$, for $f\neq f'$ (where the interior is taken in the subspace topology of $\Lambda_H$),
\item each $T_f$ is injectively mapped by $N_\T$ onto the union $\displaystyle\bigcup_{f'\neq f}T_{f'}$.
\end{itemize} 
Hence, the sets $\{T_f\}_{f\in F_\T}$ form a Markov partition for $N_\T:\Lambda_H\to\Lambda_H$ with the $(d+1)\times(d+1)$ transition matrix 
$$
\begin{bmatrix}
0&1&1&\dots&1&1\\
1&0&1&\dots&1&1\\
\hdotsfor{6}\\
1&1&1&\dots&1&0
\end{bmatrix},
$$
where $d=|F_\T|-1$ is the number of faces of $\T$.

\begin{proposition}[Orbit equivalence]\label{prop:orbit_equiv_1}
The reflection map $N_\mathcal{T}$ is orbit equivalent to the reflection group $H$ on $\widehat{\C}$.
\end{proposition}
\begin{proof}
Recall that two points $x,y\in\widehat{\C}$ are said to lie in the same grand orbit of $N_\mathcal{T}$ if there exist non-negative integers $n_1,n_2$ such that $N_\mathcal{T}^{\circ n_1}(x)=N_\mathcal{T}^{\circ n_2}(y)$. On the other hand, a $H$-orbit is defined as the set $\{h(z):h\in H\}$, for some $z\in\widehat{\C}$. We need to show two points lie in the same grand orbit of $N_\mathcal{T}$ if and only if they lie in the same $H$-orbit.

To this end, let us pick $x,y\in\widehat{\C}$ in the same grand orbit of $N_\mathcal{T}$. Since $N_\mathcal{T}$ acts by the generators $R_f$ ($f\in F_\mathcal{T}$) of the group $H$, it directly follows that there exists an element of $H$ that takes $x$ to $y$; i.e., $x$ and $y$ lie in the same $H$-orbit.

Conversely, let $x,y\in\widehat{\C}$ lie in the same $H$-orbit; i.e., there exists $h\in H$ with $h(x)=y$. By definition, we have that $h=R_{s_1}R_{s_2}\cdots R_{s_k}$, for some $s_1,\cdots s_k\in F_\mathcal{T}$. A simple application of mathematical induction shows that it suffices to prove grand orbit equivalence of $x$ and $y$ (under $N_\mathcal{T}$) in the case $k=1$. Therefore, we assume that $R_{s_1}(x)=y$. Note that either $x$ or $y$ must belong to $\overline{D_{s_1}}$. Since $R_{s_1}(x)=y$ implies $R_{s_1}(y)=x$, there is no loss of generality in assuming that $x\in \overline{D_{s_1}}$. Now, the condition $R_{s_1}(x)=y$ can be written as $N_\mathcal{T}(x)=y$, which proves that $x$ and $y$ lie in the same grand orbit of $N_\mathcal{T}$.
\end{proof}

\begin{proposition}\label{prop:domain_discont}
$$\Omega_H=\bigcup_{n\geq0} N_\mathcal{T}^{-n}(T^0).$$
\end{proposition}
\begin{proof}
This follows from Proposition~\ref{prop:orbit_equiv_1} and the fact that $T^0$ is a fundamental domain for the action of $H$ on its domain of discontinuity $\Omega_H$.
\end{proof}

Let us conclude this subsection with a brief discussion on the index two Kleinian subgroup $\Gamma_\T\leqslant H_\T$ consisting of the (orientation preserving) M{\"o}bius maps in $H_\T$, in the case when the dual graph $\check{\T}$ of the triangulation $\T$ is Hamiltonian (i.e., $\check{\T}$ contains a cycle that visits every vertex exactly once). Note that $\Gamma_\T$ is a geometrically finite Kleinian group (its fundamental polyhedron in $\mathbb{H}^3$ is obtained by `doubling' $\mathcal{P}_\T$, and hence has finitely many sides). 

Denoting the index two Fuchsian subgroup of $\mathbf{H}_d$ by $\pmb{\Gamma}_d$ (where $\mathbf{H}_d$ is the regular ideal $(d+1)$-gon reflection group introduced in Subsection~\ref{nielsen_1_subsec}), one easily sees that the top and bottom surfaces $\D/\pmb{\Gamma}_d$ and $(\widehat{\C}\setminus\overline{\D})/\pmb{\Gamma}_d$ associated with the Fuchsian group $\pmb{\Gamma}_d$ are $(d+1)$-times punctured spheres. The assumption that $\check{\T}$ is Hamiltonian implies that the group $\Gamma_{\mathcal{T}}$ can be obtained as a limit of a sequence of quasi-Fuchsian deformations of $\pmb{\Gamma}_d$. More precisely, this is achieved by considering a sequence of quasi-Fuchsian deformations of $\pmb{\Gamma}_d$ that pinch suitable collections of simple closed non-peripheral geodesics on $\D/\pmb{\Gamma}_d$ and $(\widehat{\C}\setminus\overline{\D})/\pmb{\Gamma}_d$ simultaneously. Thus, $\Gamma_\T$ is a \emph{cusp group} that lies on the boundary of the quasi-Fuchsian deformation space of the Fuchsian group $\pmb{\Gamma}_d$. These geodesics (that we pinch) lift by $\pmb{\Gamma}_d$ to the universal covers $\D$ and $\widehat{\C}\setminus\overline{\D}$ giving rise to a pair of geodesic laminations (of $\D$ and $\widehat{\C}\setminus\overline{\D}$ respectively), \cite[\S 3.9.1]{Marden}. Since $\pmb{\rho}_d\vert_{\mathbb{T}}$ is orbit equivalent to the action of $\mathbf{H}_d$, it follows that these two laminations (viewed as equivalence relations on $\mathbb{T}$) are $\pmb{\rho}_d$-invariant; i.e., the $\pmb{\rho}_d$-images of the end-points of a leaf are the end-points of some leaf of the lamination. Moreover, the quotient of $\mathbb{T}$ by identifying the endpoints of the leaves of both these laminations produces a topological model of the limit set $\Lambda_H$, and the (equivariant) quotient map from $\mathbb{T}$ onto $\Lambda_H$ semi-conjugates $\pmb{\rho}_d:\mathbb{T}\to\mathbb{T}$ to $N_\T:\Lambda_H\to\Lambda_H$ (see \cite{Floyd}, also compare \cite{MS}). In fact, this quotient map is the Cannon-Thurston map for $H_\T$ (see \cite[\S 2.2]{MS} for the definition of Cannon-Thurston maps).

The domain of discontinuity of $\Gamma_\T$ is equal to $\Omega_H$. If the valences of the vertices of $\mathcal{T}$ are $n_1,\cdots, n_{\vert V_\T\vert}$, then the quotient $$\mathcal{M}(\Gamma_\T):=(\mathbb{H}^3\cup\Omega_H)/\Gamma_\T$$ is an infinite volume $3$-manifold whose conformal boundary $\partial\mathcal{M}(\Gamma_\T):=\Omega_H/\Gamma_\T$ consists of $\vert V_\T\vert$ Riemann surfaces which are spheres with $n_1,\cdots n_{\vert V_\T\vert}$ punctures (respectively). 

In the particular case of the tetrahedral triangulation $\T$ (which gives rise to the classical Apollonian gasket limit set), the conformal boundary $\partial\mathcal{M}(\Gamma_\T)$ consists of $4$ triply punctured spheres. In this case, $\Gamma_\T$ is obtained by pinching two geodesics, one on the top and one on the bottom $4$-times punctured sphere determined by $\pmb{\Gamma}_3$. In fact, these geodesics correspond to \emph{pants decompositions} of the top and bottom $4$-times punctured spheres. Thus in this case, $\Gamma_\T$ is a \emph{maximal cusp} group (see \cite[\S 5.3]{Marden} for a discussion of maximal cusp groups). Such groups are known to be \emph{rigid}. More precisely, if a $3$-manifold $\mathcal{M}(\Gamma')$ (arising from some Kleinian group $\Gamma'$) is homeomorphic to $\mathcal{M}(\Gamma_\T)$, then they are in fact isometric, and the two Kleinian groups $\Gamma_\T$ and $\Gamma'$ are conjugate by a M{\"o}bius map. In particular, $\Gamma_\T$ is quasiconformally rigid \cite[Theorems~3.13.4, 5.1.3]{Marden}.

\section{Topological surgery: from Nielsen map to a branched covering}\label{Sec:RoundGasketsNoninvertible}

Let $g:\sphere\to\sphere$ be a branched cover, and let $C_g$ be its set of critical points. The \emph{postcritical set} is given by 
$
P_g=\displaystyle\bigcup_{i>0}g^{\circ i}(C_g).
$
A \emph{Thurston map} is a branched cover $g:\sphere\to\sphere$ of degree $d$ so that $|P_g|<\infty$ and $|d|>1$. Contrary to the usual definition, we admit orientation reversing branched covers.

In this section we construct an orientation reversing branched cover $\G_{\T}:\sphere\to\sphere$ associated to a triangulation $\T$. This is done in such a way that each vertex in $\T$ is fixed by $\G_\T$ and each edge in $\T$ is invariant.

Recall that each face $f\in \F_\T$ contains a unique interstice $\Delta_f$ of the circle packing.
\begin{lemma}\label{L:TrianglesShrink}
Let $\Delta_f$ and $h_n=R_{f_{i_1}}R_{f_{i_2}}\dots R_{f_{i_n}}\in H,\ i_j\neq i_{j+1},\ j=1,2,\dots, n-1$, be arbitrary, and let $\Delta_{f, h_n}$ be defined by $\Delta_f=h_n(\Delta_{f, h_n})$. Then 
$$
{\rm diam}(\Delta_{f, h_n})\to 0 
$$
as the word length $|h_n|=n$ goes to infinity.
\end{lemma} 
\begin{proof}
We argue by contradiction and assume that there exist $\delta>0$ and a sequence $(h_n)$ such that ${\rm diam}(\Delta_{f, h_n})\geq \delta$ for all $n\in\N$. Since there are only finitely many disks in $\Omega_H$ whose diameters are bounded away from 0 by a fixed positive number, there must exist a subsequence $(h_{n_k})$ such that one of the sides of $\Delta_{f, h_{n_k}}$ is contained in the boundary of a fixed disk of $\Omega_H$, for all $k\in\N$. In fact, the triangle inequality applied to the vertices of $\Delta_{f, h_{n_k}}$ implies that, by possibly passing to a further subsequence of $(h_{n_k})$, we may assume that two of the sides of $\Delta_{f, h_{n_k}}$ are contained in the boundaries of two distinct fixed disks of $\Omega_H$. 

Now, let $D_{n_k}$ be the disk in $\Omega_H$ whose boundary contains the third side of $\Delta_{f, h_{n_k}}$. This disk has to be mutually touching with the two disks whose boundaries contain the other two sides. If ${\rm diam}(D_{n_k})\to0$ as $k\to\infty$, we get a contradiction with our assumption. Otherwise, by possibly passing to yet another subsequence, we may assume that $D_{n_k}$ is also fixed for all $k\in\N$. This however also leads to a contradiction because the word length $|h_{n_k}|$ goes to $\infty$. 
\end{proof}
  
We first define $\G_\T$ on each closure $\overline{\Delta_f}$ to be the restriction $\G_\T=R_f|_{\overline{\Delta_f}}$. Note that this implies $N_\mathcal{T}\equiv G_\mathcal{T}$ on $\Lambda_H.$ Now, let $v\in V_\T$ be arbitrary and let $D_v$ be the corresponding open disk in the circle packing $\mathcal C_\T$. The map $\G_\T$ is already defined on the boundary circle $\partial D_v$ of $D_v$, and it is a piecewise reflection map. Let $k=k(v)$ be the number of triangular interstices adjacent to $D_v$ minus one. We have $k+1$ points $p_0, p_1,\dots, p_{k}$ on $\partial D_v$ that are common points of pairs of adjacent interstices. We assume that they are enumerated in a cyclic order along $\partial D_v$. These points are fixed by $\G_\T$, and they are the only fixed points of $\G_\T$. Moreover, the degree of the map $\G_\T|_{\partial D_v}$ is $-k$. Therefore, there is an orientation preserving homeomorphism $\phi_v$ of $\partial D_v$ onto the unit circle in the plane that conjugates $\G_\T|_{\partial D_v}$ to the map $g_k(z)=\bar z^k$. The map $\phi_v$ takes $p_0, p_1, \dots, p_k$ to $e^{2\pi i j/(k+1)},\ j=0,1,\dots, k$, the fixed points of $g_k$. The fact that such a conjugating homeomorphism $\phi_v$ exists follows from Lemma~\ref{L:TrianglesShrink}. Indeed, this lemma implies that the lengths of the complementary intervals of $\displaystyle\bigcup_{j=0}^k \G_\T^{-n}(p_j)$ go to 0 as $n\to\infty$.

Let $\Phi_v$ be a homeomorphism of the closure $\overline{D_v}$ onto $\overline{\D}$ that extends $\phi_v$.
For each $v\in V_\T$, we define $\G_\T|_{D_v}=\Phi_v^{-1}\circ g_k\circ \Phi_v$. This defines a global continuous map $\G_\T$ of $\rs$ to itself. 

The map $g_k$ fixes setwise each ray $\rho_j$ from the origin to the fixed point $e^{2\pi i j/(k+1)}$, $j=0,1,\dots, k$.
For each $v\in V_\T$, the point $u_v=\Phi^{-1}_v(0)\in D_v$ is a fixed point of $\G_\T$ and the arcs $\alpha_{v,j}=\Phi_v^{-1}(\rho_j),\ j=0,1,\dots, k$, are setwise fixed. If two vertices $v_1, v_2\in V_\T$ are such that the corresponding disks $D_{v_1}, D_{v_2}$ are tangent, then there exist two fixed arcs $\alpha_{v_1,j_1}$ in $D_{v_1}$ and $\alpha_{v_2, j_2}$ in $D_{v_2}$ that have the same endpoint, the tangent point of $D_{v_1}$ and $D_{v_2}$. Their concatenation along with the tangent point form a fixed arc $t_{v_1v_2}$ that connects $u_{v_1}$ and $u_{v_2}$. 
The triangulation $\T$ is isotopic to the triangulation $\T'$ whose vertices are the points $u_v,\ v\in V_\T$, and the edges are $t_{v_1v_2}$, where $v_1, v_2$ are such that $D_{v_1}, D_{v_2}$ are tangent. In what follows, we identify $\T$ and $\T'$, and therefore conclude that $\G_\T$ keeps $\T$ invariant. More specifically, $\G_\T$ fixes the vertices of $\T$ pointwise and fixes the edges of $\T$ setwise. 
We summarize the properties of $\G_T$ in the following proposition.

\begin{proposition}[Properties of the branched covering $G_\T$]\label{prop:AntiThurstonMappingProperties}
The map $\G_\T:\sphere\to\sphere$ is an orientation reversing branched cover such that
\begin{enumerate}
\item the degree of $\G_\T$ is $1-|\F_\T|$,
\item the set of critical points of $\G_\T$ is given by $\V_\T$, and all critical points are fixed,
\item all edges and vertices are invariant, and
\item the restriction of $\G_\T$ to any face of $\T$ is univalent.
\item $N_\mathcal{T}\equiv G_\mathcal{T}$ on $\Lambda_H.$
\end{enumerate}
\end{proposition}

\begin{remark}
Here is an equivalent way of constructing the map $G_\T$ from the Nielsen map $N_\T$. Note that $N_\T$ has $\vert V_\T\vert$ invariant components each of which is a round disk. For such a round disk $D_v$, the restriction $N_\T:\overline{D_v}\setminus\Int{T^0}\to\overline{D_v}$ is topologically conjugate to $\pmb{\rho}_k:\overline{\D}\setminus\Int{\Pi}\to\overline{\D}$ (see Subsection~\ref{nielsen_1_subsec}), where the component of $T^0$ contained in $D_v$ is an ideal $(k+1)$-gon. Using a homeomorphic extension of $\pmb{\mathcal{E}}_k$ to $\overline{\D}$, one can now glue the action of $\overline{z}^k:\D\to\D$ in $D_v$. Clearly, this produces a branched cover of $\mathbb{S}^2$ that agrees with $N_\T$ on $\Lambda_H$, and that is Thurston equivalent to $G_\T$.
\end{remark}

\begin{proposition}\label{cor:orbit_equiv_2}
The branched cover $G_\mathcal{T}$ is orbit equivalent to the reflection group $H_\mathcal{T}$ on $\Lambda_H$. In particular, $\Lambda_H$ is the minimal non-empty $\G_\T$-invariant compact subset of $\rs$.
\end{proposition}
\begin{proof}
This follows from Proposition~\ref{prop:orbit_equiv_1} and the fact that $N_\mathcal{T}\equiv G_\mathcal{T}$ on $\Lambda_H$.
\end{proof}

\begin{remark}
With a slight modification, the construction above can be extended to any polyhedral graph in place of $\T$. 
\end{remark}

\section{Gasket Julia sets}\label{sec:GasketJulia}
Let $f:\rs\to\rs$ be an anti-rational map; i.e., the complex conjugate of a rational map. The \emph{Fatou set} of $f$ is denoted $\mathcal{F}(f)$ and is defined to be the set of $z\in\rs$ so that $\{f^{\circ n}\}_{n=1}^\infty$ is a normal family on some neighborhood of $z$. The \emph{Julia set} of $f$ is defined by $\mathcal{J}(f):=\rs\setminus\mathcal{F}(f)$. It is apparent from the definition that
\[f(\mathcal{F}(f))=\mathcal{F}(f)=f^{-1}(\mathcal{F}(f)),\]
\[f(\mathcal{J}(f))=\mathcal{J}(f)=f^{-1}(\mathcal{J}(f)).\] 
This section shows that $\G_\T$ is realized by an anti-holomorphic map whose Julia set has a natural dynamical equivalence with the Apollonian limit set $\Lambda_H$ discussed above.

\subsection{No obstructions}\label{no_obs_subsec}
W. Thurston's characterization theorem for rational maps \cite{DH} has not yet been extended to anti-rational maps, but the existing techniques can be leveraged by passing to the second iterate.

Two Thurston maps $f$ and $g$ are \emph{equivalent} if there exist two orientation-preserving homeomorphisms $h_0,h_1:(\mathbb{S}^2,P_f)\to(\mathbb{S}^2,P_g)$ so that $h_0\circ f = g\circ h_1$ where $h_0$ and $h_1$ are homotopic relative to $P_f$. The Teichm\"uller space $\text{Teich}(\mathbb{S}^2,P_f)$ associated to a Thurston map $f$ is the set of homeomorphisms $\phi:\mathbb{S}^2\to\rs$ subject to the following equivalence relation: two homeomorphisms $\phi_1$ and $\phi_2$ are equivalent if and only if there is a M\"obius transformation $M$ so that $M\circ \phi_1$ is isotopic to $\phi_2$ rel $P_f$.

It is known that each orientation preserving Thurston map $g$ has an associated pullback map $\sigma_g:\text{Teich}(\mathbb{S}^2, P_g)\to \text{Teich}(\mathbb{S}^2, P_g)$ on Teichm\"uller space. In Douady and Hubbard's proof, the pullback $\sigma_g$ was taken to be the map on Teichm\"uller space induced by the pullback on almost complex structures. To avoid discussion of quasi-regularity of $g$, an equivalent definition of the pullback can be made directly on Teichm\"uller space (see e.g. \cite{BEKP}). It is known that $g$ is equivalent to a rational map if and only if $\sigma_g$ has a fixed point \cite[Proposition 2.1]{DH}. If $g$ has hyperbolic orbifold, the second iterate of $\sigma_g$ is strictly contracting in the Teichm\"uller metric which implies uniqueness of the fixed point \cite[Corollary 3.4]{DH}. Recall that any Thurston map with $|P_g|>4$ has hyperbolic orbifold.

To each orientation reversing Thurston map $f:\mathbb{S}^2\to\mathbb{S}^2$ with $|P_f|\geq 3$ we now show how to define the associated pullback map \[\sigma_f:\text{Teich}(\mathbb{S}^2, P_f)\to \text{Teich}(\mathbb{S}^2, P_f).\] For convenience fix a triple $\{p_1,p_2,p_3\}\subseteq P_f$ and let $\tau\in \text{Teich}(\mathbb{S}^2, P_f)$ be represented by a homeomorphism $\phi:\mathbb{S}^2\to\rs$ so that $\phi(p_1)=0,\ \phi(p_2)=1,\ \phi(p_3)=\infty$. Then $\phi\circ f:\mathbb{S}^2\to\rs$ defines a complex structure on its domain (the restriction of $f$ to $\mathbb{S}^2\setminus f^{-1}(P_f)$ is a cover so charts are immediate there, leaving only finitely many removable singularities). Let $\psi:\mathbb{S}^2\to\rs$ be the unique uniformizing map of this complex structure normalized so that $\phi(p_i)=\psi(p_i),\ i=1,2,3$, and observe that $F_\tau:=\phi\circ f\circ \psi^{-1}$ is an anti-rational map so that the following diagram commutes:
\[ \begin{tikzcd}
\mathbb{S}^2 \arrow{r}{\psi} \arrow[swap]{d}{f} & \rs \arrow{d}{F_\tau} \\
\mathbb{S}^2 \arrow{r}{\phi}& \widehat{\C}
\end{tikzcd}
\]
 Let $\tau'$ be the point in Teichm\"uller space represented by $\psi$ and define $\sigma_f(\tau)=\tau'$. This is well-defined by the homotopy lifting property.

The map $\sigma_f$ has a fixed point if and only if $f$ is equivalent to an anti-rational map using the same argument found in \cite[Proposition 2.3]{DH}. It is also immediate that $\sigma_{f\circ f}=\sigma_f\circ\sigma_f$ where we emphasize that the pullback map on the left is the classical pullback for orientation preserving case as defined in \cite{BEKP}.

\begin{proposition}\label{prop:secondIterate} 
Let $f$ be an orientation reversing Thurston map so that $f\circ f$ has hyperbolic orbifold. Then $f$ is equivalent to an anti-rational map if and only if $f\circ f$ is equivalent to a rational map. Moreover, if $f$ is equivalent to an anti-rational map, the map is unique up to M\"obius conjugacy.
\end{proposition}

\begin{proof} 
Suppose $f$ is equivalent to an anti-rational map. Then there exists $\tau$ so that $\sigma_f(\tau)=\tau$, and $\sigma_{f^2}(\tau)=\sigma_f\circ\sigma_f(\tau)=\tau$ so $f^2$ is equivalent to a rational map.

Now suppose that $f^2$ is equivalent to a rational map. Then $\sigma_{f^2}$ has a unique fixed point. Since $\sigma_{f^2}=\sigma_f\circ\sigma_f$, it follows that $\sigma_f$ either has a two-cycle or a fixed point. A two cycle for $\sigma_f$ would yield two distinct fixed points for $\sigma_{f^2}$, but this is impossible since an iterate of $\sigma_{f^2}$ is contracting in the Teichm\"uller metric. Thus $\sigma_f$ has a fixed point and so $f$ is equivalent to an anti-rational map

Suppose that $f$ is equivalent to an anti-rational map. As mentioned, each fixed point of $\sigma_f$ is a fixed point of $\sigma_{f^2}$. If $\sigma_f$ fixes $\tau_1$ and $\tau_2$, then $\sigma_{f^2}$ fixes $\tau_1$ and $\tau_2$. Since some iterate of $\sigma_{f^2}$ contracts the Teichm\"uller metric $\tau_1=\tau_2$  and there is a unique fixed point of  $\sigma_f$ which implies that $f$ is unique up to M\"obius conjugacy.
\end{proof}

Let $\G_\T$ be one of the orientation reversing Thurston maps from Proposition \ref{prop:AntiThurstonMappingProperties}. 

\begin{proposition}[$G_\T$ is unobstructed]\label{prop:NoObstructionNerveGamma}
The map $\G_\T$ is equivalent to an anti-rational map $\g_\T$ that fixes all of its critical points. The map $\g_\T$ is unique up to M\"obius conjugacy.
\end{proposition}
\begin{proof}
We write $\G_\T$ as $\G$ for simplicity. Using W. Thurston's characterization theorem for rational maps, we first argue that $\G^2$ is equivalent to a rational map by showing that no obstruction exists for this orientation preserving branched cover. 

Assume that there is a Thurston obstruction $\Gamma$ for $\G^2$ to be a rational map; see~\cite{PT} for background and terminology. Such a $\Gamma$ is a curve system in the punctured sphere $\rs\setminus V_\T$, where $V_\T$ is the vertex set of $\T$. By possibly passing to a subsystem, we may further assume that $\Gamma$ is irreducible, i.e., that the corresponding Thurston linear transformation $\G^2_\Gamma$ is irreducible. Finally, by possibly changing $\T$ and $\Gamma$ in their respective homotopy classes relative to $V_\T$, we may assume that $\Gamma$ minimally intersects each edge of $\T$. Indeed, the latter can be seen by choosing the edges of $\T$ and the curves of $\Gamma$ to be geodesics in the punctured sphere $\rs\setminus V_\T$ equipped with the hyperbolic metric.

It follows from Proposition~\ref{prop:AntiThurstonMappingProperties} that the homotopy class of each edge of $\T$ is invariant under $\G$. Thus, each such edge forms an irreducible arc system $\Lambda$ which is forward invariant under $\G^2$ up to isotopy relative to the vertex set $V_\T$, in the terminology of~\cite{PT}. Let $\tilde\Lambda$ denote the component of $G^{-2}(\Lambda)$ that is isotopic to $\Lambda$ relative to $V_\T$, and let $\tilde\T$ be the arc system consisting of all arcs $\tilde\Lambda$. The system $\tilde\T$ forms a triangulation of $\rs$ isotopic to $\T$ relative to $V_\T$. This follows from the fact that $\G$ is univalent away from the vertices $V_\T$.
Similarly, let $\tilde\Gamma$ denote the union of those components of $G^{-2}(\Gamma)$ which are isotopic to elements of $\Gamma$.  Since $\Gamma$ is assumed to be irreducible, we conclude that $\tilde\Gamma$ contains a curve system that is isotopic to  $\Gamma$ relative to $V_\T$.

 It now follows from \cite[Theorem 3.2]{PT} that, as subsets of $\rs$, the arc $\tilde\Gamma$ may not intersect $\G^{-2}(\Lambda)\setminus\tilde \Lambda$, for any edge $\Lambda$ of $\T$. Indeed, since $\Gamma$ minimally intersects each $\Lambda$, the first case of \cite[Theorem 3.2]{PT} means that $\Gamma\cap \Lambda=\emptyset$ as sets. Thus we have that $\G^{-2}(\Gamma)\cap \G^{-2}(\Lambda)=\emptyset$, and, in particular, $\tilde\Gamma\cap (\G^{-2}(\Lambda)\setminus\tilde \Lambda)=\emptyset$. If, in the second case, $\Gamma\cap \Lambda\neq\emptyset$, then 
\cite[Theorem 3.2, 2(a)]{PT} gives that $\tilde \Gamma\cap (G^{-2}(\Lambda)\setminus\tilde\Lambda)=\emptyset$.

Since $\tilde \Gamma\cap (G^{-2}(\Lambda)\setminus\tilde\Lambda)=\emptyset$ for each edge $\Lambda$ of $\T$, we conclude that $\tilde \Gamma$ cannot intersect the set $\G^{-2}(\T)\setminus\tilde\T$. Indeed, if $\tilde\Gamma$ did intersect $\G^{-2}(\T)\setminus\tilde\T$, then it would intersect $\G^{-2}(\T)$, i.e., there would exist an edge $\Lambda$ of $\T$ such that $\tilde\Gamma\cap\G^{-2}(\Lambda)\neq\emptyset$. Since $\tilde \Gamma\cap (G^{-2}(\Lambda)\setminus\tilde\Lambda)=\emptyset$, we would conclude that $\tilde \Gamma$ can only intersect $\G^{-2}(\T)$ at points of $\tilde\T$, and the claim follows.

We now argue that the closure of  $\G^{-2}(\T)\setminus\tilde \T$ contains a connected graph containing the postcritical set of $\G^2$, which is $V_\T$. 
Since $\G$ is a covering map over $\rs\setminus V_\T$, it is enough to check this statement for the original triangulation $\T$ whose edges are invariant under $\G$, rather than the isotopic triangulation whose edges are the geodesics as above. In this case $\tilde\T=\T$, and since $\G^{-1}(\T)\subseteq\G^{-2}(\T)$, it is enough to check that $\G^{-1}(\T)\setminus\T$ is a connected graph containing $\V_\T$.
Let $f\in F$ be a face of the graph $\T$. Since $\T$ is a triangulation, the closure of $\T\setminus\overline{f}$ is a connected set containing the three vertices of $f$. Recall that $\G$ is univalent on each face (Proposition \ref{prop:AntiThurstonMappingProperties}), so $\G^{-1}(\T\setminus\overline{f})\cap f$ is also a connected set containing the three vertices of $f$. Carrying out this procedure for all faces and taking the union, it is shown that the closure of $\G^{-1}(\T)\setminus\T$ is connected and contains $V_\T$. 

The curve system $\tilde\Gamma$ contains an isotopic copy relative to $V_\T$ of a Thurston obstruction $\Gamma$. Therefore, $\tilde\Gamma$ must separate the postcritical set $V_\T$ without intersecting the closure of the connected set $\G^{-2}(\T)\setminus\tilde\T$ that contains $V_\T$. This is impossible, and thus no Thurston obstructions exist for $\G^2$, and so $\G^2$ is equivalent to a rational map. The conclusion that $\G$ is anti-rational follows from Proposition \ref{prop:secondIterate}. Thurston equivalence preserves local degrees and postcritical dynamics, so $\g_\T$ must also fix all of its critical points.

The uniqueness statement will follow from Proposition \ref{prop:secondIterate} once it is seen that $G^2$ has hyperbolic orbifold. Since $G(P_G)=P_{G}$, the equation $|P_{ G\circ G}|=|P_{\G}|$ holds. Since $\T$ was assumed to have at least 4 vertices, it follows from Proposition \ref{prop:AntiThurstonMappingProperties} that $\G^2$ has at least 4 postcritical points. If $|P_{\G}|>4$, it is immediate that $\G^2$ has hyperbolic orbifold because $|P_{\G^2}|>4$. If $|P_{\G}|=4$, then $\G \circ \G$ has postcritical set consisting of four fixed critical points, and by direct computation, $\G\circ \G$ has hyperbolic orbifold. 
\end{proof}

\subsection{Isotopic nerves}\label{nerves_isotopic_subsec}

Suppose an anti-rational map $g$ fixes each of its critical points. If a Fatou component contains a critical point, it is called a \emph{critical Fatou component}. Let $U$ be a critical Fatou component. Adapting the classical B\"ottcher theorem, there is a B\"ottcher coordinate $\phi:U\to\mathbb{D}$ so that $\phi\circ g=g_d\circ\phi$, where $g_d(z)=\bar{z}^d$. By Carath\'eodory's theorem, the map $\phi^{-1}$ extends continuously to a semiconjugacy of $\overline\D$ onto $\overline U$. A \emph{B\"ottcher ray of angle $\theta_0$} for $\phi$ is defined to be the subset of $\overline{U}$ of the form $\phi^{-1}(re^{i\theta_0})$ where $r\in[0,1]$.  Note that the ray of angle $\frac{j}{d+1}$ is $g_d$-invariant for $j=0,...,d$, and so the Fatou component $U$ has $d+1$ corresponding $g$-invariant rays. A \emph{ray connection} of $g$ is the union of two B\"ottcher rays (either in the same or different Fatou components) whose intersection contains a point in the Julia set. Two distinct critical Fatou components are said to \emph{touch} if there is a ray connection between their corresponding critical points.

There is a general result of Pilgrim that can be used to prove the existence of ray connections \cite[Theorem 5.13]{KMPthesis} and show that only finitely many ray connections exist (though the precise number of ray connections is not specified). The proof of the following lemma adapts Pilgrim's argument to our specific setting.

 Let $h_0:\sphere\to\rs$ represent the unique fixed point of the pullback map $\sigma_G$ on Teichm\"uller space, where $h_0$ is normalized to carry the postcritical set of $\G_\T$ (this is the same as the set of critical points) to that of $\g_\T$.

\begin{lemma}\label{Lem:ExistsRay}
Let $\alpha$ be an edge in $\T$. Then the arc $h_0(\alpha)$ is isotopic (rel the postcritical set) to a ray connection of $g_\T$. Moreover, there is a lift of $h_0(\alpha)$ under $g_\T$ that is isotopic to $h_0(\alpha)$.
\end{lemma}
\begin{proof}
Recall from the construction of the orientation reversing Thurston map $G_\T$ that an edge in $\T$ is a geodesic arc $\alpha$ connecting the centers of two circles that are tangent to each other, and that $\alpha$ is invariant under $G_\T$ (see Proposition \ref{prop:AntiThurstonMappingProperties}).
Define  the sequence of orientation preserving  homeomorphisms $\{h_i:\sphere\to\rs\}_{i=1}^{\infty}$ inductively by the pullback equation $h_{i-1}\circ G_\T = g_\T\circ h_{i},\ i=1, 2,\dots$. Each $h_i$ is likewise normalized so that it carries the postcritical set  of $\G_\T$ to that of $\g_\T$.

Let $\beta_i:=h_i(\alpha)$. Note that $\beta_i$ does not intersect postcritical points of $g_\T$, other than its endpoints, because $\alpha$ does not intersect any of the postcritical points of $G_\T$ other than its end points. Since $\alpha$ is a lift of itself under $G_\T$, it follows that $\beta_{i+1}$ is a lift of $\beta_i$ under $g_\T$ for $i\geq 0$, and that $\beta_{i+1}$ is isotopic relative to the postcritical set to $\beta_i$ (though possibly tracing a different arc). Denote by $U_1$ and $U_2$ the two Fatou components that contain the endpoints of $\beta_i$ for all $i$. Applying an isotopy to $\beta_0$ relative to the postcritical set, it may be assumed that $\beta_0\cap U_1$ and $\beta_0\cap U_2$ each consist of exactly one component which is a B\"ottcher ray. The forward invariance of $U_1$ and $U_2$ imply that for each $i$, the sets $\beta_i^1:=\beta_i\cap U_1$ and $\beta_i^2:=\beta_i\cap U_2$ each consist of a single B\"ottcher ray. Define the sequence of compact sets $K_i:=\beta_i\setminus(\beta_i^1\cup \beta_i^2)$, and observe that $g_\T(K_{i+1})=K_i$ and $g_\T|_{K_{i+1}}$ is injective for each $i$.

A hyperbolic rational map is uniformly expanding on compact subsets of $\rs$ that do not intersect the postcritical set \cite[\S 19]{Milnor}. Thus the sequence of compact sets $K_i$ has diameter converging to zero as $i\to\infty$, and therefore some subsequence of $\{\beta_i\}$ has Hausdorff limit $\beta$ that is a ray connection between the critical points in $U_1$ and $U_2$. 

Since every postcritical point outside of $U_1\cup U_2$ is contained in a Fatou component and hence has positive distance from the Julia set, it follows that $K_i$ has a definite positive distance from each such point for all $i$. Thus the limiting ray connection $\beta$ is in the same isotopy class as $\beta_i$ for all $i$.
\end{proof}

\begin{lemma}\label{lem:raysIsoToEdges}
Each ray connection of $g_\T$ that is not a loop is isotopic to an edge of $h_0(\T)$. 
\end{lemma}

\begin{proof}
Let $x$ and $y$ be two distinct critical points of $g_\T$.  Let $\gamma_1$ be an arc with endpoints $x$ and $y$ and let $\gamma_2$ satisfy the same properties. In this proof, all homotopies are considered in $\rs\setminus P_{g_\T}$ rel the endpoints of the arc. Denote by $\iota(\gamma_1,\gamma_2)$ the minimum of the quantity $|\gamma_1'\cap\gamma_2'|$ for all $\gamma_i'$ in the homotopy class of $\gamma_i$ rel $\{x,y\}$, $i=1,2$. If $\gamma_1$ and $\gamma_2$ are both ray connections, it is evident that $\iota(\gamma_1,\gamma_2)\leq 3$. Let $\{\gamma_k\}_{k=1}^\infty$ be a sequence of arcs with endpoints $\{x,y\}$ so that no two arcs are pairwise homotopic rel $\{x,y\}$. Then for any integer $M>0$, there exist indices $l$ and $m$ so that $\iota(\gamma_l,\gamma_m)>M$. 

Let $\beta$ be a ray connection of $g_\T$ with endpoints distinct. Then, for each $i\geq0$, $g_\T^{\circ i}(\beta)$ is also a ray connection. 
Therefore, by the previous paragraph, there exist integers $i\geq0$ and $j>0$ (taken to be minimal) so that $g_\T^{\circ i}(\beta)$ and $g_\T^{\circ i+j}(\beta):=\beta'$ are isotopic rel endpoints. Then $\alpha:=h_0^{-1}(\beta')$ is an arc with distinct endpoints in the vertex set of $\T$. After applying a homotopy, we may assume that $\alpha$ intersects the disks $D_1,D_2$ containing its endpoints radially. Some lift of $\alpha$ under $G_\T^{\circ j}$ is isotopic to $\alpha$. Denote by $\alpha_k$ some choice of a lift of $\alpha$ under $G_\T^{\circ kj}$ that is isotopic to $\alpha$. We argue as in the previous lemma.  The map $G_\T$ is expansive by Lemma \ref{L:TrianglesShrink}, and the nontriviality of the $\alpha_k$ implies that $D_1$ and $D_2$ must touch. By construction, $D_1$ and $D_2$ touch in at most one point which is contained in an edge in $\T$, and thus $\{\alpha_k\}_{k=1}^\infty$ converges in the Hausdorff topology to this edge of $\T$. Since edges lift to themselves under $\G_\T$ it follows that $j=1$. 

Each edge of $\T$ has exactly one $G_\T$-lift with the property that both endpoints are critical points, namely the edge itself. Thus $i=0$ and $h_0(\alpha)=\beta'$  is isotopic to $\beta$.
\end{proof}

\begin{corollary}\label{cor:LocDegFatCount}
A Fatou component of $g_\T$ with fixed critical point of multiplicity $d$ touches exactly $d+1$ invariant Fatou components distinct from itself.
\end{corollary}

\begin{proof}
The $h_0$ preimage of such a critical point is the endpoint of $d+1$ edges in the triangulation $\T$. The conclusion follows from Lemmas \ref {Lem:ExistsRay} and \ref{lem:raysIsoToEdges}.
\end{proof}

\begin{lemma}\label{lem:rayIntersect} 
No point is contained in the boundary of three or more critical Fatou components.
\end{lemma}

\begin{proof}
Suppose $z$ lies in the closure of three distinct Fatou components. By the Jordan curve theorem, there are at most two points that lie in the closure of the three Fatou components. Thus the forward orbit of $z$ consists of at most two points.

 If $z$ is fixed by $g_\T$, three invariant Fatou components touch at a fixed point. This is incompatible with the (orientation reversing) local linearization at that fixed point. A similar argument applies if $g_\T(z)$ is fixed.
 
The final case to consider is that $z$ is in a two-cycle. Then there are three invariant critical Fatou components $U, V,$ and $W$ so that $\{z,g_\T(z)\}\subset\overline{U}\cap\overline{V}\cap\overline{W}$. There must be some critical point $x$ in a complementary component of $\overline{U\cup V\cup W}$. Without loss of generality, we may assume that $x$ is separated from $W$ by $\overline{U\cup V}$. But then there are two non-homotopic ray connections connecting the critical points in $U$ and $V$, which is contrary to Lemma \ref{lem:raysIsoToEdges} and the fact that $\T$ contains no graph two-cycles.
\end{proof}

From Lemma \ref{lem:raysIsoToEdges} it is known that ray connections must be invariant up to isotopy, but we prove a stronger statement.

\begin{lemma}\label{Lem:RaysInv}
Each ray connection that is not a loop is invariant under $g_\T$.
\end{lemma}

\begin{proof}
Let $U$ be a fixed critical Fatou component and recall that $U$ has a B\"ottcher coordinate $\phi:U\to\mathbb{D}$ so that $\phi\circ g_{\T}=g_d\circ\phi$, where $g_d(z)=\bar{z}^d$. By Carath\'eodory's theorem, the map $\phi^{-1}$ extends continuously to a semiconjugacy of $\overline\D$ onto $\overline U$. For a subset $E$ in $\partial U$, abusing notations, we denote by $\phi(E)$ the full preimage of $E$ under the extended semiconjugacy $\phi^{-1}$.

Let $U_1,...,U_{d+1}$ be the critical Fatou components that touch $U$, as guaranteed in Corollary \ref{cor:LocDegFatCount}. Define $K_i\subset\partial\mathbb{D}$ to be the compact set $\phi(\overline{U}\cap\overline{U_i})$ for all $i$. Note that for all $j\neq i$, the set $K_i$ is not separated by $K_j$ in $\partial\mathbb{D}$ since $U, U_i,$ and $U_j$ are pairwise disjoint.  Furthermore, $g_d(K_i)\subset K_i$ for all $i$. Denote by $H(K_i)$ the smallest closed circular arc in $\partial\mathbb{D}$ that contains $K_i$, and denote by $|K_i|$ the length of $H(K_i)$. Any two distinct sets of the form $H(K_i)$ have disjoint interior by planarity of the corresponding Fatou components and have disjoint boundary by Lemma \ref{lem:rayIntersect}.

Each set $H(K_i)$ is now shown to contain at least one fixed point of the power map $g_d$. If $|K_i|\geq \frac{2\pi}{d+1}$, the conclusion is immediate because of the equal distribution of the $d+1$ fixed points of $g_d$ on the circle. If $|K_i|< \frac{2\pi}{d+1}$, then the expansion of $g_d$ and the forward invariance of $K_i$ implies that $|K_i|=0$. Thus $K_i$ consists of a single point which must be fixed since $g_d(K_i)\subset K_i$. 

There are $d+1$ fixed points of $g_d$ in the circle  and $d+1$ distinct $K_i$ so each $H(K_i)$ contains exactly one fixed point of $g_d$. If $|K_i|< \frac{2\pi}{d+1}$ it was just argued that $K_i$ is a singleton. It will now be shown that this is always the case. Suppose that $|K_i|\geq \frac{2\pi}{d+1}$. Then $K_i$ contains a fixed point $z_0$ as well as a point $z_1$ that has circular distance from $z_0$ contained in $(\frac{\pi}{d+1},\frac{2\pi}{d+1})$. It follows that $g_d(z_1)$ is separated from $z_0$ by another fixed point which must also then be contained in $H(K_i)$.  Thus $H(K_i)$ contains two fixed points of $g_d$ which is a contradiction. Each $K_i$ has been shown to be a singleton.

Under the semiconjugacy $\phi^{-1}$, these $d+1$ fixed points are carried to $d+1$ distinct fixed points of $g_\T$ in $\partial U$ by Lemma \ref{lem:rayIntersect}. The fact that $K_i$ is a singleton implies that there is a unique B\"ottcher ray in $U$ landing at $\phi^{-1}(K_i)$. Similarly there is a unique B\"ottcher ray in $U_i$ landing at $\phi^{-1}(K_i)$. The union of these two B\"ottcher rays forms a ray connection which is the unique connection between the two critical points. Thus the ray connection is forward invariant.
\end{proof}

\begin{lemma}\label{Lem_NoLoops}
No ray connection of $g_\T$ is a loop.
\end{lemma}

\begin{proof}
Let $U$ be a fixed critical Fatou component, and suppose that  $\beta_1$ and $\beta_2$ are B\"ottcher rays that terminate at a common endpoint  $z$ in the Julia set, i.e. $\beta_1\cup\beta_2$ forms a loop.

First, suppose $z$ is fixed and $\beta_1$ and $\beta_2$ are invariant. Suppose the local degree of the critical point in $U$ is $-d$. Then by Corollary \ref{cor:LocDegFatCount}, there are exactly $d+1$ other fixed critical points that are connected by a single ray connection to the critical point in $U$. By Lemma \ref{lem:rayIntersect}, the intersection of the $d+1$ rays with $U$ is a collection of $d+1$ distinct B\"ottcher rays. Each of the rays is invariant by Lemma \ref{Lem:RaysInv} and so the collection of B\"ottcher rays must include $\beta_1$ and $\beta_2$. Thus there is an invariant ray connection that contains $\beta_1$ and terminates at the critical point of another Fatou component. But three distinct invariant arcs are incompatible with the local linearization of the anti-holomorphic map $g_\T$ at $z$. Thus $\beta_1=\beta_2$. 

Now suppose $\beta_1$ and $\beta_2$ are invariant after some finite number of iterations of $g_\T$. By the previous paragraph and the fact that there are no critical points in the Julia set (hence the iterates at $z$ are locally univalent), it follows that $\beta_1=\beta_2$.

The case that $\beta_1$ is invariant after a finite number of iterates, and $\beta_2$ is not invariant after a finite number of iterates is incompatible with the local linearization at $z$ and the B\"ottcher coordinate on $U$ (for a similar argument see \cite[Lemma 18.12]{Milnor}).

The final case to consider is that $\beta_1$ and $\beta_2$ are distinct B\"ottcher rays  that are not (eventually) invariant under $g_\T$. As before, let $z$ denote the common endpoint of $\beta_1$ and $\beta_2$ in the Julia set. Recall that the restriction of $g_\T$ to $U$ is conformally conjugate to $g_d(z)=\bar{z}^d$ on the open unit disk. Let $\alpha_1,\alpha_2$ each be a radius of the unit disk that is not eventually invariant under iteration of $g_d$. Since $g_d$ is a power map, there is some iterate $n$ so that $g_d^{\circ n}(\alpha_1)\setminus\{0\}$ and $g_d^{\circ n}(\alpha_2)\setminus\{0\}$ are separated by the union of two invariant radii. Thus under iteration, $\beta_1$ and $\beta_2$ are separated in $U$ by two invariant rays $\gamma_1$ and $\gamma_2$ in $U$. Without loss of generality, we replace $\beta_1$ and $\beta_2$ by their separated iterates. Also we may assume that $\gamma_1$ and $\gamma_2$ are neighbors in $U$, in the sense that their union has a complementary component in $U$ that does not intersect any invariant rays.  Moreover, $\gamma_1$ and $\gamma_2$ are subsets of ray connections to other Fatou components by Corollary \ref{cor:LocDegFatCount} and Lemma \ref{Lem:RaysInv}. Thus there are invariant critical Fatou components $U_1$ and $U_2$ with closures containing exactly one endpoint of $\overline{\gamma_1}$ and $\overline{\gamma_2}$ respectively. 

Suppose first that $U_1$ and $U_2$ are contained in distinct complementary components of $\overline{U}$ (see Figure \ref{F:impossible}). But $U_1$ and $U_2$ must touch since they arise from neighboring invariant rays, so $U$, $U_1$, and $U_2$ touch at $z$. This contradicts Lemma \ref{lem:rayIntersect}. Suppose next that $U_1$ and $U_2$ are contained in the same complementary component of $\overline{U}$. But then one of $\gamma_1$ or $\gamma_2$ contains $z$ in its closure, so $z$ is fixed by $g_\T$ and is the landing point of an invariant B\"ottcher ray. Once again, one can use the linearization of $g_\T$ at $z$ and the existence of the B\"ottcher coordinate on $U$ to conclude that $\beta_1$ and $\beta_2$ are invariant under finitely many iterates which contradicts the hypothesis. Thus $\beta_1=\beta_2$.
\end{proof}

\begin{figure}[h]
\begin{tikzpicture}
    \node[anchor=south west,inner sep=0] at (0,0) {\includegraphics[width=6.6cm]{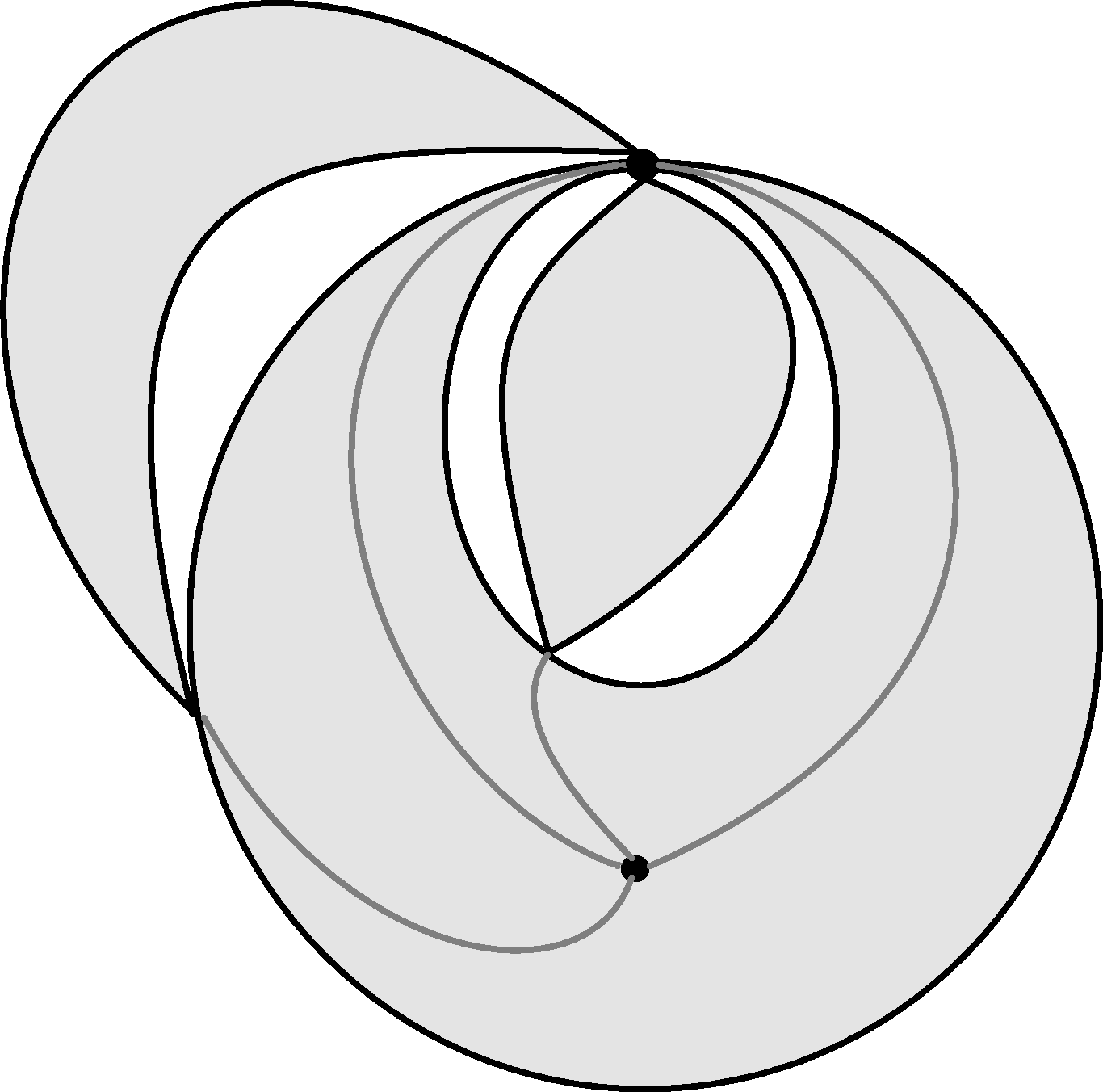}};
\node at (2,2.74) {$\beta_1$};
\node at (3.6,2) {$\gamma_2$};
\node at (3,0.6) {$\gamma_1$};
\node at (5.95,2.84) {$\beta_2$};
\node at (5,1) {$U$};
\node at (1.43,5.8) {$U_1$};
\node at (3.9,4.1) {$U_2$};
\node at (4.1,5.8) {$z$};
\end{tikzpicture}
\caption{Sample impossible configuration from the final case of the proof of Lemma \ref{Lem_NoLoops}.}\label{F:impossible}
\end{figure}

The \emph{nerve} of $g_\T$ is defined to be the graph whose vertex set is the set of fixed critical points of $g_\T$ and edge set is given given by the collection of all ray connections both of whose endpoints are fixed critical points. We do not assume that the nerve has a finite number of edges, or even that it is an embedded graph. The nerve is said to be \emph{naturally embeddable} if the intersection of each pair of ray connections is a subset of the vertex set. If the nerve is naturally embeddable, we consider it to be an embedded graph given by the obvious embedding.

\begin{proposition}[Nerve of $g_\T$]\label{Prop_NervesIsotopic}
The nerve of $g_\T$ is a naturally embeddable graph in $\rs$ that is isotopic to $\T\subset\sphere$. Moreover, each vertex of the nerve is fixed by $g_\T$ and each edge is invariant.
\end{proposition}

\begin{proof} 
Suppose that two ray connections $\beta$ and $\beta'$ intersect but are not identical. It is impossible for $\beta\cap\beta'$ to be the union of a disjoint B\"ottcher ray and a point because this would imply the existence of a loop contrary to Lemma  \ref{Lem_NoLoops}. Moreover $\beta\cap\beta'$ may not be a single B\"ottcher ray because this configuration would imply that three Fatou components touch at the same point, contrary to Lemma \ref{lem:rayIntersect}. Thus $\beta\cap\beta'$ is a subset of the vertex set and the nerve is naturally embeddable.

The homeomorphism $h_0:\sphere\to\rs$ maps the embedded triangulation $\T$ to another triangulation. Since $h_0$ is a global orientation preserving homeomorphism, $h_0(\T)$ is isotopic to $\T$ (here the isotopy is not rel vertices).
Recall from Lemma \ref{Lem:ExistsRay} that each individual edge of $h_0(\T)$ is isotopic relative postcritical set to an edge in the nerve of $g_\T$, and each edge of the nerve arises in this way by Lemma \ref{lem:raysIsoToEdges}. Thus there is a global isotopy that carries $h_0(\T)$ to the nerve. It follows that $\T$ is isotopic to the nerve of $g_\T$.

Invariance of vertices was a simple consequence of Thurston equivalence in Proposition \ref{prop:NoObstructionNerveGamma}. Invariance of edges is the conclusion of Lemma \ref{Lem:RaysInv}.
\end{proof}

In conjunction with Proposition \ref{prop:NoObstructionNerveGamma} we have the following analogue of the Circle Packing theorem.
\begin{corollary}
For any triangulation $\T$ of the sphere, there exists an anti-rational map that fixes each of its critical points and has nerve that is naturally embeddable and isotopic to $\T$.
\end{corollary}

\subsection{Promoting Thurston equivalence to conjugacy}\label{promoting_Thurston_subsec}

Now it is shown that the Thurston equivalence between $\G_\T$ and $g_\T$ can be promoted to a conjugacy on the Julia set using a pullback argument. The existence of the conjugacy has been shown generally in \cite[Theorem 4.4]{Kam}, \cite[Corollary 1.2]{BD}, and similar results have been shown in different contexts, see e.g. \cite[Corollary 1.2]{CPT}, \cite[Theorem~11.1]{BM17}.

\begin{theorem}[Equivariant homeomorphism between limit set and Julia set] \label{Thm_IdentificationOfLimitAndJulia}
There is a homeomorphism $h:\Lambda_H\to\mathcal{J}(\g_\T)$. Moreover, $h \circ \G_\T=\g_\T \circ h$, and $h$ extends to an orientation preserving homeomorphism of the sphere.
\end{theorem}

\begin{proof} We have just shown that the nerve of $g_{\T}$ is isotopic to the nerve of the circle packing $\mathcal C_\T$, in particular there is an orientation preserving  homeomorphism $h_0$ of $\rs$ carrying the nerve of $G_{\T}$ to the nerve of $g_\T$. Moreover, by possibly changing $h_0$ in the same isotopy class relative to the critical points of $G_\T$, we may and will assume that the map $h_0$ takes each fixed connected component of the complement of $\Lambda_H$ onto a fixed Fatou component of $\mathcal J$, and still takes the nerve of $G_{\T}$ to the nerve of $g_\T$.

A \emph{dynamical interstice of generation} $n\in\N$ is the closure of a connected component obtained by removing from the sphere the closures of the Fatou components of $g_\T$ of generation at most $n$. A \emph{group interstice of generation} $n$ is defined similarly, but with respect to the complementary components of $\Lambda_H$. Both, dynamical and group interstices are topological triangles whose vertices are touching points of two complementary components of $\mathcal J$ and $\Lambda_H$, respectively.
 The above assumption implies that $h_0$ takes a group interstice of generation 0 onto a dynamical interstice of generation 0.

Let $h_i,\ i\ge 1$, be a lift of $h_{i-1}$, namely, 
\begin{equation}\label{E:Conj}
g_\T \circ h_i=h_{i-1}\circ G_\T.
\end{equation}
The existence of such a lift follows from the fact that $G_\T$ and $g_\T$ are equivalent maps and they are topological coverings outside their respective fixed critical points. 
Since $g_{\T}$ may not have critical points outside of the nerve, it follows from Proposition~\ref{Prop_NervesIsotopic} that $g_{\T}$ is univalent on each dynamical interstice $\Delta$ of generation 0. Moreover, it takes each such interstice $\Delta$ onto the closure of its complement minus the three Fatou components whose boundaries intersect $\Delta$.

Applying equation~\eqref{E:Conj} inductively, we conclude that the map $h_i$ takes each group interstice of generation $i$ onto a dynamical interstice of generation $i$. Moreover, the same equation gives that for each $j\ge i$, the map $h_j$ takes each group interstice $\Delta$ of generation $i$ onto the same dynamical interstice $h_i(\Delta)$ of generation $i$. 

Let $P_i,\ i\ge0$, denote the set of points common to two group interstices of generation $i$. Since the nerves of $G_\T$ and $g_\T$ are preserved by these maps, respectively, the map $h_0$ takes the nerve of  $G_\T$ to the nerve of $g_\T$, and the maps $h_1$ and $h_0$ are isotopic relative to the critical points of $G_\T$, we conclude from~\eqref{E:Conj} that $h_1$ agrees with $h_0$ on $P_0$.
Arguing inductively, equation~\eqref{E:Conj} gives that $h_j=h_i$ on $P_i$ for all $0\le i\le j$. Let 
$$
P_\infty=\cup_{i=0}^\infty P_i,
$$
and define
$$
h(x)=\lim_{i\to\infty} h_i(x),\quad x\in P_\infty.
$$
This limit exists by the preceding discussion and the set $P_\infty$ is dense in $\Lambda_H$ by Lemma~\ref{L:TrianglesShrink}. 
Equation~\eqref{E:Conj} gives
$$
g_\T \circ h(x)=h\circ G_\T(x),\quad x\in P_\infty.
$$

We now show that the family $\{h_i\}_{i=0}^\infty$ is equicontinuous on $\Lambda_H$. Indeed, let $\epsilon>0$ be arbitrary. We choose $N\in\N$ such that all dynamical interstices of $\mathcal J$ of generation $N$ have diameter at most $\epsilon$. 
It is possible to choose such an $\epsilon$ since $g_\T$ is expanding. 
If $\Delta$ is a group interstice of $\Lambda_H$ of generation $N$, we define the \emph{height} of $\Delta$ to be the smallest of distances (in the geodesic distance on the sphere) from each vertex of $\Delta$ to the side opposite to this vertex. 
Let $\delta>0$ be the smallest height among all group interstices of $\Lambda_H$ of generation $N$. Such a $\delta$ exists because there are only finitely many group interstices of generation $N$.

Let $x,y\in\Lambda_H$ be such that $d(x,y)<\delta$, where $d$ denotes the geodesic metric on the sphere. Let $l$ be the geodesic in the sphere of length $d(x,y)$ that joins $x$ and $y$. Let $D$ be a complementary disk of $\Lambda_H$ of generation at most $N$. By replacing the intersection of $l$ with each such disk $D$ by the shortest arc on the boundary $\partial D$ with the same end points, we conclude that there exists a path $l'$ in $\Lambda_H$ that connects $x$ and $y$ and whose length is at most $\pi\delta$. Note that the path $l'$ cannot be self-intersecting. We claim that there exists an absolute constant $C$ such that $l'$ intersects at most $C$ group interstices of generation $N$. If this is the case, we have $d(h_i(x), h_i(y))<C\epsilon$ for all $i\ge N$, and the equicontinuity follows. 

Now, if $x$ and $y$ are in the same generation $N$ group interstice or in two generation $N$ interstices that share a vertex, the claim is immediate. Assume that this is not the case, and let $k$ be the number of generation $N$ group interstices that intersect $l'$, excluding the interstices that contain $x$ and $y$. If $\Delta$ is one of the $k$ such group interstices, then $l'$ must contain two distinct vertices of $\Delta$. 
This follows from the observation that $l'$ is not self-intersecting.
Therefore, the length of $l'$ is at least $k\delta$. Since the length of $l'$ is at most $\pi\delta$, we conclude that $k\le 3$. Thus $C=5$ works and the proof of equicontinuity of $\{h_i\}_{i=0}^\infty$ is complete. 

The equicontinuity of  $\{h_i\}_{i=0}^\infty$ and the density of $P_\infty$ in $\Lambda_H$ implies that the map $h$ has a unique continuous extension to all of $\Lambda_H$. We continue to denote this extension by $h$. Moreover, the map $h$ has to satisfy
$$
g_\T \circ h(x)=h\circ G_\T(x),\quad x\in \Lambda_H,
$$ 
i.e., $h$ semi-conjugates $\Lambda_H$ to $\mathcal J$. 

The map $h$ is a surjective map from $\Lambda_H$ to $\mathcal J$ because $h(P_\infty)$ is dense in $\mathcal J$, which follows from hyperbolicity of $g_\T$. 

We now argue that $h$ is also injective. Let $x$ and $y$ be two distinct points in $\Lambda_H$. From Lemma~\eqref{L:TrianglesShrink} we know that there exists $i\in\N$ such that $x$ and $y$ belong to two disjoint group interstices $\Delta_x$ and $\Delta_y$ of generation $i$. Then the dynamical interstices $h_i(\Delta_x)$ and $h_i(\Delta_y)$ are disjoint. As stated above, we also have that $h_j(\Delta_x)=h_i(\Delta_x)$ and $h_j(\Delta_y)=h_i(\Delta_y)$ for all $j\ge i$. 
Thus, $h(\Delta_x\cap P_\infty)\subseteq h_i(\Delta_x)$ and 
$h(\Delta_y\cap P_\infty)\subseteq h_i(\Delta_y)$
are disjoint. Therefore, by taking closures and using the continuity of $h$, we get that $h(\Delta_x\cap\Lambda_H)\subset h_i(\Delta_x)$ and 
$h(\Delta_y\cap\Lambda_H)\subset h_i(\Delta_y)$ are disjoint, and hence $h(x)\neq h(y)$.

The fact that $h$ extends to an orientation preserving homeomorphism of the sphere follows from Lemma~\ref{L:Extension}.
\end{proof}

The previous theorem tells us that the anti-rational map $g_\mathcal{T}$ is intimately related to the reflection group $H_\mathcal{T}$. 

\begin{corollary}[Conjugacy between anti-rational map and Nielsen map]\label{prop:group_anti_rat_conjugate}
The anti-rational map $g_\mathcal{T}:\mathcal{J}(g_\mathcal{T})\to\mathcal{J}(g_\mathcal{T})$ is topologically conjugate to the Nielsen map $N_\mathcal{T}:\Lambda_H\to \Lambda_H$.
\end{corollary}
\begin{proof}
Follows from Theorem~\ref{Thm_IdentificationOfLimitAndJulia} and the fact that $G_\mathcal{T}\equiv N_\mathcal{T}$ on $\Lambda_H$.
\end{proof}

\begin{figure}[ht!]
\centerline{\includegraphics[width=60mm]{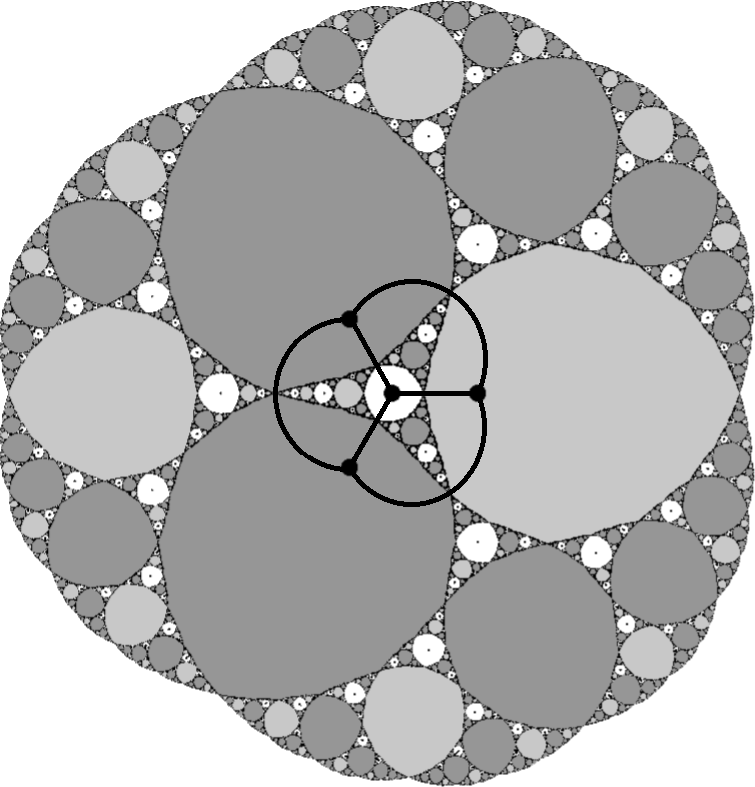}}
\caption{The Julia set corresponding to the tetrahedron with nerve superimposed.  The  anti-rational map is given by $\overline{f}$ where $f(z)=\frac{3z^2}{2z^3+1}$. (The image of $J_f$ appeared in \cite{BEKP}.) }
\label{pic:JuliaRays}
\end{figure}

\begin{proposition}[Global conjugacy between $G_\T$ and $g_\T$]\label{P:Cjgc}
There exists a homeomorphism $h$ of the whole Riemann sphere $\rs$ such that
\begin{equation}\label{E:Cjgc}
h\circ G_\T=g_\T\circ h
\end{equation} 
  on $\rs$.
\end{proposition}
\begin{proof}
In Theorem~\ref{Thm_IdentificationOfLimitAndJulia} we proved the existence of a homeomorphism $h$ such that \eqref{E:Cjgc} holds true on $\Lambda_H$. This theorem also states that $h$ has a homeomorphic extension to the whole sphere, but in general this extension does not have to satisfy \eqref{E:Cjgc} outside of $\Lambda_H$. We now show that a homeomorphic extension to $\rs\setminus\Lambda_H$ that satisfies  \eqref{E:Cjgc} exists.

Let $D$ be a component of $\Omega_H=\rs\setminus\Lambda_H$ that is fixed by $G_\T$,  and let $U$ be the Fatou component of $g_\T$ that corresponds to $D$ under the map $h$ from Theorem~\ref{Thm_IdentificationOfLimitAndJulia}. Furthermore, let $\Phi\colon\overline{D}\to\overline\D$ be the map from Section~\ref{Sec:RoundGasketsNoninvertible}
such that 
\begin{equation}\label{E:limitset}
\Phi\circ G_\T=g_d\circ\Phi
\end{equation}
on $\overline{D}$,
where $g_d(z)=\bar{z}^d$, and let
  $\phi\colon \overline{U}\to\overline{\D}$ be the B\"ottcher coordinate of $\overline{U}$. Note that since $U$ is a Jordan domain, such $\phi$ exists. Moreover, from conjugation of $G_\T$ and $g_\T$ on the boundary of $D$ we know that 
 \begin{equation}\label{E:Juliaset}
  \phi\circ g_T=g_d\circ \phi
  \end{equation}
   on $\overline{U}$. Also, the map $\phi$ is unique up to post-composition with a rotation by an angle which is an integer multiple of $2\pi/(d+1)$. Therefore, we may assume that $\phi$ is selected in such a way that for each fixed point $z_i$ of $G_\T$ on the boundary of $D$, we have $\phi(h(z_i))=\Phi(z_i)$. Putting together equations \eqref{E:Cjgc} on the boundary of $D$ and \eqref{E:limitset}, \eqref{E:Juliaset}, we conclude that
   $\phi(h(z))=\Phi(z)$ holds for all $z$ on the boundary of $D$. 
  
  We now extend $h$ from  Theorem~\ref{Thm_IdentificationOfLimitAndJulia} to $D$ using the formula
  $$
  h=\phi^{-1}\circ\Phi.
  $$
  In such a way we obtain a homeomorphic extension of $h$ from the boundary of each fixed component $D$ inside $D$ that satisfies \eqref{E:Cjgc} in $D$. 
  
  If $D$ is a non-fixed by $G_\T$ component of $\Omega_H$, let $k\in\N$ be the smallest integer such that $G_\T^k(D)$ is fixed by $G_\T$. From Theorem~\ref{Thm_IdentificationOfLimitAndJulia} we know that the component $h(G_\T^k(D))=g_\T^k(h(D))$ must be fixed by $g_\T$. Note that $k$ must also be the smallest integer with this property, and hence $g_\T^k$ is univalent on $h(D)$. We define the extension of $h$ inside such a $D$ by
  $$
  h=g_\T^{-k}\circ\phi^{-1}\circ\Phi\circ G_\T^k,
  $$
  where the branch of $g_\T^{-k}$ is chosen so that 
  $$
  g_\T^{-k}(g_T^k(h(D)))=h(D).
  $$
  From the above, such an extension satisfies \eqref{E:Cjgc} in $D$.
  
  Since we now were able to extend $h$ into each component of $\Omega_H$ and diameters of $D$ as well as the corresponding Fatou components $U$ go to 0, we conclude that such extensions paste into a global homeomorphism, and the proof is complete.
\end{proof}

For example, the Julia set where $\T$ is the tetrahedron is shown in Figure \ref{pic:JuliaRays}.

\begin{remark}
Recall from Subsection~\ref{nielsen_2_subsec} that if the dual graph $\check{\T}$ of the triangulation $\T$ is Hamiltonian, then the group $\Gamma_{\mathcal{T}}$ (which is the index two Kleinian subgroup of the reflection group $H_\mathcal{T}$) is obtained as a limit of a sequence of quasi-Fuchsian deformations of $\pmb{\Gamma}_d$ that pinch a suitable collection of simple closed non-peripheral geodesics on the $(d+1)$-times punctured spheres $\D/\pmb{\Gamma}_d$ and $(\widehat{\C}\setminus\overline{\D})/\pmb{\Gamma}_d$. Moreover, these geodesics lift by $\pmb{\Gamma}_d$ to the universal covers $\D$ and $\widehat{\C}\setminus\overline{\D}$ (respectively) giving rise to a pair of $\pmb{\rho}_d$-invariant geodesic laminations such that the quotient of $\mathbb{T}$ by pinching the endpoints of the leaves of both these laminations produces a topological model of the limit set $\Lambda_H$.

Each of these two laminations can be viewed as an equivalence relation on $\mathbb{T}$. Pushing forward these two laminations by the topological conjugacy $\pmb{\mathcal{E}}_d$ between $\pmb{\rho}_d$ and $g_d$, where we recall that $g
_d(z)=\overline{z}^d$, we obtain two $g_d$-invariant formal rational laminations in the sense of \cite{Kiwi}. An anti-holomorphic version of \cite[Theorem~1.1]{Kiwi} implies that these laminations are admitted by two critically fixed degree $d$ anti-polynomials. It is not hard to see that the topological mating of these two critically fixed anti-polynomials (see \cite[Definition~4.1]{MP} for the definition of topological mating of two polynomials, these notions carry over to anti-polynomials in the obvious way) is a degree $d$ orientation reversing branched cover of $\mathbb{S}^2$ that is topologically conjugate to $G_\T$. In light of Proposition~\ref{P:Cjgc}, we now conclude that the topological mating of these two critically fixed anti-polynomials is topologically conjugate to $g_\T$ such that the conjugacy can be chosen to be conformal in the interior of the Fatou sets of the anti-polynomials. In other words, $g_\T$ is a \emph{geometric mating} of these two anti-polynomials (see \cite[Definition~4.4]{MP}). In particular, the quotient of $\mathbb{T}$ by the above-mentioned $g_d$-invariant rational laminations yields a topological model of $\mathcal{J}(g_\mathcal{T})$, and the quotient map from $\mathbb{T}$ onto $\mathcal{J}(g_\mathcal{T})$ semi-conjugates $g_d:\mathbb{T}\to\mathbb{T}$ to $g_\mathcal{T}:\mathcal{J}(g_\mathcal{T})\to\mathcal{J}(g_\mathcal{T})$. 
\end{remark}

\section{Gasket Julia set quasisymmetries}\label{sec:JuliaQuasisym}

Let $\Delta$ be a complementary component of the nerve of the Fatou set of $\g_\T$. The restriction of $\g_\T$ to $\overline{\Delta}$ is evidently univalent, and $\g_\T(\overline{\Delta})=\rs\setminus\Delta$. We associate to $\Delta$ a dynamically defined map $\rho_\Delta:\rs\to\rs$ as follows:

\[\rho_\Delta(z) =
\begin{cases}
\g_\T(z) & \text{if } z\in \overline{\Delta},\\
\g_\T^{-1}(z)  & \text{if } z\notin \overline{\Delta}, \text{ where } \g_\T^{-1}:\rs\setminus \overline{\Delta}\to \Delta.
\end{cases}
\]
Note that the map $\rho_\Delta$ is not continuous on all of $\rs$, but  restricts to a continuous involution on the Julia set $\mathcal{J}(\g_\T)$. 

\begin{lemma} \label{lem:PWDynQS} For any $\Delta$ as above, $\rho_\Delta$ restricted to $\mathcal{J}(\g_\T)$ is quasisymmetric.
\end{lemma}

\begin{proof}
Let $U_1, U_2$, and $U_3$ be the three Fatou components of $g_\T$ that intersect the boundary of $\Delta$. 
From the definition of $\rho_{\Delta}$, we immediately conclude that this map restricted to $\Delta\cup (\rs\setminus\overline{\Delta})$ is anti-conformal. Hence, $\rho_{\Delta}$ is anti-conformal in 
$$
\left(\Delta\setminus\overline{U_1\cup U_2\cup U_3}\right)\cup \left(\rs\setminus\overline{\Delta\cup(U_1\cup U_2\cup U_3)}\right).
$$ 
Let $\phi_i\colon \overline{U_i}\to\overline{\D}$, be a B\"ottcher coordinate of $\overline{U_i},\ i=1,2,3$, that conjugates the map $g_\T|_{\overline{U_i}}$ to $g_{d_i}(z)=\bar z^{d_i}$ on $\overline{\D}$. The restriction $\rho_i$ of $\phi_i\circ\rho_{\Delta}\circ\phi_i^{-1},\ i=1,2,3$, to the boundary circle of $\D$ is then given piecewise as follows. It is $z\mapsto \bar z^{d_i}$ on an arc between two successive fixed points of $g_{d_i}$, and $z\mapsto \bar z^{1/d_i}$ on the complement of the arc, where the branch of $\bar z^{1/d_i}$ is selected so that the resulting piecewise map $\rho_i$ is a homeomorphism of the circle. Each such $\rho_i,\ i=1,2,3$, is a bi-Lipschitz orientation reversing homeomorphism of the circle. Therefore each complex conjugate map $\overline{\rho_i},\ i=1,2,3$, is an orientation preserving bi-Lipschitz map, and, in particular, it is quasisymmetric. According to the Ahlfors--Beurling theorem
it has a quasiconformal extension to $\D$. Thus each map $\rho_i,\ i=1,2,3$, has an anti-quasiconformal extension to $D$, i.e., the complex conjugation of each extension is quasiconformal. Conjugating back using the B\"ottcher coordinates $\phi_i,\ i=1,2,3$,
we conclude that the map $\rho_{\Delta}$ has an anti-quasiconformal extension into each of the Fatou components $U_1, U_2, U_3$. This way we obtain a global homeomorphism of $\rs$ that is anti-quasiconformal outside the union of three boundaries $\partial U_1\cup\partial U_2\cup\partial U_3$. Since $g_\T$ is a hyperbolic rational map, each boundary $\partial U_i,\ i=1,2,3$, is a quasicircle. According to~\cite[Theorem~4 and Proposition~9]{Y}, the set $\partial U_1\cup\partial U_2\cup\partial U_3$ is quasiconformally removable. 

Therefore, the restriction of $\rho_\Delta$ to the Julia set $\mathcal{J}(\g_\T)$ extends to an  anti-quasiconformal map of $\rs$. Since the classes of quasiconformal and quasisymmetric maps of $\rs$ coincide (see, e.g., \cite{H}), we conclude that the restriction $\rho_\Delta|_{\mathcal{J}(\g_\T)}$ must be quasisymmetric. 
\end{proof}

Recall that Theorem \ref{Thm_IdentificationOfLimitAndJulia} gives a homeomorphism $h:\Lambda_H\to\mathcal{J}(\g_\T)$. There is an obvious induced isomorphism
\[h_*:\homeo\to \text{Homeo}(\mathcal{J}(\g_{\T}))\] defined by $\xi\mapsto h\xi h^{-1}$ for $\xi\in\homeo$.

\begin{theorem}[Gasket Julia quasisymmetries]\label{thm:homeosAsQuasisym} For a reduced triangulation $\T$,  \[\text{Homeo}(\mathcal{J}(\g_{\T}))= \QS(\mathcal{J}(\g_{\T}))\] and so there is an isomorphism
\[h_*:\homeo\to \QS(\mathcal{J}(\g_{\T})).\]
\end{theorem}

\begin{proof}
The splitting $\homeo=\text{Aut}^{\T}(\rs) \ltimes H$ from Theorem \ref{thm:groupSplits} and surjectivity of $h_*$ implies that every element of $\text{Homeo}(\mathcal{J}(\g_{\T}))$ is a composition of an element of $h_*(H)$ and an element of $h_*(\text{Aut}^{\T}(\rs))$. To prove the theorem, it suffices to show that all such elements are quasisymmetric.

Since $h^{-1}\rho_\Delta h$ is a homeomorphism of $\Lambda_H$ that acts invariantly on three of the generation zero disks, it follows from Lemma \ref{L:ThreeDisks} that $h^{-1}\rho_\Delta h$ is a generator of $H$. Thus $ h_*(h^{-1}\rho_\Delta h)=\rho_\Delta$ is a generator of $h_*(H)$ for each $\Delta$. Lemma \ref{lem:PWDynQS} asserts that $\rho_\Delta$ is a quasisymmetry. Thus, every element of $h_*(H)$ is a composition of quasisymmetries, and hence a quasisymmetry itself. Thurston rigidity implies that every element of $h_*(\text{Aut}^{\T}(\rs))$ is either a M\"obius or anti-M\"obius symmetry of $\mathcal{J}(g_{\T})$ and thus a quasisymmetry.
\end{proof}

To conclude the section, we note that unreduced triangulations are still very much of interest, even though our theory does not directly apply to compute their symmetry group. For example, the two graphs in Figure \ref{pic:Unreduced} are realized by anti-holomorphic maps of degree $-3$ and $-5$ respectively, with Julia set homeomorphic to the classical Apollonian gasket. Evidently neither map is an iterate of the other since their degrees are prime. The procedure easily generalizes to produce infinitely many anti-holomorphic maps with Julia set homeomorphic to the classical Apollonian gasket.

\section{A quasiregular model}\label{affine_model_sec}

Throughout this section, $\mathcal{T}$ will denote the tetrahedral triangulation of the topological $2$-sphere. 

The goal of this section is to construct an orientation reversing anti-quasiregular map $\mathcal{G}$ on a tetrahedron which is piecewise affine outside the fixed Fatou components and quasiconformally conjugate to the critically fixed cubic anti-rational map $g\equiv g_\T$ of the Riemann sphere (see Proposition~\ref{prop:NoObstructionNerveGamma}). It is worth pointing out that the main result of this section provides us with an alternative construction of the anti-rational map $g$ which does \emph{not} use Thurston's characterization of rational maps.

We consider a tetrahedron; i.e., a polyhedron composed of four congruent (equilateral) triangular faces, six straight edges, and four vertices. The graph $\mathcal{T}$ defining this triangulation can be identified with the union of the edges $AB, AC, BC, AD, BD,$ and $CD$ including the vertices (see Figure~\ref{tetra_subdivision} (left)).

Let us denote the mid-points of the edges of $\mathcal{T}$ by $E, F, \cdots, J$ (see Figure~\ref{tetra_subdivision} (right)). Recall that there is a circle packing on $\widehat{\C}$ whose nerve is isomorphic to $\mathcal{T}$. In the current setting, the role of the round disks of this circle packing will be played by the open \emph{caps} with triangle boundaries $EGJ$, $EFH$, $FGI$, and $HIJ$, and which contain the vertices $A, B, C$, and $D$ respectively. We will denote them by $\widehat{EGJ}$, $\widehat{EFH}$, $\widehat{FGI}$, and $\widehat{HIJ}$. 

The complementary components of the union of the closures of the caps are equilateral triangles each of which is contained in a face of the triangulation. We denote the closures of these equilateral triangles by $\Delta EFG, \Delta FHI, \Delta EHJ, \Delta GIJ$, and call them \emph{interstices}. They play the role of the interstices of the corresponding circle packing.

\begin{figure}[ht!]
\begin{tikzpicture}
\node[anchor=south west,inner sep=0] at (-3,0) {\includegraphics[width=0.49\textwidth]{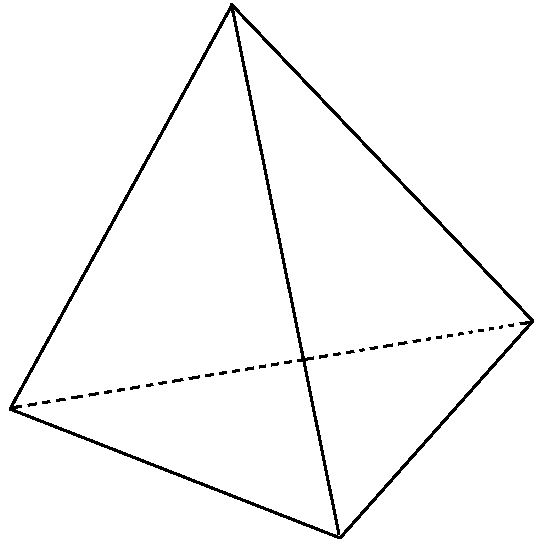}};
\node[anchor=south west,inner sep=0] at (4,-0.2) {\includegraphics[width=0.48\textwidth]{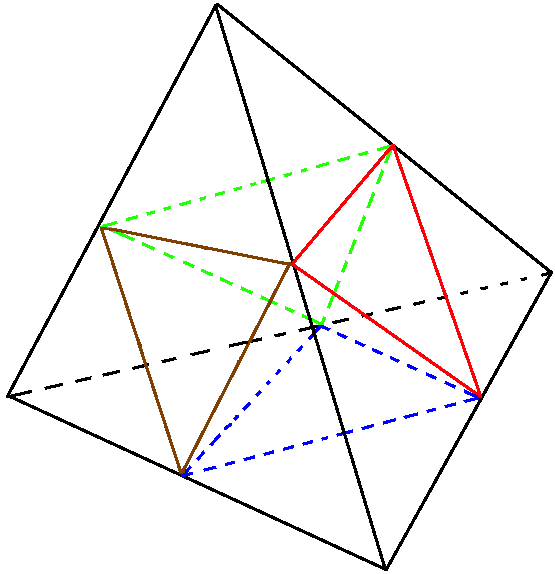}};
\node at (-3,1.2) {$A$};
\node at (-0.2,6.4) {$B$};
\node at (0.8,-0.2) {$C$};
\node at (3.2,2.25) {$D$};
\node at (4,1.4) {$A$};
\node at (6.4,6.4) {$B$};
\node at (8.2,-0.4) {$C$};
\node at (10.2,2.7) {$D$};
\node at (7.64,2.3) {\begin{tiny}$J$\end{tiny}};
\node at (5.8,0.7) {\begin{tiny}$G$\end{tiny}};
\node at (4.8,3.6) {\begin{tiny}$E$\end{tiny}};
\node at (7.44,3.24) {\begin{tiny}$F$\end{tiny}};
\node at (9.4,1.6) {\begin{tiny}$I$\end{tiny}};
\node at (8.5,4.6) {\begin{tiny}$H$\end{tiny}};
\end{tikzpicture}
\caption{Left: The tetrahedral triangulation of the topological $2$-sphere. Right: Partition of the tetrahedral surface into caps and interstices.}
\label{tetra_subdivision}
\end{figure}

The tetrahedron is naturally endowed with an affine structure via identification of its faces with equilateral triangles in the plane. More precisely, the tetrahedron can be folded from the union of four equilateral triangles in the plane as depicted in Figure~\ref{tetra_net}; the four triangles are bounded by bold edges. This configuration of triangles is called a \emph{net} of the tetrahedron. We identify the faces of the tetrahedron with the corresponding equilateral triangles in this net. Note that the vertices $D_1, D_2, D_3$ all correspond to the same vertex $D$ on the tetrahedron. Similarly, as the edge $AD$ (on the tetrahedron) is obtained by folding $AD_1$ and $AD_2$ (on the net), two points on $AD_1$, $AD_2$ that are equidistant from $A$ correspond to the same point on $AD$. The same is true for the pairs of edges $BD_1, BD_3$, and $CD_2, CD_3$. Moreover, the interstices on the tetrahedron correspond to the equilateral triangles $\Delta EFG$, $\Delta EH_1J_1$, $\Delta GI_1J_2$, and $\Delta FH_2I_2$ in the net. We will use the tetrahedron and its net (which gives an affine structure to the tetrahedron) interchangeably.

\begin{figure}[ht!]
\begin{tikzpicture}
\node[anchor=south west,inner sep=0] at (-4,0) {\includegraphics[width=1\textwidth]{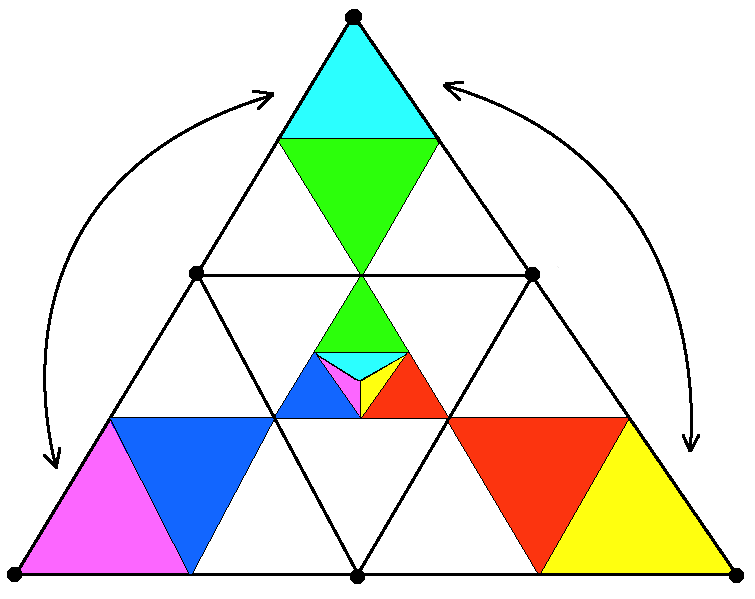}};
\node[anchor=south west,inner sep=0] at (-2.4,-1.6) {\includegraphics[width=0.75\textwidth]{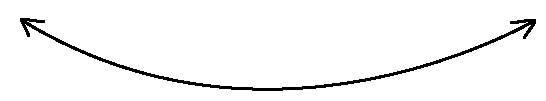}};
\node at (-1.2,5.5) {\begin{Large}$A$\end{Large}};
\node at (5.45,5.5) {\begin{Large}$B$\end{Large}};
\node at (2,-0.1) {\begin{Large}$C$\end{Large}};
\node at (2,10.2) {\begin{Large}$D_1$\end{Large}};
\node at (-4,-0.1) {\begin{Large}$D_2$\end{Large}};
\node at (8.6,-0.1) {\begin{Large}$D_3$\end{Large}};
\node at (-2.7,3) {\begin{Large}$J_2$\end{Large}};
\node at (0.32,7.7) {\begin{Large}$J_1$\end{Large}};
\node at (0.2,3.3) {$G$};
\node at (2.45,5.24) {$E$};
\node at (3.88,3.24) {$F$};
\node at (5,-0.1) {\begin{Large}$I_2$\end{Large}};
\node at (-0.8,-0.1) {\begin{Large}$I_1$\end{Large}};
\node at (7,3.1) {\begin{Large}$H_2$\end{Large}};
\node at (3.8,7.8) {\begin{Large}$H_1$\end{Large}};
\node at (1,4.2) {$K$};
\node at (3.1,4.2) {$L$};
\node at (2,2.8) {$M$};
\node at (2.06,3.88) {\begin{tiny}$N$\end{tiny}};
\end{tikzpicture}
\caption{The union of the equilateral triangles $ABC, ABD_1, ACD_2,$ and $BCD_3$ form a net of the tetrahedron. The color-coding illustrates the action of the quasiregular map $\mathcal{G}$ on the tetrahedron. The attracting basin of the $D$-vertex, represented by $D_1, D_2, D_3$ in the figure, is the union of the three attached triangles (blue, yellow, violet).}
\label{tetra_net}
\end{figure}

We now proceed to define our desired quasiregular map $\mathcal{G}$ on the tetrahedron. This will be done in two steps.

\noindent\textbf{Step I (Defining $\mathcal{G}$ on the interstices).} As in the construction of $G_\T$ in Section~\ref{Sec:RoundGasketsNoninvertible}, we first define the map on the interstices. For definiteness, let us work with the interstice $\Delta EFG$ (the definition on the other interstices will be symmetric). We subdivide $\Delta EFG$ into four congruent equilateral triangles $\Delta EKL$, $\Delta KGM$, $\Delta LMF$, and $\Delta KLM$, and further subdivide $\Delta KLM$ into three congruent triangles by joining the vertices $K, L, M$ to the barycenter $N$ (as shown in Figure~\ref{tetra_net}). Note that the three triangles $KLN$, $KMN$, $LMN$ are not equilateral. 

Let us now map $\Delta EKL$ onto $\Delta EJ_1H_1$ by an orientation reversing affine map. This can be achieved by reflecting $\Delta EKL$ in the line $AB$, and then scaling it to match with $\Delta EJ_1H_1$. This defines an anti-conformal map $\mathcal{G}$ on $\Delta EKL$. Now extend $\mathcal{G}\vert_{KL}$ affinely to the triangle $\Delta KLN$ such that it maps onto the triangle $\Delta J_1H_1D_1$ in an orientation reversing manner (since $\Delta J_1H_1D_1$ is equilateral and $\Delta KLN$ is not, the map $\mathcal{G}$ is not  anti-conformal on $\Delta KLN$). This completes the definition of $\mathcal{G}$ on the quadrilateral $EKNL$. We can now extend $\mathcal{G}$ to the quadrilaterals $GKNM$ and $FLNM$ by reflecting the previously defined map $\mathcal{G}$ in the line segments $KN$ and $NL$ such that the extended map sends the triangles $\Delta GKM$, $\Delta KMN$, $\Delta LMN$, and $\Delta FLM$ affinely onto the triangles $\Delta GJ_2I_1$, $\Delta J_2I_1D_2$, $\Delta H_2I_2D_3$, and $\Delta FH_2I_2$ (respectively). It is easy to see that this definition is compatible with the identifications of the edge pairs $(D_1J_1,D_2J_2)$, $(D_1H_1,D_3H_2)$, and $(D_2I_1,D_3I_2)$. Hence, we obtain a well-defined orientation reversing map $\mathcal{G}$ on the interstice $\Delta EFG$ of the tetrahedron. Moreover, $\mathcal{G}$ is a piecewise defined (anti-)similarity (in particular, anti-conformal) on $\Delta EFG\setminus \Delta KLM$. 

We now repeat the above procedure on the other three interstices $\Delta IFH, \Delta EHJ$, and $\Delta GJI$. Note that the above definition of the map $\mathcal{G}$ on the four interstices completely determines it on the boundaries of the four caps $\widehat{HIJ}, \widehat{EGJ}, \widehat{FGI}, \widehat{EFH}$ (around the vertices $D, A, C, B$, respectively).

\noindent\textbf{Step II (Defining $\mathcal{G}$ on the caps).} Let us first focus on the cap around the vertex $D$ (this is $\widehat{HIJ}$ on the tetrahedron). The definition of $\mathcal{G}$ on the interstices implies that $\mathcal{G}$ is an expanding double covering of the quasicircle $HIJ$ on the tetrahedron (in fact, it doubles the distance between a pair of nearby points). 

Since $\mathcal{G}:HIJ\to HIJ$ is an orientation reversing expanding double covering of a topological circle, it is easy to see that there is a unique orientation preserving topological conjugacy $\psi$ between $\mathcal{G}:HIJ\to HIJ$ and $\overline{z}^2:\mathbb{T}\to\mathbb{T}$ that sends the fixed points $H, I$ and $J$ of $\mathcal{G}$ to the fixed points $1, \omega$, and $\omega^2$ of $\overline{z}^2$ (where $\omega$ is a primitive third root of unity). Note that $\mathcal{G}:HIJ\to HIJ$ is piecewise affine, and the left and right multipliers of $\mathcal{G}$ at each fixed point are equal. It follows that the map $\mathcal{G}$ satisfies the distortion estimate of \cite[Lemma~19.65]{Lyu} with a constant $C=1$. The proof of \cite[Proposition~19.64]{Lyu} now applies to show that the conjugacy $\psi$ is a quasisymmetry.

Finally, as the cap $\widehat{HIJ}$ is quasisymmetrically equivalent to $\D$, the Ahlfors-Beurling extension theorem provides us with a quasiconformal extension of $\psi$ that maps $\widehat{HIJ}$ to the unit disk $\D$, still denoted by $\psi$. We now extend $\mathcal{G}$ to the cap $\widehat{HIJ}$ as $\psi^{-1}\circ\overline{z}^2\circ\psi$. 
Performing the same ``surgery'' on the remaining three caps, we obtain the desired anti-quasiregular map $\mathcal{G}$ of degree $-3$ on the tetrahedron. We note that $\mathcal{G}$ is \emph{not} affine on the caps. 

For an orientation reversing map $f$, the pull-back of a Beltrami coefficient $\mu$ under $f$ is defined as (see~\cite[Exercise~1.2.2]{BF14})
$$
f^*(\mu(z))=\overline{\left(\frac{\partial f/\partial z+\mu({f(z)})\partial\bar f/\partial z}{\partial f/\partial\bar z+\mu({f(z)})\partial \bar f/\partial{\bar z}}\right)}.
$$

\begin{proposition}\label{prop:tetra_model_affine}
The map $\mathcal{G}$ on the tetrahedron is quasiconformally conjugate to an anti-rational map $R$ on $\hat{\C}$.\end{proposition}
\begin{proof}
We will construct a $\mathcal{G}$-invariant Beltrami coefficient on the tetrahedron, and apply the Measurable Riemann Mapping Theorem to straighten $\mathcal{G}$ to an anti-rational map.

Denote by $\mu_0$ the standard complex structure on $\D$. Pulling $\mu_0$ back by $\psi$, we get a $\mathcal{G}$-invariant Beltrami coefficient $\mu$ on the caps of the tetrahedron. We now extend $\mu$ to the tetrahedron by pulling back the previously defined Beltrami coefficient $\mu$ (on the caps) by the iterates of $\mathcal{G}$, and setting it equal to zero outside the iterated pre-images of the caps. 
Since $\mathcal{G}$ is a piecewise (anti-)similarity outside the first pre-images of the caps, the infinitesimal ellipse field defined by $\mu$ on the caps is only distorted (i.e., the dilatation is changed) under the first pullback. Moreover, these inverse branches of $\mathcal{G}$ are piecewise affine. It follows that $\mu$ is a $\mathcal{G}$-invariant Beltrami coefficient on the tetrahedron with $\vert\vert\mu\vert\vert_\infty<1$. By the Measurable Riemann Mapping Theorem, there exists a quasiconformal homeomorphism from the tetrahedron to the Riemann sphere $\hat{\C}$ that pulls back the standard complex structure on $\hat{\C}$ to the one defined by $\mu$ on the tetrahedron. Therefore, this quasiconformal map conjugates $\mathcal{G}$ to an anti-rational map $R$ on $\hat{\C}$.
\end{proof}

\begin{corollary}\label{apollo_rat_map_formula_cor}
Up to M{\"o}bius conjugacy, $R$ can be chosen to be $g(z)=\frac{3\overline{z}^2}{2\overline{z}^3+1}$ (whose existence was demonstrated in Proposition~\ref{prop:NoObstructionNerveGamma}).
\end{corollary}
\begin{proof}
Note that since $R$ is quasiconformally conjugate to $\mathcal{G}$, it follows that $R$ is a critically fixed (in particular, postcritically finite) anti-rational map. By construction, its nerve is isotopic to $\T$. 

In Proposition~\ref{prop:NoObstructionNerveGamma}, we constructed the postcritically finite anti-rational map $g$ (of degree $-3$) with four fixed critical points. By Proposition~\ref{Prop_NervesIsotopic}, the nerve of $g$ is also isotopic to $\T$. By a M{\"o}bius map, we can send two of the distinct fixed critical points of $g$ to $0$ and $1$, and the only other pre-image of $0$ (under $g$) to $\infty$. Then, $g$ takes the form $$g(z)=\frac{(2a+b+3)\overline{z}^2}{1+a\overline{z}+b\overline{z}^2+(a+2)\overline{z}^3},$$ for some $a,b\in\C$. A direct computation using the fact that $g$ fixes its two other distinct critical points $\left(-\frac12\pm\frac12\sqrt{\frac{a-6}{a+2}}\right)$ now shows that $a=b=0$; i.e. $$g(z)=\frac{3\overline{z}^2}{2\overline{z}^3+1}.$$

Since $R$ and $g$ have isotopic nerves, it follows that they are Thurston equivalent. By Thurston rigidity, the anti-rational maps $R$ and $g$ are M{\"o}bius conjugate; i.e., $R$ is M{\"o}bius conjugate to $g(z)=\frac{3\overline{z}^2}{2\overline{z}^3+1}$ (see Figure~\ref{pic:JuliaRays} for the dynamical plane of $g$).
\end{proof}

\begin{remark}
Since $\mathcal{G}$ is quasiconformally conjugate to an anti-rational map, it has a unique measure of maximal entropy $\nu=\nu(\mathcal{G})$ such that the measure-theoretic entropy of $\mathcal{G}$ with respect to $\nu$ is $\ln{3}$ (the degree of $\mathcal{G}$ is $-3$). Moreover, $\nu$ is supported on the ``Julia set'' of $\mathcal{G}$, and is the Hausdorff measure of the Julia set. Since $\mathcal{G}$ is a piecewise similarity with a constant derivative $2$ on the Julia set, the Lyapunov exponent of $\mathcal{G}$ with respect to $\nu$ is $\ln{2}$. A classical formula relating Lyapunov exponent, Hausdorff dimension and entropy of a measure (see \cite{Manning}) now yields that the Hausdorff dimension of the measure $\nu$ is $\ln{3}/\ln{2}$, which is in accordance with the fact that the Julia set of $\mathcal{G}$ is the union of four affine copies of the Sierpi\'nski gasket, and hence its Hausdorff dimension is equal to $\ln{3}/\ln{2}$. 
\end{remark}

\begin{remark}
Note that while all four critical points of the cubic \emph{anti-rational} map $g$ are simple and fixed, it follows from \cite[Theorem~1]{CGNPP} that there is no cubic \emph{rational} map with this property. 
\end{remark}

\section{David surgery}
\label{sec:David}
Let $h:=\pmb{\mathcal{E}}_2^{-1}$ be the orientation preserving homeomorphism of the unit circle that conjugates the dynamics of $g_2(z)=\bar z^2$ to the dynamics of the Nielsen map $\pmb{\rho}_2$ associated to the ideal triangle $\Pi$ with vertices at the cube roots of unity (see Subsection~\ref{nielsen_1_subsec}). Namely,
$$
h\circ g_2=\pmb{\rho}_2\circ h.
$$
Since both maps, $g_2$ and $\pmb{\rho}_2$, fix the cube roots of unity, we may assume that so does $h$ (this defines $h$ uniquely). It follows that such $h$ commutes with the rotation $z\mapsto e^{2\pi i/3}z$ as well as the complex conjugation $z\mapsto\bar z$.  
Our first goal is to show that the homeomorphism $h$ has a homeomorphic David extension inside the unit disk $\mathbb D$. 

Recall, that a map $H\colon U\to V$  between two domains in $\mathbb C$ or in the Riemann sphere is called \emph{David} if $H$ is in the Sobolev class $W^{1,1}_{\rm loc}$ and there exist constants 
$C,\alpha, K_0>0$ with
$$
\sigma\{z\in U\colon K_H(z)\geq K\}\leq Ce^{-\alpha K},\quad K\geq K_0.
$$  
Here $\sigma$ denotes the Lebesgue or spherical measure and $K_H$ is the distortion function of $H$ given by
$$
K_H=\frac{1+|\mu_H|}{1-|\mu_H|},
$$
with
$$
\mu_H=\frac{\partial H/\partial\bar z}{\partial H/\partial z}
$$
being the Beltrami coefficient of $H$. Both, $\mu_H$ and $K_H$ are defined almost everywhere. 
The reader may consult~\cite{D88} for background on David maps.
Note that a map $H\colon U\to V$ in $W^{1,1}_{\rm loc}$ is David  if and only if there exist constants $M, \alpha, \epsilon_0>0$ such that its Beltrami coefficient $\mu_H$ satisfies
\begin{equation}\label{E:DavidIneq}
\sigma\{z\in U\colon |\mu_H(z)|\geq1-\epsilon\}\leq Me^{-\alpha /\epsilon},\quad \epsilon\leq \epsilon_0.
\end{equation}

To show that a David extension of $h$ exists, we will show that the scalewise distortion function $\rho_{\tilde h}(t)$ of $h$, in the sense of S.~Zakeri, satisfies 
\begin{equation}\label{E:SWD}
\rho_{\tilde h}(t)=O\left
(\log\frac{1}{t}\right),\quad t\to0+.
\end{equation}
The scalewise distortion is defined as follows. Let $\tilde h$ be the lift of the map $h$ under the covering map $x\mapsto e^{2\pi i x}$. The map $\tilde h$ is then an orientation preserving homeomorphism of the real line such that $\tilde h(x+1)=\tilde h(x)+1$. We may and will assume that $\tilde h(0)=0$. The distortion function $\rho_{\tilde h}(x,t)$ is defined to be
\begin{equation}\label{E:Distortion}
\rho_{\tilde h}(x,t)=\max\left\{\frac{\tilde h(x+t)-\tilde h(x)}{\tilde h(x)-\tilde h(x-t)},\frac{\tilde h(x)-\tilde h(x-t)}{\tilde h(x+t)-\tilde h(x)}\right\},
\end{equation}
for $x\in\mathbb R$ and $t>0$. 
The scalewise distortion is
$$
\rho_{\tilde h}(t)=\sup_{x\in\mathbb R}\rho_{\tilde h}(x,t).
$$
Since $h$ commutes with the rotation by angle $2\pi/3$, it is enough to take the above supremum over $x\in[0,1/3]$. 

To find the asymptotics of $\rho_{\tilde h}(t)$ as $t\to0+$, it is convenient to replace the map $\tilde h$ in a neighborhood of $[0,1/3]$ by the homeomorphism 
$$
h_{\rm new}(x)=\phi\circ\tilde h(x/3)
$$ 
of $[0,1]$, where $\phi$ is a bi-Lipschitz map of a neighborhood of $[0,1/3]$ onto a neighborhood of $[0,1]$ with $\phi([0,1/3])=[0,1]$, defined  as follows. There exists a M\"obius transformation $m$ that takes the upper half-plane onto the unit disk and such that $m(0)=1, m(1)=e^{2\pi i/3}$, and $m(\infty)=e^{4\pi i/3}$.  We now define 
$$
\phi(x)=m^{-1}(e^{2\pi i x})
$$ 
for $x\in(-1/6,1/3+1/6)$. The map $\phi$ is $K$-bi-Lipschitz for some $K>1$, and therefore we have the following relation
$$
\frac{1}{K^2}\rho_{\tilde h}(3t)\leq \rho_{h_{\rm new}}(t)\leq K^2\rho_{\tilde h}(t/3),
$$
for all $t>0$ small enough.
Therefore, $\rho_{\tilde h}$ satisfies~\eqref{E:SWD} if and only if 
\begin{equation}\label{E:SWD2}
\rho_{h_{\rm new}}(t)=O\left
(\log\frac{1}{t}\right),\quad t\to0+,
\end{equation}
where 
$$
\rho_{h_{\rm new}}(t)=\sup_{x\in[0,1]}\rho_{h_{\rm new}}(x,t),
$$
and the distortion function $\rho_{h_{\rm new}}(x,t)$ is defined as in~\eqref{E:Distortion} with $\tilde h$ replaced by $h_{\rm new}$, and $t>0$ small enough.

Note that the Nielsen map 
$$
\theta:\R\cup\{\infty\}\to\R\cup\{\infty\},\qquad 
\theta(t)= \left\{\begin{array}{ll}
                    -t \qquad t\in\left[-\infty,0\right],  \\
                     \frac{t}{2t-1} \qquad t\in\left[0,1\right],\\
                     2-t \qquad t\in\left[1,+\infty\right],
                                          \end{array}\right. 
$$
associated to the ideal triangle in the upper half-plane with vertices at $0, 1$, and $\infty$, maps $[0,\frac12]$ (respectively, $[\frac12,1]$) to $[-\infty,0]$ (respectively, to $[1,+\infty]$). Composing $\theta$ with a rotation that brings $\theta([0,\frac12])=[-\infty,0]$ (respectively, $\theta([\frac12,1])=[1,+\infty]$) back to $[0,1]$ defines the orientation reversing double covering 
 $$
 \tau:[0,1)\to[0,1),\qquad \displaystyle \tau(t)= \left\{\begin{array}{ll}
                     \frac{2t-1}{t-1}\ ({\rm mod}\ 1) \qquad t\in\left[0,\frac12\right),  \\
                     \frac{1-t}{t}\hspace{2.5mm} ({\rm mod}\ 1) \qquad t\in\left[\frac12,1\right).
                                          \end{array}\right. 
$$

The advantage of passing to the map $h_{\rm new}(x)=m^{-1}\left(h\left(e^{2\pi i\frac{x}{3}}\right)\right)$ is that, by construction, it conjugates the dynamics of the orientation reversing doubling map $m_{-2}:[0,1)\to[0,1)$ given by
$$
m_{-2}(x)=\left\{\begin{array}{ll}
                     -2x+1\ ({\rm mod}\ 1) \qquad x\in\left[0,\frac12\right),  \\
                     -2x+2\ ({\rm mod}\ 1) \qquad x\in\left[\frac12,1\right),
                                          \end{array}\right.
$$
to the dynamics of $\tau$. Therefore, each dyadic point in $[0,1]$ corresponds under the map $h_{\rm new}$ to the point of tangency of the corresponding Ford circle~\cite{F38} with the real line. Indeed, due to the conjugation, points of the dyadic subdivisions of $[0,1]$ correspond under $h_{\rm new}$ to points in $[0,1]$ obtained by iterated reflections in the hyperbolic geodesics that are the sides of the ideal triangle with vertices at 0, 1, and $\infty$. The three dual horocircles centered at 0, 1, and $\infty$, with those centered at 0 and 1 having equal Euclidean radii $1/2$, generate the full family of Ford circles under the reflections in the sides of the ideal triangle above.    
 
\begin{remark}\label{question_mark_rem}
The homeomorphism $h_{\rm new}$ is known in the literature as the Conway's box function; its inverse is the classical Minkowski question mark function, see \cite[\S 4]{Salem} (cf. \cite[\S 4.4.2]{LLMM1}). 
\end{remark}

Recall that a Ford circle $C[p/q]$ that corresponds to a fraction $p/q\in[0,1]$ in its lowest terms is a circle whose radius is $1/(2q^2)$ and center $(p/q,1/(2q^2))$; see Figure~\ref{fig:ford}. Two Ford circles are either disjoint or exterior-wise tangent to each other. Two Ford circles $C[p/q]$ and $C[r/s]$ are tangent to each other if and only if $p/q$ and $r/s$ are neighbors in some Farey sequence. Also, if  $C[p/q]$ and $C[r/s]$ are tangent Ford circles, then $C[(p+r)/(q+s)]$ is the Ford circle that touches both of them.
 
 \begin{figure}[h]
  \includegraphics[height=7cm]{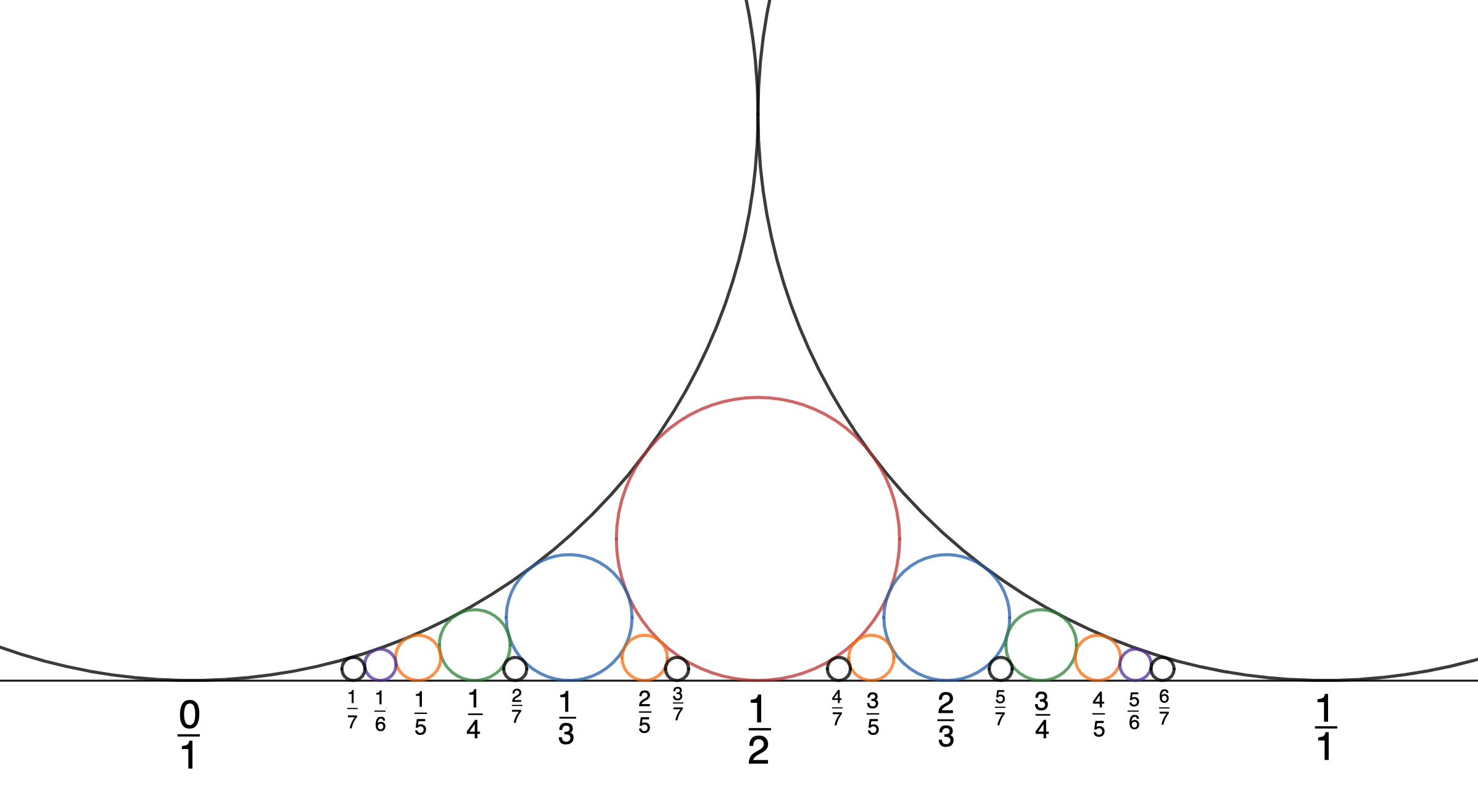}
  \caption{Ford circles.}
  \label{fig:ford}
 \end{figure}
  
Let $D_n=\{k/2^n\colon k=0,1,\dots, 2^n\}$ be the sequence of dyadic points of level $n\in\mathbb N\cup\{0\}$, and let $F_n=h_{\rm new}(D_n)$ be the corresponding sequence of Farey numbers. 
Note that if $p/q$ and $r/s$ are two neighbors in $F_n$ that correspond under $h_{\rm new}$ to two neighbors $k/2^n$ and $(k+1)/2^n$ in $D_n$, then the point $t/u$ in $F_{n+1}$ that corresponds to the midpoint $(2k+1)/2^{n+1}$ is given by
$$
\frac{t}{u}=\frac{p+r}{q+s}.
$$  
Moreover, the Euclidean distance between two neighbors $p/q$ and $r/s$ in $F_n$ is given by 
$$
\frac{1}{qs}.
$$ 
This follows from the fact that the Farey numbers are generated by the modular group. 

Furthermore, we have the following relations between Ford circles.
We say that a Ford circle $C[p/q]$, as well as $p/q$, has generation $n$ if the point $p/q$ belongs to $F_n$ but not to $F_{n-1}$. Note that $F_{n-1}\subset F_n$.  The only two Ford circles of the same generation that are tangent to each other are the circles $C[0/1]$ and $C[1/1]$ of generation 0. Also, for each pair $C[p/q]$ and $C[r/s]$ of tangent Ford circles, one of which has generation at least 1, there are exactly two other Ford circles that are tangent to both, $C[p/q]$ and $C[r/s]$. 
We need the following lemma.  
\begin{lemma}\label{L:Comps}
	There exists an absolute constant $L\geq 1$ with the following property. Let $C[p/q]$ be a Ford circle of generation $m\geq1$ and $C[t/u], C[v/w]$ be two Ford circles that are tangent to  $C[p/q]$ and belong to the same generation $n>m$. Then
	$$
	\frac{1}{L}\leq \frac{u}{w}\leq L.
	$$	
	In other words, the Euclidean lengths of two neighboring complementary intervals of $F_n$, with common end point in $F_m,\ m<n$, are comparable. 
\end{lemma}
\begin{proof}
	Since $C[t/u]$ and $C[v/w]$ have the same generation $n$ and are tangent to the same Ford circle $C[p/q]$ of lower generation, the points $t/u$ and $v/w$ are separated by $p/q$. Without loss of generality we assume that $t/u<p/q<v/w$. 
	
	Let $C[t_{m+1}/u_{m+1}]$ and $C[v_{m+1}/w_{m+1}]$ be the Ford circles of generation $m+1$ that are tangent to $C[p/q]$ and such that 
	$$
	t_{m+1}/u_{m+1}<t/u<p/q<v/w<v_{m+1}/w_{m+1}.
	$$ 
	Moreover, both radii $1/(2u_{m+1}^2), 1/(2w_{m+1}^2)$ of $C[t_{m+1}/u_{m+1}], C[v_{m+1}/w_{m+1}]$, respectively, are smaller than the radius $1/(2q^2)$ of $C[p/q]$.  This implies that the $y$-coordinate of the point of tangency of $C[t_{m+1}/u_{m+1}]$ and $C[p/q]$ is less than $1/(2q^2)$. On the other hand, this $y$-coordinate has to be greater than $(1-1/\sqrt{2})/(2q^2)$, because otherwise $C[t_{m+1}/u_{m+1}]$ would not be tangent to any Ford circle of generation less than $m$, but it must be. The same is true for the $y$-coordinate of the tangency point of $C[v_{m+1}/w_{m+1}]$ and $C[p/q]$. 
	
	Since the desired inequalities are scale invariant, we may rescale and assume that the radius of $C[p/q]$ is 1/2 and it is tangent to the real line at 0. Then the corresponding $y$-coordinates are in the segment $[(1-1/\sqrt{2})/2,1/2]$. Furthermore, we may apply the inversion $z\mapsto 1/\bar z$ and look at the corresponding lengths in the spherical metric. Note that the Euclidean and the spherical metrics are locally bi-Lipschitz. The circles that correspond to  $C[t_{m+1}/u_{m+1}]$ and  $C[v_{m+1}/w_{m+1}]$ under these transformations are circles of radii 1/2 that are tangent to the real line at points contained in the segments $[-C,-1/C]$ and $[1/C,C]$, respectively, where $C\geq1$ is an absolute constant.      
	
	Now, the circles that correspond to the above transformation, i.e., scaling, translation and inversion, $C[t/u]$ and $C[v/w]$ are the circles of radii 1/2 that touch the real line at points contained in the segments 
	$$
	[-n+m+1-C, -n-m+1-1/C],\quad [n-m-1+1/C, n-m-1+C],
	$$
	respectively. 
	For $\alpha>0$, the spherical length of $[\alpha,\infty]$ is comparable to $1/\alpha$. Therefore, the largest ratio of the lengths of the two intervals is comparable to 
	$$
	\frac{n-m-1+C}{n-m-1+1/C},
		$$ 
		which has uniform lower and upper bound because $n>m\geq1$. 
	\end{proof}

The following proposition is crucial in estimating the distortion of $h_{\rm new}$. 

\begin{proposition}\label{P:Neighbors}
	There exists an absolute constant $C\geq1$ such that the following holds.
	For $n\geq 1$, let $I$ and $J$ be two complementary intervals of $F_n$ that are separated by at most two adjacent complementary intervals of $F_n$. Then for the corresponding Euclidean lengths $|I|$ and $|J|$ we have
	\begin{equation}\label{E:KeyIneq}
	\frac{1}{Cn}\leq\frac{|I|}{|J|}\leq Cn.
	\end{equation}
\end{proposition}
\begin{proof}
The proof is by case analysis.

\noindent
\emph{Case~1.}
We assume first that $I$ and $J$ are adjacent. There are two subcases to consider, according to whether the Ford circle $C[p/q]$ at the common end point $p/q$ of $I$ and $J$ has generation strictly less than $n$ or equal to $n$. In the former case, the other two end points of $I$ and $J$ would have to be in $F_n\setminus F_{n-1}$, and thus Lemma~\ref{L:Comps} gives
$$
\frac{1}{L}\leq\frac{|I|}{|J|}\leq L,
$$  
for some absolute constant $L\geq1$, which is stronger than~\eqref{E:KeyIneq}. 
In the latter case, the common end point $p/q$ of $I$ and $J$ is in $F_n$. We argue by induction on $n$. If $n=1$, then $|I|=|J|=1/2$, and we are done. Suppose that~\eqref{E:KeyIneq} is true for generations at most $n-1$. 
Let $r/s$ and $t/u$ be the other two end points of $I$ and $J$, respectively. Then both of them would be in $F_{n-1}$ and 
$$
\frac{p}{q}=\frac{r+t}{s+u}.
$$ 
Moreover, exactly one of the end points $r/s$ and $t/u$ will be in $F_{n-1}\setminus F_{n-2}$. Without loss of generality we assume that it is $t/u$ and 
$$
\frac{r}{s}<\frac{p}{q}<\frac{t}{u}.
$$ 
It follows that $r/s\in F_{n-2}$ because, unless the generation is 0, no two Ford circles of the same generation are tangent: they are separated by Ford circles of lower generation. 
Let $v/w\in F_{n-2}$ be a neighbor of $r/s$ in $F_{n-2}$ with
$$
\frac{r}{s}<\frac{t}{u}<\frac{v}{w}.
$$   
Then, $u=s+w$, and thus
$$
\frac{|I|}{|J|}=\frac{u}{s}=\frac{s+w}{s}=1+\frac{w}{s}\leq n.
$$
The last inequality follows from the induction hypothesis. Using symmetry arguments, we conclude that if $I$ and $J$ are adjacent complementary intervals of $F_n$, then
$$
\frac{1}{n}\leq\frac{|I|}{|J|}\leq n,
$$
which is stronger than~\eqref{E:KeyIneq}.

\noindent
\emph{Case 2.}
The next case to consider is when $I$ and $J$ are separated by a single complementary interval $K$ of $F_n$. Let the four end points of these intervals be $a_1, a_2, a_3$, and $a_4$, with
$$
a_1<a_2<a_3<a_4.
$$ 
We assume that $I=(a_1, a_2), K=(a_2, a_3)$, and $J=(a_3, a_4)$.
By symmetry, we may further assume that $a_1$ and $a_3$ have generation $n$, and thus the generations of $a_2$ and $a_4$ are strictly less than $n$. From Lemma~\ref{L:Comps} we know that 
$$
\frac{1}{L}\leq\frac{|I|}{|K|}\leq L,
$$
for some absolute constant $L\geq1$, and thus from Case~1 we conclude that
$$
\frac{1}{Ln}\leq\frac{|I|}{|J|}\leq L n.
$$

\noindent
\emph{Case~3.} Now we look at the case when $I$ and $J$ are separated by two adjacent intervals $K_1$ and $K_2$. Let $I=(a_1, a_2), K_1=(a_2, a_3), K_2=(a_3, a_4)$, and $J=(a_4, a_5)$. There are two subcases: the generations of $a_1, a_3, a_5$ are $n$ and  the generations of $a_2, a_4$ are strictly less than $n$, or the generations of $a_1, a_3, a_5$ are strictly less than $n$ and the generations of $a_2, a_4$ are $n$. In the first subcase we use  Lemma~\ref{L:Comps} to conclude that
$$
\frac{1}{L}\leq\frac{|I|}{|K_1|}\leq L,\quad \frac{1}{L}\leq\frac{|K_2|}{|J|}\leq L.
$$
Also, from Case~1 we have
$$
\frac{1}{n}\leq\frac{|K_1|}{|K_2|}\leq n.
$$
Putting these together, we obtain
$$
\frac{1}{L^2n}\leq\frac{|I|}{|J|}\leq L^2n.
$$

In the second subcase, $(a_1, a_3)$ and $(a_3, a_5)$ are adjacent complementary intervals of $F_{n-1}$ of lengths $|I|+|K_1|$ and $|K_2|+|J|$, respectively. Therefore, either $a_1, a_5$ have generations $n-1$ and the generation of $a_3$ is strictly less than $n-1$, or $a_3$ has generation $n-1$ and the generations of $a_1, a_5$ are strictly less than $n-1$. If it is the former subsubcase, by Lemma~\ref{L:Comps} we obtain
$$
\frac{1}{L}\leq \frac{|I|+|K_1|}{|K_2|+|J|}\leq L.
$$
Using Case~1 we get
$$
\frac{|K_2|+|J|}{|J|}\leq n+1, 
$$
and therefore
$$
\frac{|I|}{|J|}\leq L(n+1)\leq 2Ln.
$$
From symmetry we get
$$
\frac{1}{2Ln}\leq\frac{|I|}{|J|}\leq 2Ln,
$$
as desired.
In the latter subsubcase we 
have by Lemma~\ref{L:Comps}
$$
\frac1{L}\leq\frac{|K_1|}{|K_2|}\leq L.
$$
Also, $|K_1|\leq|I|$ and $|K_2|\leq|J|$. Indeed, 	we prove the first of these inequalities and the second follows by symmetry. If $a_1=p/q, a_3= r/s$, and $a_5=t/u$, then 
		$$
		s=q+u,
		$$ 
		and therefore
		$$
		\frac{|K_1|}{|I|}=\frac{q}{s}\leq1.
		$$
		
Now, combining these estimates with Case~1, we obtain
$$
\frac{|I|}{|J|}\leq \frac{|I|}{|K_1|}\frac{|K_1|}{|K_2|}\frac{|K_2|}{|J|}\leq Ln,
$$
and the proposition follows.
\end{proof}

We are now ready to prove the following lemma.
\begin{lemma}\label{L:DavidExt}
The map $h$ defined at the beginning of the current section has a David extension homeomorphism $H$ of $\mathbb D$.
\end{lemma}
\begin{proof}
As discussed earlier, we need to show that  $\rho_{h_{\rm new}}$ satisfies~\eqref{E:SWD2}. 
Indeed, we choose an arbitrary $x\in[0,1]$ and $t>0$ small enough so that the segment $[x-t,x+t]$ is contained in a neighborhood of $[0,1]$ where $h_{\rm new}$ is defined. Let $n\in\mathbb N$ be chosen so that
$$
\frac{1}{2^{n}}\leq t< \frac{1}{2^{n-1}}.
$$
Then, $[x-t,x]$ and $[x,x+t]$ are each contained in at most three consecutive complementary intervals of $F_n$, and contain at least one such complementary interval.
Therefore,
$$
\frac{h_{\rm new}(x+t)-h_{\rm new}(x)}{h_{\rm new}(x)-h_{\rm new}(x-t)}\leq\frac{|I_1|+|I_2|+|I_3|}{|J|},
$$  
where $I_1, I_2, I_3$, and $J$ are distinct complementary intervals of $F_n$ so that the pairs $\{J, I_1\}, \{I_1,I_2\}$ and $\{I_2, I_3\}$ are adjacent. From Proposition~\ref{P:Neighbors} we conclude that
$$
\frac{h_{\rm new}(x+t)-h_{\rm new}(x)}{h_{\rm new}(x)-h_{\rm new}(x-t)}\leq3Cn,
$$
for some absolute constant $C\geq 1$. From symmetry we also have
$$
\frac{1}{3Cn}\leq\frac{h_{\rm new}(x+t)-h_{\rm new}(x)}{h_{\rm new}(x)-h_{\rm new}(x-t)},
$$
and thus~\eqref{E:SWD2} follows.

We now apply~\cite[Theorem~3.1]{Z08} to conclude that $h$ has a David extension inside the unit disk $\mathbb D$. 
\end{proof}

\begin{theorem}[David surgery]\label{T:DavidS}
Let $f$ be a critically periodic anti-rational map and $U_1, U_2,\dots, U_n$ be fixed Jordan domain Fatou components of $f$ so that the restriction $f|_{\partial U_i}$ to each $\partial U_i$ has degree $-2,\ i=1, 2, \dots, n$. Then there is a global David surgery that replaces the dynamics of $f$ on each $U_i$ by the dynamics of the Nielsen map $\pmb{\rho}_2$ associated to the ideal triangle $\Pi$.  
\end{theorem}
\begin{proof}
Let $i=1,2,\dots, n$, and let $\phi_i\colon U_i\to\mathbb D$ be a B\"ottcher coordinate that conjugates $f|_{U_i},\ i=1,2,\dots,n$, to the map $g_2(z)=\bar z^2$. Our assumptions imply that such  maps $\phi_i, i=1,2,\dots, n$, exist. 

Let $H$ be the David extension of $h$ guaranteed by Lemma~\ref{L:DavidExt}. We replace the map $f$ by the map 
$$
f_H=
\begin{cases}
\phi_i^{-1}\circ H^{-1}\circ \pmb{\rho}_2\circ H\circ\phi_i,\quad {\rm in}\ U_i\setminus \Int{T_{i,H}},\quad i=1,2,\dots, n,\\
f,\quad {\rm in\ } \hat{\mathbb C}\setminus\cup_{i=1}^n U_i,
\end{cases}
$$
where $T_{i,H}=\phi_i^{-1}\circ H^{-1}(\Pi)$, to obtain a continuous orientation reversing map of $\hat{\mathbb C}\setminus\cup_{i=1}^n\Int{T_{i,H}}$ onto $\hat{\mathbb C}$.
If $\mu_H$ is the pull-back to $U_i,\ i=1,2,\dots, n$, of the standard complex structure in $\mathbb D$ by the map $\phi_i\circ H$, we have $(f_H|_{U_i\setminus T_{i,H}})^*(\mu_H)=\mu_H$. 

We now  use the dynamics of $f$ to spread the Beltrami coefficient out to all the preimage components of $U=U_i,\ i=1,2,\dots, n$, under all the iterates of $f$. 
Everywhere else we use the standard complex structure, i.e., the zero Beltrami coefficient. This way we obtain a global Beltrami coefficient, still denoted by $\mu_H$, that is invariant under $f_H$.

Since $f$ is hyperbolic, the Beltrami coefficient $\mu_H$ satisfies the David condition~\eqref{E:DavidIneq} in the whole Riemann sphere $\hat{\mathbb C}$. Indeed, this follows from an observation that $\mu_H$ is David in each $U=U_i,\ i=1,2,\dots, n$, with the same constants $M,\alpha, \epsilon_0>0$, and from the Koebe Distortion Theorem. The former is a consequence of the fact that each $U=U_i$ is a quasidisk and hence $\phi$ has a global quasiconformal extension, which in turn implies that $\phi$ distorts areas via a power law; see, e.g., \cite{B57}, \cite{A94}.
The latter would give us that if $U'$ is a component of the preimage of $U=U_i, \ i=1,2,\dots, n$, under some iterate of $f$, i.e., such that $f^k(U')=U$, where $k\in\mathbb N$ is the smallest, then $f^k\circ \lambda_{U'}$ is an $L$-bi-Lipschitz map between $\frac{1}{{\rm diam}(U')}U'$ and $U$, for some absolute constant $L\geq1$, where $\lambda_{U'}(z)={\rm diam}(U')z$ is a scaling map.  This implies that, given any $0<\epsilon\leq\epsilon_0$,  
$$
\sigma\{z\in U'\colon|\mu_H(z)|\geq 1-\epsilon\}\leq L^2({\rm diam}(U'))^2\sigma\{z\in U\colon|\mu_H(z)|\geq 1-\epsilon\}.
$$  
Moreover, since all the Fatou components $U'$ (which are iterated preimages of $U$) are uniform quasidisks, there exists a constant $C>0$ such that
$$
({\rm diam}(U'))^2\leq C\sigma\{U'\}.
$$
Therefore,
\begin{equation}\notag
\begin{split}
\sigma\{z\in\hat{\mathbb C}\colon|\mu_H(z)|\geq 1-\epsilon\}&=\sum_{i=1}^n\sum_{U'}\sigma\{z\in{U'}\colon|\mu_H(z)|\geq 1-\epsilon\}\\
&\leq L^2 C\left(\sum_{U'}\sigma\{U'\}\right)\sum_{i=1}^n\sigma\{z\in U_i\colon|\mu_H(z)|\geq 1-\epsilon\}\\
&\leq nL^2C\sigma\{\hat{\mathbb C}\}Me^{-\alpha/\epsilon},\quad \epsilon\leq\epsilon_0.
\end{split}
\end{equation}

The David Integrability Theorem then gives us an orientation preserving homeomorphism $\Psi$ of $\hat{\mathbb C}$ such that the pull-back of the standard complex structure under $\Psi$ is equal to $\mu_H$.

The last claim is that the map $F=\Psi\circ f_H\circ \Psi^{-1}$ is analytic. This is the desired map that replaces the dynamics of $f$ on each $U_i,\ i=1,2,\dots,n$, with the dynamics of the Nielsen map $\pmb{\rho}_2$ associated to the ideal triangle $\Pi$. The conclusion of analyticity follows from the uniqueness part of the David Integrability Theorem. The arguments below are similar to~\cite[Section~9]{BF14}. Indeed, since $F\circ \Psi=\Psi\circ f_H$, it is enough to show that $\Psi\circ f_H$ is in $W_{\rm loc}^{1,1}$. Then, since we know that $\Psi\in W_{\rm loc}^{1,1}$ and both $\Psi$ and $\Psi\circ f_H$ integrate $\mu_H$, we can apply the uniqueness of the David Integrability Theorem to obtain that $F$ is analytic. 

Let $V$ be an open set in $\hat{\mathbb C}$. Since quasicircles are removable for David maps, see, e.g., \cite[Lemma~4.2]{Z04} that applies verbatim to quasicircles in place of the unit circle, 
to show that $\Psi\circ f_H\in W_{\rm loc}^{1,1}$, it is enough to prove that $\Psi\circ f_H\in W_{\rm loc}^{1,1}(V)$ for the  two cases: $V$ does not intersect any of the components $U'$ of the preimage of each $U=U_i,\ i=1,2,\dots, n$, under all the iterates of $f$, or $V$ is completely contained in one such component $U'$. In the first case, $f_H$ is analytic, and hence the composition $\Psi\circ f_H$ is in $W_{\rm loc}^{1,1}(V)$. For the second case, let $k\in\mathbb N\cup\{0\}$ be the smallest number such that $f^k(U')=U$. We can write 
$$
\Psi\circ f_H=(\Psi\circ f^{-k}\circ\phi^{-1}\circ H^{-1})\circ(\pmb{\rho}_2\circ H\circ \phi\circ f^k).
$$  
Since $\phi\circ f^k$ and $\pmb{\rho}_2$ are (anti-)analytic, the composition $\pmb{\rho}_2\circ H\circ \phi\circ f^k$ with a $W_{\rm loc}^{1,1}$-map is in $W_{\rm loc}^{1,1}$. Also, the map $H\circ\phi\circ f^k$ is a composition of a $W_{\rm loc}^{1,1}$-map and an (anti-)analytic map, and thus is itself in $W_{\rm loc}^{1,1}$. Both maps $\Psi$ and $H\circ\phi\circ f^k$ also integrate $\mu_H$. Therefore, from the uniqueness part of  the David Integrability Theorem we obtain that $\Psi\circ f^{-k}\circ\phi^{-1}\circ H^{-1}$ is analytic. Now, since the composition of an analytic map with a $W_{\rm loc}^{1,1}$-map is $W_{\rm loc}^{1,1}$, we are done.
\end{proof}

\begin{corollary}\label{David_surgery_cor}
The anti-analytic map $$F=\Psi\circ f_H\circ \Psi^{-1}:\widehat{\C}\setminus\bigcup_{i=1}^n\Int{\Psi(T_{i,H})}\to\widehat{\C}$$ is topologically conjugate to $\pmb{\rho}_2$ on $\Psi(U_i),\ i=1,\dots,n$, and to $f$ on $\displaystyle\widehat{\C}\setminus\bigcup_{i=1}^n\Psi(U_i)$. 
\end{corollary}

\section{From anti-rational map to Nielsen map and Schwarz reflections}\label{mating_sec}

For the rest of the paper, $\T$ will stand for the tetrahedral triangulation of the $2$-sphere. To ease notations, we will omit the subscript $\T$ for the objects $\mathcal{C}_\T, H_\T, N_\T, G_\T$, and $g_\T$.

Recall that the cubic anti-rational map $g$ (constructed in Section~\ref{affine_model_sec}) has four fixed Fatou components on each of which the action of $g$ is conformally conjugate to $z\mapsto\overline{z}^2$ on $\D$. In particular, the boundaries of these Fatou components are Jordan curves, and the restriction of $g$ to each of these boundaries has degree $-2$. Hence, we can apply David surgery Theorem~\ref{T:DavidS} on $n$ of these fixed Fatou components ($n\in\{1,2,3,4\}$) to replace the action of $g$ on these Fatou components by the action of the Nielsen map $\pmb{\rho}_2$ of the ideal triangle group. As a result, we produce anti-analytic maps (defined on a subset of $\widehat{\C}$) that combine the features of anti-rational maps and Nielsen maps of reflection groups. Such hybrid dynamical systems are realized as {Schwarz reflection maps} associated with {quadrature domains}. 

In particular, if the David surgery is performed on all four fixed Fatou components of $g$, we recover the Nielsen map $N$ of the classical Apollonian gasket reflection group. On the other hand, if the David surgery is carried out on three fixed Fatou components of $g$, we obtain a ``mating'' of $g$ with the Nielsen map of the Apollonian gasket reflection group. We explicitly characterize this anti-holomorphic map as the Schwarz reflection map with respect to a deltoid and an inscribed circle.

\subsection{Background on Schwarz reflection maps}\label{schwarz_back_sec}

We will denote the complex conjugation map on the Riemann sphere by $\iota$, and reflection in the unit circle by $\eta$. 

\subsubsection{Basic definitions and properties}
\begin{definition}[Schwarz function]\label{schwarz_func_def}
Let $\mathfrak{D}\subsetneq\widehat{\C}$ be a domain such that $\infty\notin\partial\mathfrak{D}$ and $\Int{\overline{\mathfrak{D}}}=\mathfrak{D}$. A \emph{Schwarz function} of $\mathfrak{D}$ is a meromorphic extension of $\iota\vert_{\partial\mathfrak{D}}$ to all of $\mathfrak{D}$. More precisely, a continuous function $S:\overline{\mathfrak{D}}\to\widehat{\C}$ of $\mathfrak{D}$ is called a Schwarz function of $\mathfrak{D}$ if it satisfies the following two properties:
\begin{enumerate}
\item $S$ is meromorphic on $\mathfrak{D}$,

\item $S=\iota$ on $\partial \mathfrak{D}$.
\end{enumerate}
\end{definition}

It is easy to see from the definition that a Schwarz function of a domain (if it exists) is unique. 

\begin{definition}[Quadrature domains]\label{qd_def}
A domain $\mathfrak{D}\subsetneq\widehat{\C}$ with $\infty\notin\partial\mathfrak{D}$ and $\Int{\overline{\mathfrak{D}}}=\mathfrak{D}$ is called a \emph{quadrature domain} if $\mathfrak{D}$ admits a Schwarz function.
\end{definition}

Note that for a quadrature domain $\mathfrak{D}$, the map $\sigma:=\iota\circ S:\overline{\mathfrak{D}}\to\widehat{\C}$ is an anti-meromorphic extension of the local reflection maps with respect to $\partial\mathfrak{D}$ near its non-singular points (the reflection map fixes $\partial\mathfrak{D}$ pointwise). We will call $\sigma$ the \emph{Schwarz reflection map of} $\mathfrak{D}$.

Simply connected quadrature domains are of particular interest, and these admit a simple characterization. 

\begin{proposition}[Simply connected quadrature domains]\label{s.c.q.d.}
A simply connected domain $\mathfrak{D}\subsetneq\widehat{\C}$ with $\infty\notin\partial\mathfrak{D}$ and $\Int{\overline{\mathfrak{D}}}=\mathfrak{D}$ is a quadrature domain if and only if the Riemann uniformization $R:\hat{\C}\setminus\overline{\mathbb{D}}\to\mathfrak{D}$ extends to a rational map on $\widehat{\C}$. In this case, the Schwarz reflection map $\sigma$ of $\mathfrak{D}$ is given by $R\circ\eta\circ(R\vert_{\hat{\C}\setminus\overline{\mathbb{D}}})^{-1}$. Moreover, if $\deg{R}\geq 2$, we have $\sigma(\overline{\mathfrak{D}})=\widehat{\C}$. 

Moreover, if the degree of the rational map $R$ is $d$, then $\sigma:\sigma^{-1}(\mathfrak{D})\to\mathfrak{D}$ is a branched covering of degree $(d-1)$, and $\sigma:\sigma^{-1}(\Int{\mathfrak{D}^c})\to\Int{\mathfrak{D}^c}$ is a branched covering of degree $d$.
\end{proposition}

\begin{proof}
The first part is the content of \cite[Theorem~1]{AS}. The statements about covering properties of $\sigma$ follow from the commutative diagram below.

\[ \begin{tikzcd}
\hat{\C}\setminus\D \arrow{r}{R} \arrow[swap]{d}{\eta} & \overline{\mathfrak{D}} \arrow{d}{\sigma} \\
\D \arrow{r}{R}& \widehat{\C}
\end{tikzcd}
\]
\end{proof}

\subsection{Recovering the Nielsen map of the classical Apollonian gasket reflection group}\label{david_applications_subsec}

We will now show that the Nielsen map $N$ associated with the reflection group $H$ (arising from the tetrahedral triangulation) can be constructed from $g$ by David surgery.

\begin{proposition}[Recovering the Nielsen map]\label{Nielsen_recovery_prop}
There is a global David surgery that replaces the action of $g$ on each of its fixed Fatou components by the action of $\pmb{\rho}_2:\D\setminus\Int{\Pi}\to\D$. The resulting anti-analytic map is the Nielsen map $N$ of the classical Apollonian gasket reflection group $H$ (up to M{\"o}bius conjugacy).
\end{proposition}
\begin{proof}
The first statement is the content of Theorem~\ref{T:DavidS}. 

Moreover, $F$ maps each of these Jordan domains anti-conformally to its exterior, and fixes the boundary pointwise. Therefore, each such Jordan domain is a quadrature domain, and $F$ acts on it as the corresponding Schwarz reflection map. The second statement of Proposition~\ref{s.c.q.d.} now implies that each of the above Jordan domains is the image of a round disk under a M{\"o}bius map, and hence is a round disk itself. In particular, $F$ acts on the disk as reflection in its boundary. 

Clearly, the configuration of these four circles is dual to the circle packing $\mathcal{C}$ corresponding to the tetrahedral triangulation (unique up to a M{\"o}bius map). Thus by definition, $F$ is the Nielsen map $N$ associated with the classical Apollonian gasket reflection group.
\end{proof}

\subsection{Schwarz reflection in a deltoid and circle}\label{schwarz_subsec}

In this final subsection, we discuss another application of David surgery that produces an anti-holomorphic dynamical system which can be viewed as a mating of the cubic anti-rational map $g$ and the Nielsen map of the classical Apollonian gasket reflection group. Furthermore, we give an explicit description of this anti-holomorphic dynamical system as a suitable Schwarz reflection map.
 
\begin{proposition}\label{deltoid_circle_prop}
There is a global David surgery that replaces the action of $g$ on three of its fixed Fatou components by the action of $\pmb{\rho}_2:\D\setminus\Int{\Pi}\to\D$. The resulting anti-analytic map is the Schwarz reflection map associated with a deltoid and a circle (up to M{\"o}bius conjugacy).
\end{proposition}
\begin{proof}
Once again, Theorem~\ref{T:DavidS} gives the existence of an anti-holomorphic map $F$, defined on a subset of $\widehat{\C}$, that is conjugate to $\pmb{\rho}_2:\D\setminus\Int{\Pi}\to\D$ on three $F$-invariant Jordan domains and conjugate to $g$ elsewhere (via a global David homeomorphism). It remains to characterize the anti-holomorphic map $F$.

By construction, $F$ has a unique critical point, and this critical point is simple and fixed. Possibly after conjugating $F$ by a M{\"o}bius map, we can assume that this critical point is at $\infty$.

It also follows from the construction that the map $F$ is defined on the complement of the interiors of three topological triangles. Since the vertices of these triangles correspond to the touching points of the fixed Fatou components of $g$, it is easily seen that the domain of definition of $F$ is the union of the closures of two disjoint Jordan domains $\mathfrak{D}_1$ and $\mathfrak{D}_2$ that touch exactly at three points. We can assume that $\infty\in\mathfrak{D}_1$.

Since the anti-holomorphic map $F$ fixes the boundaries of $\mathfrak{D}_i$ (for $i=1,2$), it follows that both $\mathfrak{D}_1$ and $\mathfrak{D}_2$ are simply connected quadrature domains. Moreover, $F$ maps $\mathfrak{D}_2$ anti-conformally to its exterior. By Proposition~\ref{s.c.q.d.}, $\mathfrak{D}_2$ is a round disk, and $F$ acts on it as reflection in the circle $\partial\mathfrak{D}_2$.

Again, the mapping properties of $F$ imply that $\infty\in\mathfrak{D}_1$ has only two pre-images in $\mathfrak{D}_1$ counting multiplicity (in fact, $F$ maps $\infty$ to itself with local degree two). Thus, $F:F^{-1}(\mathfrak{D}_1)\to\mathfrak{D}_1$ is a branched covering of degree $2$. By Proposition~\ref{s.c.q.d.}, there exists a rational map $R$ of degree $3$ which maps $\widehat{\C}\setminus\overline{\mathbb{D}}$ univalently onto $\mathfrak{D}_1$. Pre-composing $R$ with a conformal automorphism of $\widehat{\C}\setminus\overline{\D}$, we may assume that $R(\infty)=\infty$. 

In light of the commutative diagram in the proof of Proposition~\ref{s.c.q.d.}, the fact that $\infty$ is a (simple) fixed critical point of $F$ implies that $R(0)=\infty$ and $R'(0)=0$. Hence, $R$ is of the form $$R(z)=az+b+\frac{c}{z}+\frac{d}{z^2},$$ for some $a,d\in\C^*$, and $b,c\in\C$. Possibly after post-composing $R$ with an affine map (which amounts to replacing $\mathfrak{D}_1$ by an affine image of it, and conjugating $F$ by the same affine map), we may write $$R(z)=z+\frac{c}{z}+\frac{d}{z^2},$$ for some $c\in\C$ and $d\in\C^*$. 

Note that the cubic anti-rational map $R$ has four critical points (counting multiplicity), one of which is at the origin. Since $F$ has only one critical point, the same commutative diagram implies that the other three critical points of $R$ lie on the unit circle (in fact, univalence of $R$ on $\hat{\C}\setminus\overline{\D}$ implies that these critical points are distinct). In particular, the product of the solutions of the equation $z^3R'(z)=0$ has absolute value $1$. A simple computation now shows that $\vert d\vert=\frac12.$ We can now conjugate $R$ by a rotation (once again, this amounts to replacing $\mathfrak{D}_1$ by a rotated image of it, and conjugating $F$ by the same rotation), we may write $$R(z)=z+\frac{c}{z}+\frac{1}{2z^2},$$ for some $c\in\C$. Denoting the three non-zero critical points of $R$ by $\alpha, \beta$, and $\gamma$, we obtain the relations $$\alpha+\beta+\gamma=0,\ \alpha\beta+\beta\gamma+\gamma\alpha=-c,\ \textrm{and}\ \alpha\beta\gamma=1.$$ Since $\alpha, \beta,\gamma\in\mathbb{T}$, we have that 
\begin{align*}
z^3-cz-1 &\equiv (z-\alpha)(z-\beta)(z-\beta)\\
& \equiv\frac{1}{\alpha\beta\gamma} (z-\alpha)(z-\beta)(z-\gamma)\\
&\equiv (\overline{\alpha}z-1)(\overline{\beta}z-1)(\overline{\gamma}z-1)\\
&\equiv z^3+\overline{c}z^2-1.
\end{align*}

We conclude that $c=0$, and hence $R(z)=z+\frac{1}{2z^2}.$ 

\begin{figure}[ht!]
\begin{tikzpicture}
\node[anchor=south west,inner sep=0] at (-0.4,4) {\includegraphics[width=0.45\textwidth]{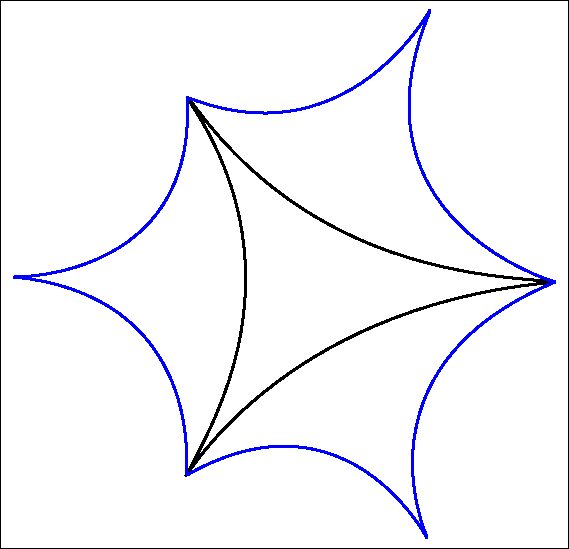}};
\node[anchor=south west,inner sep=0] at (5.6,4) {\includegraphics[width=0.45\textwidth]{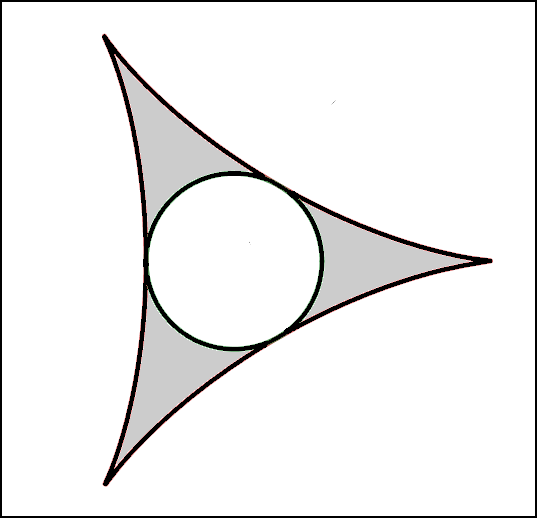}};
\node at (9.5,8.4) {\begin{Large}$\mathfrak{D}_1$\end{Large}};
\node at (8,6.8) {\begin{Large}$\mathfrak{D}_2$\end{Large}};
\node at (9.5,6.75) {$\mathfrak{T}^0_1$};
\node at (7.4,7.8) {$\mathfrak{T}^0_2$};
\node at (7.4,5.6) {$\mathfrak{T}^0_3$};
\node at (3.2,7.5) {\begin{Large}$\partial\mathfrak{D}_1$\end{Large}};
\node at (2.4,9) {\begin{Large}$\partial\sigma_1^{-1}(\mathfrak{D}_1)$\end{Large}};
\end{tikzpicture}
\caption{Left: The quadrature domain $\mathfrak{D}_1$ is the exterior of the deltoid curve (in black). The region $\sigma_1^{-1}(\mathfrak{D}_1)$ is the exterior of the hexagonal curve (in blue). Right: The domain of definition of $F$ is the closure of the union of the exterior of the deltoid and the interior of the inscribed disk. The fundamental tile $\mathfrak{T}^0$ has three connected components; namely $\mathfrak{T}^0_1, \mathfrak{T}^0_2,$ and $\mathfrak{T}^0_3$.}
\label{deltoid_pre_image}
\end{figure}

We already know that $R(\mathbb{T})$, where $\mathbb{T}$ is the unit circle, is a Jordan curve (in fact, this can be easily checked directly from the above formula of $R$). The curve $R(\mathbb{T})$ is a classical deltoid curve (compare \cite[\S 4]{LLMM1}, where the dynamics of the Schwarz reflection map associated with $\mathfrak{D}_1$ was studied in detail). Since $R$ commutes with multiplication by the third roots of unity, it follows that $\mathfrak{D}_1=R(\hat{\C}\setminus\overline{\D})$ is symmetric under rotation by $\frac{2\pi}{3}$. Moreover, the three simple critical points of $R$ on $\mathbb{T}$ produce three $3/2$-cusps on the boundary $\partial\mathfrak{D}_1$ (see Figure~\ref{deltoid_pre_image} (left)). 

Since the boundaries of $\mathfrak{D}_1$ and $\mathfrak{D}_2$ touch at three points, it follows that $\mathfrak{D}_2$ is the largest disk inscribed in $\mathfrak{D}_1$ centered at $0$ (see Figure~\ref{deltoid_pre_image} (right)).

Finally, the map $F$ is explicitly given by the Schwarz reflection maps associated with the exterior of the deltoid, and the inscribed disk. More precisely, we have that 
$$
F(w)= \left\{\begin{array}{ll}
                    \sigma_1(w) & \mbox{if}\ w\in\overline{\mathfrak{D}_1}, \\
                    \sigma_2(w) & \mbox{if}\ w\in\overline{\mathfrak{D}_2},
                                          \end{array}\right. 
$$
where $\sigma_1\equiv R\circ\eta\circ(R\vert_{\hat{\C}\setminus\mathbb{D}})^{-1}:\overline{\mathfrak{D}_1}\to\widehat{\C}$ is the Schwarz reflection map of $\mathfrak{D}_1$, and $\sigma_2$ is reflection in the circle $\partial\mathfrak{D}_2$.

This completes the proof.
\end{proof}

We will conclude by showing that the Riemann sphere splits into two $F$-invariant subsets on one of which $F$ is conjugate to $g$, and on the other it is conjugate to the Nielsen map $N$ arising from the classical Apollonian gasket reflection group $H$.

\begin{figure}[ht!]
\begin{tikzpicture}
\node[anchor=south west,inner sep=0] at (0,0) {\includegraphics[width=0.45\textwidth]{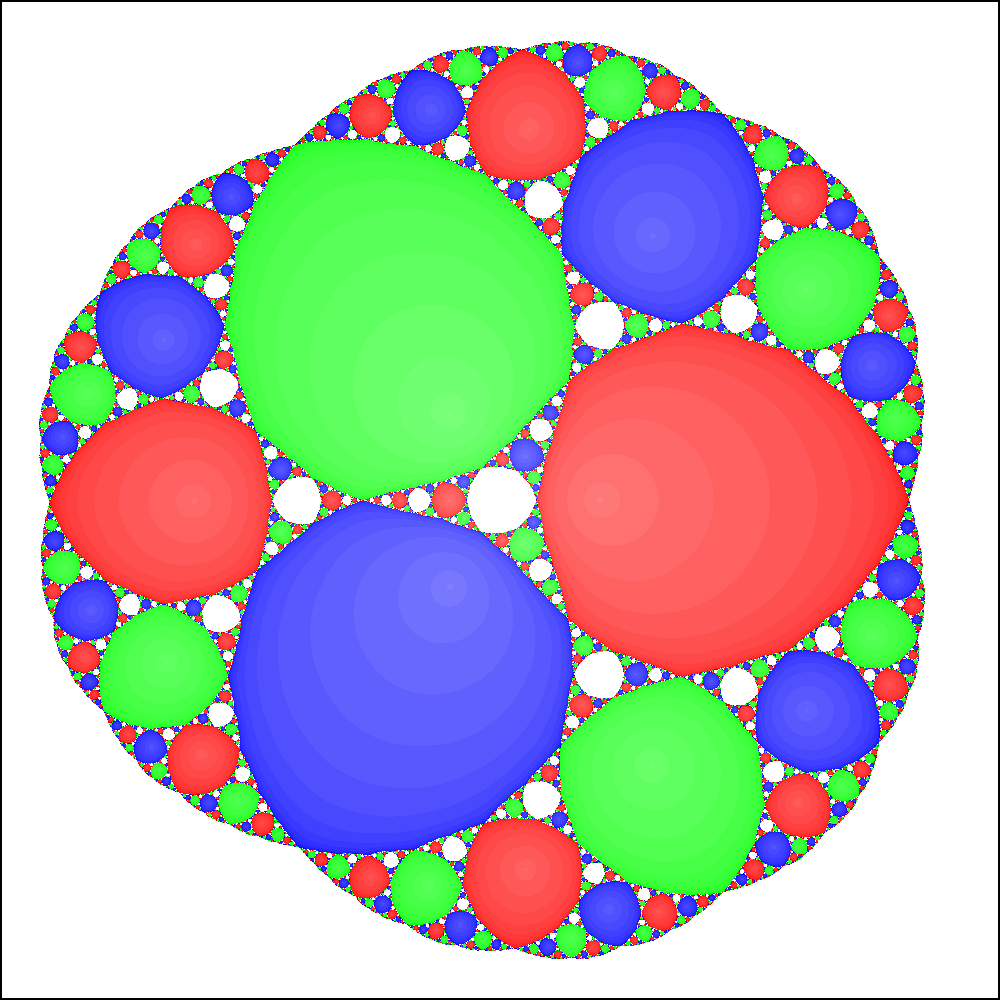}};
\node[anchor=south west,inner sep=0] at (6.4,0) {\includegraphics[width=0.455\textwidth]{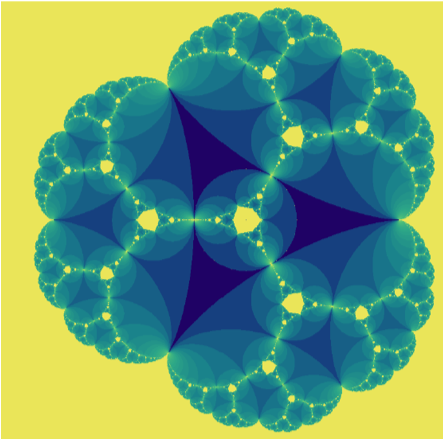}};
\node at (7.2,5.25) {\begin{huge} $\mathfrak{B}_\infty^\mathrm{imm}$ \end{huge}};
\node at (10.8,2.7) {\begin{huge}\textcolor{white}{$\mathfrak{U}_1$} \end{huge}};
\node at (8.8,3.8) {\begin{huge} \textcolor{white}{$\mathfrak{U}_2$} \end{huge}};
\node at (8.8,2) {\begin{huge} \textcolor{white}{$\mathfrak{U}_3$} \end{huge}};
\node at (5.06,5.25) {\begin{huge}\textcolor{black}{$U_4$} \end{huge}};
\node at (2.1,1.8) {\begin{huge} \textcolor{black}{$U_3$} \end{huge}};
\node at (2.1,3.8) {\begin{huge} $U_2$ \end{huge}};
\node at (3.88,2.7) {\begin{huge} \textcolor{black}{$U_1$} \end{huge}};
\draw [->,line width=0.8pt,color=black] (4.8,5) to (2.84,2.84);
\end{tikzpicture}
\caption{Left: the dynamical plane of $g$ is depicted, and the four fixed Fatou components are colored in blue, green, red, and white. Replacing the $\overline{z}^2$ dynamics on $U_1, U_2, U_3$ by $\pmb{\rho}_2$ produces the piecewise Schwarz reflection $F$. Right: The dynamics of $F$ with its basin of infinity (in yellow) and tiling set (in blue/green) marked. Their common boundary is the limit set $\mathfrak{L}$. (Picture courtesy: Seung-Yeop Lee.)}
 \label{fig:apollo}
 \end{figure}
 
Let us set $\mathfrak{T}:=\widehat{\C}\setminus\left(\mathfrak{D}_1\cup\mathfrak{D}_2\right)$. Note that $\partial \mathfrak{T}$ has six singular points. Three of them are $3/2$-cusps on $\partial\mathfrak{D}_1$. The other three singular points on $\partial \mathfrak{T}$ are the tangency points between $\partial\mathfrak{D}_1$ and $\partial\mathfrak{D}_2$. We denote the set of singularities of $\partial\mathfrak{T}$ by $\mathfrak{S}$, and define the \emph{fundamental tile} $\mathfrak{T}^0$ as $\mathfrak{T}\setminus\mathfrak{S}$. The fundamental tile $\mathfrak{T}^0$ has three connected components which we denote as $\mathfrak{T}^0_1, \mathfrak{T}^0_2,$ and $\mathfrak{T}^0_3$ (see Figure~\ref{deltoid_pre_image} (right)). In the dynamical plane of $N$ (or the Apollonian gasket reflection group), we denote the three components of $T^0$ corresponding to $\mathfrak{T}^0_1, \mathfrak{T}^0_2, \mathfrak{T}^0_3$ by $T^0_1, T^0_2, T^0_3$ (see Subsection~\ref{nielsen_2_subsec} and Figure~\ref{fig:group_map}).

\begin{definition}[Tiling Set of $F$]\label{tiling_limit_def}
We define the \emph{tiling set} $\mathfrak{T}^\infty$ of $F$ as $$\mathfrak{T}^\infty:=\displaystyle\bigcup_{n\geq0}F^{-n}(\mathfrak{T}^0).$$ The boundary of $\mathfrak{T}^\infty$ is called the \emph{limit set} of $F$, and is denoted by $\mathfrak{L}$.
\end{definition}

Let us now describe the structure of $\mathfrak{T}^\infty$. For $i=1,2,3$, we define $\mathfrak{U}_i$ to be the connected component of $\mathfrak{T}^\infty$ containing $\mathfrak{T}_i^0$. Note that each point in $\mathfrak{U}_i$ maps to $\mathfrak{T}^0_i$ under iteration of $F$. In particular, $\mathfrak{U}_i$ is a $F$-invariant component of $\mathfrak{T}^\infty$. Every other connected component of $\mathfrak{T}^\infty$ eventually maps to one of these three invariant components (see Figure~\ref{fig:apollo} (right)). In the dynamical plane of $g$, we denote the fixed Fatou components corresponding to $\mathfrak{U}_i$ by $U_i$, $i=1,2,3$ (see Figure~\ref{fig:apollo} (left)). Further, in the dynamical plane of the Nielsen map $N$, we denote the components of the domain of discontinuity corresponding to $\mathfrak{U}_i$ by $\mathcal{U}_i\supset T^0_i$, $i=1,2,3$ (in Figure~\ref{fig:group_map} (right), these are the three bounded round disks enclosed by the black circles). By construction, the tiling set $\mathfrak{T}^\infty$ of $F$ corresponds to 
$$
\Omega_{\mathrm{part}}:= \bigcup_{k\geq 0} N^{-k}\left(\bigcup_{i=1}^3 T^0_i\right) \subsetneq \Omega_H
$$
in the Nielsen dynamical plane.

Recall that $\infty$ is a super-attracting fixed point of $F$. We denote the basin of attraction of $\infty$ by $\mathfrak{B}_\infty$, and the immediate basin of attraction (i.e., the connected component of $\mathfrak{B}_\infty$ containing $\infty$) by $\mathfrak{B}_\infty^\mathrm{imm}$. The corresponding fixed Fatou component of $g$ is denoted by $U_4$ (see Figure~\ref{fig:apollo} (left)). By Corollary~\ref{apollo_rat_map_formula_cor}, the superattracting fixed point of $g$ in $U_4$ is at the origin. We denote the basin of attraction of this superattracting fixed point by $\mathcal{B}_0(g)$.

Since the tangency patterns of the fixed Fatou components of $g$ are preserved by the global David surgery, the next result is immediate.

\begin{proposition}\label{prop:desired_touching}
The Jordan domains $\mathfrak{B}_\infty^\mathrm{imm}$, $\mathfrak{U}_1$, $\mathfrak{U}_2$, and $\mathfrak{U}_3$ pairwise touch precisely at the six singular points on $\partial \mathfrak{T}$.
\end{proposition}

Finally, the construction of the map $F$ from the anti-rational $g$ gives rise to the following description of $F$ as a mating of $g$ and the Nielsen map $N$ of the classical Apollonian gasket reflection group.

\begin{theorem}\label{them:schwarz_mating}
The Riemann sphere admits a decomposition into three $F$-invariant subsets $$\widehat{\C}=\mathfrak{B}_\infty\sqcup\mathfrak{L}\sqcup\mathfrak{T}^\infty$$ such that $\partial\mathfrak{B}_\infty=\mathfrak{L}=\partial\mathfrak{T}^\infty.$ Moreover, $F$ is a conformal mating of $g$ and $N$ in the following sense: 
$$
F:\overline{\mathfrak{B}_\infty}\longrightarrow\overline{\mathfrak{B}_\infty}
$$ 
is topologically conjugate to 
$$
g:\overline{\mathcal{B}_0(g)}\longrightarrow\overline{\mathcal{B}_0(g)},
$$ 
and 
$$
\displaystyle F:\overline{\mathfrak{T}^\infty}\setminus\bigcup_{i=1}^3\Int{\mathfrak{T}_i^0}\longrightarrow\overline{\mathfrak{T}^\infty}
$$ 
is topologically conjugate to 
$$
N \colon \overline{\Omega_{\mathrm{part}}}\setminus\bigcup_{i=1}^3\Int{T_i^0}\longrightarrow\overline{\Omega_{\mathrm{part}}}
$$ 
such that both conjugacies are conformal on the interior of their respective domains of definition. 
\end{theorem}

\begin{corollary}\label{gaskets_homeo}
$\mathfrak{L}$ is homeomorphic to the classical Apollonian gasket $\Lambda_H$.
\end{corollary}

\begin{remark}[Broken symmetry] 
Despite the fact that $\mathfrak{L}$ is homeomorphic to the classical Apollonian gasket $\Lambda_H$ and the Julia set $\mathcal{J}(g)$, each of which has equal homeomorphism and quasisymmetry groups, the group of quasisymmetries of $\mathfrak{L}$ is a strict subgroup of its homeomorphism group. Indeed, there is a homeomorphism of $\mathfrak{L}$ (induced by a tetrahedral symmetry) which carries $\partial\mathfrak{U}_1$ onto $\partial\mathfrak{B}_\infty^{\textrm{imm}}$, and sends the fixed points (of $F$) on $\partial\mathfrak{U}_1$ to those on $\partial\mathfrak{B}_\infty^{\textrm{imm}}$ (see Figure~\ref{fig:apollo}). Since one of the fixed points on $\partial\mathfrak{U}_1$ is an inward pointing cusp, and all three fixed points on $\partial\mathfrak{B}_\infty^{\textrm{imm}}$ are outward pointing cusps, it follows that this homeomorphism cannot be a quasisymmetry. This observation implies that while $\textrm{Aut}^{\mathcal{T}}(\widehat{\C})$ is isomorphic to the symmetric group $S_4$, only six of the corresponding homeomorphisms of $\mathfrak{L}$ are quasisymmetric; namely the ones generated by $2\pi/3$-rotation and complex conjugation. In fact, we believe that $\textrm{QS}(\mathfrak{L})$ is isomorphic to $S_3\ltimes H$. As a result, the quasisymmetry groups of $\mathfrak{L}$ and $\Lambda_H$ allow one to distinguish the two fractals quasiconformally.
\end{remark}

\end{document}